\documentclass[a4paper,10pt
]{article}

\usepackage{geometry}
\usepackage{amsmath}
\usepackage{amssymb}
\usepackage{amsthm}
\usepackage[latin1]{inputenc}
\usepackage{eurosym}
\usepackage[dvips]{graphics}
\usepackage{graphicx}
\usepackage{epsfig}

\usepackage{comment}

\usepackage[displaymath,
mathlines]{lineno}

\allowdisplaybreaks

\usepackage{hyperref}

\usepackage{ifthen}

\newcommand{\supp}{\operatorname{supp}}

\numberwithin{equation}{section}

\newtheoremstyle{thmlemcorr}{10pt}{10pt}{\itshape}{}{\bfseries}{.}{10pt}{{\thmname{#1}\thmnumber{ #2}\thmnote{ (#3)}}}
\newtheoremstyle{thmlemcorr*}{10pt}{10pt}{\itshape}{}{\bfseries}{.}\newline{{\thmname{#1}\thmnumber{ #2}\thmnote{ (#3)}}}
\newtheoremstyle{defi}{10pt}{10pt}{\itshape}{}{\bfseries}{.}{10pt}{{\thmname{#1}\thmnumber{ #2}\thmnote{ (#3)}}}
\newtheoremstyle{remexample}{10pt}{10pt}{}{}{\bfseries}{.}{10pt}{{\thmname{#1}\thmnumber{ #2}\thmnote{ (#3)}}}
\newtheoremstyle{ass}{10pt}{10pt}{}{}{\bfseries}{.}{10pt}{{\thmname{#1}\thmnumber{ A#2}\thmnote{ (#3)}}}

\theoremstyle{thmlemcorr}
\newtheorem{theorem}{Theorem}
\numberwithin{theorem}{section}
\newtheorem{lemma}[theorem]{Lemma}

\newtheorem{proposition}[theorem]{Proposition}

\theoremstyle{thmlemcorr*}
\newtheorem{theorem*}{Theorem}
\newtheorem{lemma*}[theorem]{Lemma}
\newtheorem{corollary*}[theorem]{Corollary}
\newtheorem{proposition*}[theorem]{Proposition}
\newtheorem{problem*}[theorem]{Problem}
\newtheorem{conjecture*}[theorem]{Conjecture}

\theoremstyle{defi}
\newtheorem{definition}[theorem]{Definition}

\theoremstyle{remexample}
\newtheorem{remark}[theorem]{Remark}

\theoremstyle{ass}

\def\Xint#1{\mathchoice 
{\XXint\displaystyle\textstyle{#1}}%
{\XXint\textstyle\scriptstyle{#1}}%
{\XXint\scriptstyle\scriptscriptstyle{#1}}%
{\XXint\scriptscriptstyle\scriptscriptstyle{#1}}%
\!\int} 
\def\XXint#1#2#3{{\setbox0=\hbox{$#1{#2#3}{\int}$} 
\vcenter{\hbox{$#2#3$}}\kern-.5\wd0}} 
\def\dashint{\,\Xint-}

\usepackage{color}

\definecolor{mygreen}{cmyk}{0.91, 0, 0.88, 0.12}

\newcommand{\lii}{\leqslant}
\newcommand{\gii}{\geqslant}
\newcommand{\epsi}{\varepsilon}
\newcommand{\aev}{{a.e.}}
\newcommand{\grad}{\nabla}
\newcommand{\intO}{\int_{\Omega}}
\newcommand{\ffi}{\varphi}

\def\weaklystar{\buildrel{\hskip-.6mm\star}\over\weakly}

\def\dist{{\rm dist}}
\def\d{{\rm d}}
\def\dx{{\rm d}x}
\def\dy{{\rm d}y}

\def\dt{{\rm d}t}
\let\weakly\rightharpoonup

\def\calH{{\cal H}}
\def\calQ{{\cal Q}}

\def\calB{{\cal B}}
\def\calC{{\cal C}}

\def\calF{{\cal F}}
\def\calQF{{\cal QF}}
\def\calA{{\cal A}}
\def\calL{{\cal L}}

\def\calP{{\cal P}}
\def\calM{{\cal M}}

\def\sminus{\backslash}

\newcommand{\II}{\mathbb{I}}
\newcommand{\RR}{\mathbb{R}}
\newcommand{\NN}{\mathbb{N}}
\newcommand{\ZZ}{\mathbb{Z}}

\begin{document}


\title{A chromaticity-brightness model for color images denoising\\ in a Meyer's ``u + v'' framework}
\author{
Rita Ferreira,
\!\!\footnote{King Abdullah University of Science and Technology (KAUST), CEMSE Division \&
KAUST SRI,
Center for Uncertainty Quantification  in Computational
Science and Engineering, Thuwal 23955-6900, Saudi Arabia.\hfill\break
E-mail address: rita.ferreira@kaust.edu.sa }
\  Irene Fonseca,\!\!
\footnote{Department of Mathematical Sciences, Carnegie Mellon University, Pittsburgh, PA 15213, USA.  \hfill\break
E-mail address: fonseca@andrew.cmu.edu}
\ and M. Lu\'{\i}sa Mascarenhas\!
\footnote{Departamento de Matem\'atica \& C.M.A. - F.C.T./U.N.L., Quinta da Torre, 2829--516 Caparica, Portugal. \hfill\break
E-mail address: mascar@fct.unl.pt}
}

\maketitle
\thispagestyle{empty}

\begin{abstract}
 A variational model for imaging segmentation
and denoising  color images  is  proposed.
 The  model combines Meyer's ``u+v" decomposition
 with a  chromaticity-brightness framework
 and is expressed by a minimization  of energy
integral functionals depending on a small parameter
$\epsi>0$. The asymptotic behavior as $\epsi\to0^+$
is characterized,
and convergence of infima, almost minimizers,
and energies are established.  In particular,
an integral representation of the lower semicontinuous
envelope, with respect to the $L^1$-norm, of
functionals with linear growth and defined
for maps taking values on a certain compact
manifold is provided. This study escapes the
realm of previous results since the underlying
manifold has boundary, and the  integrand and
its recession function fail to satisfy   hypotheses
commonly  assumed  in the literature. The main
tools
are $\Gamma$-convergence and relaxation techniques. \\


\end{abstract}




\parindent=0pt
\parskip=2mm

\tableofcontents 

\section{Introduction and Main Results}\label{intro}
An important problem in image processing is
the restoration, or denoising,  of a given
``noisy" image. 
Deterioration of images may be caused by several factors, some of which  occur  in the process of acquisition (e.g., blur may derive from
an incorrect lens adjustment
or due to motion)  or transmission. 
Variational
PDE methods have proven
to be  successful in
the restoration  process, where the desired clean and sharp image is obtained as a minimizer of a certain energy functional. The energy functionals proposed in the literature share the common feature of taking into account a balance between a certain distance to the given noisy image, the so-called fidelity term, and a filter acting as a regularization of the image.

In the seminal work by Tikhonov and Arsenin \cite{TALXXVII}, the fidelity term is expressed in terms of the $L^2$-distance to the noisy image, while the regularization term is given by the $L^2$-norm of the gradient. This model suffers from an important drawback, as an over smoothing is observed, and edges in images are not preserved. It turns out that the $L^2$-norm for the gradient allows the removal of noise, but penalizes too much the gradient near and on the edges of  an image. The
same observation applies
to any $L^p$-norm, $p>1$,
and this suggests using instead the $L^1$-norm, as first noticed by Rudin, Osher, and Fatemi \cite{ROFCXII}. Precisely, representing by $\Omega\subset\RR^2$ the image domain and by $u_0:\Omega\to\RR$ the observed noisy version
of the true unknown image
$u$, Rudin-Osher-Fatemi's model 
(the ROF model) aims at finding
\begin{eqnarray*}
\begin{aligned}
\inf_{u\in W^{1,1}(\Omega)\atop u_0-u\in L^2(\Omega)} \bigg\{ \int_{\Omega} |\grad u|\,\dx + \lambda\int_\Omega |u_0 - u|^2 \,\dx\bigg\}
\end{aligned}
\end{eqnarray*}
or, equivalently,
\begin{eqnarray*}
\begin{aligned}
\min_{u \in  BV(\Omega)\atop u_0-u\in L^2(\Omega)} \bigg\{ |D u|(\Omega) + \lambda\int_\Omega |u_0 - u|^2 \,\dx\bigg\},
\end{aligned}
\end{eqnarray*}
where $\lambda>0$ is a tuning parameter and
$BV(\Omega)$ is the space of functions of
bounded variation in
$\Omega$. The ROF model, also known as the total variation model (TV model) since the filter used is the total variation of the image,  searches functions $u$ that best fit the data, measured in terms of the $L^2$
fidelity term, and whose gradient (total variation) is low so that noise is removed. It yields a decomposition of the type
\begin{eqnarray}\label{u+v}
\begin{aligned}
u_0=u+v,
\end{aligned}
\end{eqnarray}
where $u$ is well-structured, aimed at modeling homogeneous regions, while $v$ encodes
noise or textures.

The ROF model removes noise while preserving
edges, and it was extended to higher-order
 and  vectorial settings to treat color images
(see, for instance, 
\cite{AKVI, CEPYVI} for an overview).
However, it leads to  undesirable phenomena like blurring,  stair-casing effect
 (see \cite{AKVI, CEPYVI}), and it may also fail
to provide a good decomposition \eqref{u+v} of the given corrupted image as, for example, some pure geometric images (represented by characteristic functions) are treated as noise or textures (see \cite{Meyer01}). The reasons pointed out in literature relate to both the fidelity term and the regularization term. In this paper, we will focus mainly on the former.

Meyer \cite{Meyer01} showed that oscillating images are often treated as texture or noise. He proved that replacing the $L^2$-norm in the fidelity term  by a certain $G$-norm leads to better decompositions. Accordingly, he suggested the model
\begin{eqnarray}\label{MeyM}
\begin{aligned}
\inf_{u\in BV(\Omega)\atop u-u_0\in G(\Omega)}\bigg\{ |Du|(\Omega) + \lambda \Vert u-u_0\Vert_{G(\Omega)}\bigg\},
\end{aligned}
\end{eqnarray}
and we refer to Subsection~\ref{Gspa} for a detailed description and main properties of the space $G(\Omega)$ established in \cite{AuAuV,Meyer01}.
 Meyer's model has motivated several contributions aiming at overcoming some numerical difficulties posed by the structure of the $G$-norm (see, for instance, 
 \cite{AABFC05,OSV03,VeseOsher03}).
Finally, we mention that the infimum in \eqref{MeyM} is  a minimum, but the uniqueness of  minimizers is still an open problem.

When dealing with color images, the general
idea of the chromaticity-brightness approach is as follows: as before, $\Omega\subset\RR^2$ denotes the image domain, while $u_0:\Omega\to(\RR^+_0)^3$ is the observed deteriorated image, represented in the RGB (red, green, blue) system and assumed to belong to  $L^\infty(\Omega;\RR^3)$. The brightness component, $(u_0)_b$, of $u_0$ measures
the intensity of $u_0$,
 is defined by
\begin{eqnarray*}
\begin{aligned}
(u_0)_b := |u_0|,
\end{aligned}
\end{eqnarray*}
and is assumed to be different from zero \aev\ in $\Omega$. The chromaticity 
component, $(u_0)_c$, of $u_0$ is given by
\begin{eqnarray*}
\begin{aligned}
(u_0)_c := {u_0\over|u_0|}= {u_0\over (u_0)_b},
\end{aligned}
\end{eqnarray*}
which is well defined \aev\ in $\Omega$ and takes values in $S^2$, the unit sphere in $\RR^3$. It stores the color information of $u_0$. The function $u_0$ and its components are related by the identity $u_0=(u_0)_b (u_0)_c$. The core of the chromaticity-brightness models is to restore these two components independently. Representing by $u_b$ and $u_c$ the restored brightness and chromaticity components, respectively, the restored imaged is given by $u:=u_b u_c$. 

Because $(u_0)_b$ behaves as a gray-scaled image, to restore this component we may use, for instance,  one of the models previously mentioned. To restore the chromaticity component $(u_0)_c$, we  adopt Kang and March's model \cite{KMVII} using weighted harmonic maps. To be
precise, we  consider the problem
\begin{eqnarray}
\label{KMmodel}
\begin{aligned}
\min_{u_c\in W^{1,2}(\Omega;S^2)}\bigg\{\int_\Omega g(|\grad u_b^\sigma|)|\grad u_c|^2\,\dx + 
\lambda\int_{\Omega} |u_c- (u_0)_c|^2\,\dx\bigg\},
\end{aligned}
\end{eqnarray}
where $\lambda$ is a tuning parameter, $u_0$
is
extended by zero outside
$\Omega$,
\begin{eqnarray*}
\begin{aligned}
u_b^\sigma:=G_\sigma* (u_0)_b,\quad G_\sigma(x):= {A\over\sigma} e^{-{|x|^2\over4\sigma}}, \enspace A>0, \,\sigma>0,
\end{aligned}
\end{eqnarray*}
is a smooth regularization of $(u_0)_b$, and $g:\RR_0^+\to\RR^+$ is a
non-increasing positive function satisfying $g(0)=1$ and
$\lim_{t\to+\infty}g(t)=0$  (see also \cite{DMFLM09},
and the references therein, for this choice of functions
\(g\)). Examples of such functions $g$ are
\begin{eqnarray*}
\begin{aligned}
g(t):={1\over 1 + \big({t\over a}\big)^2}, \quad g(t):= e^{-({t\over a})^2},\enspace a>0.
\end{aligned}
\end{eqnarray*}
The model proposed by Kang and March in \cite{KMVII} is aimed at image colorization. It is assumed that the brightness data is known everywhere in $\Omega$, while the color data is only available in a subset $D$ of $\Omega$. Thus, in \cite{KMVII}, the second integral (fidelity term) in \eqref{KMmodel} is taken over $D$ (here we assume $D=\Omega$),
and it forces the function $u_c$ to be close  to the chromaticity data in $D$. The first integral acts as a regularization functional
and allows for sharp transitions of $u_c$ across the edges of $(u_0)_b$ since the value of $g(|\grad u_b^\sigma|)$ is close to zero in the regions where $u_b^\sigma$ varies fast. To deal with the nonconvex $S^2$ constraint,  in \cite{KMVII} the authors introduce a penalized version of the variational problem above, and convergence to the original variational problem as the penalty parameter tends to infinity is established. Numerical simulations are also performed. 

A natural question that is not considered in \cite{KMVII} is the asymptotic behavior of the variational model \eqref{KMmodel} as $\sigma$ tends to zero. Since $u_b^\sigma\in C^\infty(\overline\Omega)$ for every $\sigma>0$, it represents a  smooth version of the brightness component and, therefore, some relevant information may not be encoded in the model \eqref{KMmodel}. Furthermore, it avoids a compactness issue since, for fixed $\sigma>0$, $\inf_{\Omega} g(|\grad u_b^\sigma|)>0$ and, therefore, every minimizing sequence for \eqref{KMmodel} is relatively compact 
in $W^{1,2}(\Omega;\RR^3)$.

In this paper, we  deal with the denoising
problem
for color images by
considering  a model that combines the strengths
of Meyer's decomposition, adapted to color
images, with the strengths of chromaticity-brightness
models, which are preferred in literature as
they are considered as reducing shadowing and
providing better simulation results
(see, for
example, \cite{ChanKangShen01,KMVII,TangSapiroCaselles01}).
To be precise, we  adopt 
Kang and March's brightness-chromaticity approach,
but we  avoid the smoothing step. A vectorial
version of Meyer's model is 
\begin{eqnarray*}
\begin{aligned}
F_0(u):= |Du|(\Omega) + \lambda_0 \Vert u-u_0
\Vert_{G(\Omega;\RR^3)},\quad u\in BV(\Omega;\RR^3),\enspace
u-u_0\in G(\Omega;\RR^3), \enspace\lambda_0\in\RR^+.
\end{aligned}
\end{eqnarray*}
We   treat  the brightness component and the chromaticity component of $u_0$ separately. 
For the former, we use Meyer's model (for gray-scaled images),
which leads to the functional
\begin{eqnarray*}
\begin{aligned}
F_1(u_b):= |Du_b|(\Omega) + \lambda_b \Vert
u_b-(u_0)_b\Vert_{G(\Omega)},\quad u_b\in BV(\Omega),
\enspace u_b-(u_0)_b\in G(\Omega), \enspace
\lambda_b\in\RR^+.
\end{aligned} 
\end{eqnarray*}
For the latter, we use Kang and March's model
replacing $u_b^\sigma$ by $u_b$ assuming for
the moment that $u_b\in W^{1,1}(\Omega)$, and
thus we introduce the functional
\begin{eqnarray*}
\begin{aligned}
F_2(u_c):= \int_\Omega g(|\grad u_b|)|\grad u_c|^2\,\dx + \lambda_c\int_\Omega |u_c- (u_0)_c|^2\,\dx, \quad u_c\in W^{1,2}(\Omega;S^2),  \enspace \lambda_c\in\RR^+.
\end{aligned}
\end{eqnarray*}
To couple these two approaches, we set $u:= u_b u_c$ and consider the minimization problem
\begin{eqnarray*}
\begin{aligned}
&  \inf_{u_b\in W^{1,1}(\Omega),  u_c\in W^{1,2} (\Omega;S^2),\atop u_b-(u_0)_b\in G(\Omega), u_0-u_cu_b\in G(\Omega;\RR^3)} \Big\{F_0(u_bu_c) + F_1(u_b) + F_2(u_c)\Big\}; 
\end{aligned}%
\end{eqnarray*}
that is,
{\setlength\arraycolsep{0.5pt}
\allowdisplaybreaks
\begin{eqnarray}
&& \inf_{u_b\in W^{1,1}(\Omega),
 u_c\in W^{1,2}(\Omega;S^2),\atop
u_b-(u_0)_b\in G(\Omega),
u_0-u_cu_b\in G(\Omega;\RR^3)}
\bigg\{ \int_{\Omega}
|\grad(u_c u_b)|\,\dx
+  \int_\Omega |\grad
u_b|\,\dx +  \int_\Omega
g(|\grad u_b|)|\grad
u_c|^2\,\dx \nonumber \\
&&\hskip35mm + \, \lambda_v\Vert u_b u_c -
u_0 \Vert_{G(\Omega;\RR^3)}
+ \lambda_b \Vert u_b-(u_0)_b\Vert_{G(\Omega)}+
\lambda_c\int_\Omega
|u_c- (u_0)_c|^2\,\dx\bigg\}.\label{ourM1}
\end{eqnarray}}%
As stated, problem \eqref{ourM1} presents a lack of uniform estimates regarding the gradient of $u_c$. To overcome this issue, we will add a constraint for the brightness component $(u_0)_b$ and its  test functions $u_b$, namely
\begin{eqnarray}\label{ubbdd}
\begin{aligned}
& (u_0)_b, \,  u_b\in [\alpha,\beta] \hbox{ \aev\ in $\Omega$,\,\, for some $0<\alpha\lii\beta$.}
\end{aligned}
\end{eqnarray}
Under hypothesis \eqref{ubbdd},
we have
{\setlength\arraycolsep{0.5pt}
\begin{eqnarray*}
&&\alpha\int_\Omega |\grad
u_c|\,\dx \lii \int_\Omega
|u_b\grad u_c + u_c\otimes\grad
u_b|\,\dx  + \int_\Omega
|u_c\otimes\grad u_b|\,\dx  \lii \int_{\Omega}
|\grad(u_c u_b)|\,\dx
+ \int_{\Omega} |\grad
u_b|\,\dx.
\end{eqnarray*}}%
Consequently, if 
{\setlength\arraycolsep{0.5pt}
\begin{eqnarray*}
\{(u_b^n,u_c^n)\}_{n\in\NN}\subset \big\{(u_b,u_c)\in W^{1,1}(\Omega;
[\alpha,\beta])\times W^{1,2}(\Omega;S^2)\!:\, &&\,u_b-(u_0)_b\in G(\Omega), \, u_b u_c -
u_0 \in G(\Omega;\RR^3)
\big\}
\end{eqnarray*}}%
is a infimizing sequence for \eqref{ourM1}, then, using the properties of the $G$- and  $BV$-spaces, up to a (not relabeled) subsequence, we conclude that there exist $\bar u_b\in BV(\Omega; [\alpha,\beta])$ and $\bar u_c\in BV(\Omega;S^2)$ such that
\begin{eqnarray}
&& u_b^n\weaklystar \bar u_b \enspace \hbox{weakly-$\star$ in $BV(\Omega)$,}\quad u_c^n\weaklystar \bar u_c \enspace \hbox{weakly-$\star$ in $BV(\Omega;\RR^3)$,} \quad \hbox{as $n\to\infty$}, \label{comp ourM1}\\
&& \bar u_b-(u_0)_b\in G(\Omega), \quad \bar u_b\bar u_c - u_0\in G(\Omega;\RR^3), \nonumber\\
&& \lim_{n\to+\infty}F^{fid}(u_b^n,u_c^n) = F^{fid}(\bar u_b,\bar u_c),\nonumber
\end{eqnarray}
where
\begin{eqnarray}\label{Ffid}
\begin{aligned}
F^{fid}(u_b,u_c):=\lambda_v\Vert u_b u_c -
u_0 \Vert_{G(\Omega;\RR^3)} + \lambda_b \Vert u_b-(u_0)_b\Vert_{G(\Omega)} +   \lambda_c\int_\Omega |u_c- (u_0)_c|^2\,\dx
\end{aligned}
\end{eqnarray}
is the sum of the three fidelity terms in \eqref{ourM1}. If it turned out that $\bar u_b\in W^{1,1}(\Omega; [\alpha,\beta])$ and $\bar u_c\in W^{1,1}(\Omega;S^2)$,
then   minimizers for \eqref{ourM1} would
exist  provided that the functional given by the first three terms in \eqref{ourM1} (the regularization terms) was sequential lower semicontinuous with respect to the convergences in \eqref{comp ourM1}. This sequential lower semicontinuity is intrinsically related to the problem of finding an integral representation for 
{\setlength\arraycolsep{0.5pt}
\begin{eqnarray}
\inf\bigg\{\liminf_{n\to\infty}
\int_\Omega h(u_b^n,u_c^n,\grad
u_b^n,\grad u_c^n)\,\dx\!:\,
&&\,u_b^n\in  W^{1,1}(\Omega;[\alpha, \beta]),
\, u_b^n\weakly u_b \hbox{
weakly in }  W^{1,1}(\Omega), \nonumber\\
&&\,\,  u_c^n\in W^{1,2}(\Omega;S^2),\,
u_c^n\weakly u_c \hbox{
weakly in } W^{1,1}(\Omega;\RR^3)\bigg\}, \label{infgap}
\end{eqnarray}}%
with
\begin{eqnarray*}
\begin{aligned}
h(r,s,\xi,\eta):=|\xi| + g(|\xi|)|\eta|^2 + |s\otimes\xi+r\eta|.
\end{aligned}
\end{eqnarray*}

In general, 
\begin{eqnarray*}
\begin{aligned}
(\xi,\eta)\mapsto h(r,s,\xi,\eta)= |\xi| + g(|\xi|)|\eta|^2 + |s\otimes\xi+r\eta| 
\end{aligned}
\end{eqnarray*}
is not quasiconvex. 
Moreover, for $(r,s)\in [\alpha,\beta] \times
S^2$, $h$ satisfies the non-standard growth conditions
\begin{eqnarray*}
\begin{aligned}
{1\over C} (|\xi| + |\eta|)\lii h(r,s,\xi,\eta)\lii C(1+|\xi| + |\eta|^2),
\end{aligned}
\end{eqnarray*}
 which leads us to a well-known, but poorly
understood, gap problem
(see \cite{FLMIV,FMXCVII,  MMV} concerning the unconstrained setting). We also observe  that the admissible sequences in \eqref{infgap}  should satisfy in addition the restrictions $u^n_b-(u_0)_b\in G(\Omega)$ and $u^n_b u_c^n -
u_0 \in G(\Omega;\RR^3)$ or, equivalently (see Proposition~\ref{Galtern}), 
{\setlength\arraycolsep{0.5pt}
\begin{eqnarray}
&&\int_\Omega (u_b^n-(u_0)_b)\,\dx=0 \enspace
\hbox{
and}\enspace  \int_\Omega
(u_b ^nu_c ^n- u_0)\,\dx=0. \label{1.9star}
\end{eqnarray}}%
It turns out to be a challenging task to construct
a recovery sequence that simultaneously satisfies
the manifold constraint and the average restrictions.

In view of these considerations, to avoid the
gap and to penalize deviations from average
zero in \eqref{1.9star}, as a first approach to  problem \eqref{ourM1}, we   study the asymptotic behavior, as $\epsi\to0^+$, of   the problems
\begin{eqnarray}\label{ourM}
\begin{aligned}
\inf_{(u_b,u_c)\in W^{1,1}(\Omega;[\alpha,\beta])\times
W^{1,1}(\Omega;S^2)} \big\{F^{reg}(u_b,u_c) + F_\epsi^{fid}(u_b,u_c)\big\},
\end{aligned}
\end{eqnarray}
where
\begin{eqnarray}\label{Freg}
\begin{aligned}
F^{reg}(u_b,u_c):=\int_\Omega |\grad u_b|\,\dx +   \int_\Omega g(|\grad u_b|)|\grad u_c|\,\dx + \int_{\Omega}|\grad(u_c u_b)|\,\dx
\end{aligned}
\end{eqnarray}
and 
{\setlength\arraycolsep{0.5pt}
\begin{eqnarray}
F_\epsi^{fid}(u_b,u_c):=&&\lambda_v \bigg\Vert u_b u_c -
u_0  - \dashint_\Omega
(u_b u_c - u_0)\,\dx \bigg \Vert_{G(\Omega;\RR^3)} + \frac{1}{\epsi} \bigg| \int_\Omega
(u_b u_c - u_0)\,\dx\bigg| \nonumber \\
&&+\, \lambda_b \bigg\Vert
u_b-(u_0)_b  - \dashint_\Omega
(u_b-(u_0)_b)\,\dx\bigg\Vert_{G(\Omega)} + \frac{1}{\epsi} \bigg| \int_\Omega
(u_b-(u_0)_b)\,\dx\bigg|  \nonumber\\
&&+\,   \lambda_c\int_\Omega
|u_c- (u_0)_c|^2\,\dx. \label{penFfid}
\end{eqnarray}}%
The integrand involved in \eqref{Freg} satisfies standard growth conditions (see \eqref{boundsf}
below), and it remains relevant in terms of applications to imaging. Note that the term
\begin{eqnarray*}
\begin{aligned}
\int_\Omega g(|\grad u_b|)|\grad u_c|^2\,\dx
\end{aligned}
\end{eqnarray*}
in \eqref{ourM1} can be viewed as a weighted version of Tikhonov and Asenin \cite{TALXXVII}'s regularization term mentioned at the beginning of this introduction, while in \eqref{ourM} we use instead a weighted version of ROF's regularization term (se also \cite{CMM00,DMFLM09}), namely
\begin{eqnarray*}
\begin{aligned}
\int_\Omega g(|\grad u_b|)|\grad u_c|\,\dx.
\end{aligned}
\end{eqnarray*}
For small $\epsi>0$, the functional $F_\epsi^{ fid}$ is a penalized version of  the functional $F^{
fid}$ in \eqref{Ffid} that, by means of the factor $\frac{1}{\epsi}$, penalizes sequences $\{(u_b^n, u_c^n)\}_{n\in\NN}$ whose averages $\int_\Omega
(u^n_b-(u_0)_b)\,\dx$ and  $\int_\Omega
(u^n_b u^n_c - u_0)\,\dx$ are far from zero. This penalization allows us to incorporate the $G$-norm and the $G$-restrictions in our model. As we will see, in the limit as $\epsi\to0^+$, we will recover the functional $F^{fid}$ and limit pairs $(u_b,u_c)$ will satisfy $u_b-(u_0)_b\in G(\Omega)$ and $u_b u_c - u_0 \in G(\Omega;\RR^3)$.

Before we state our main theorem, we introduce some notation. Regarding functions of bounded variation, we  adopt the notations in \cite{AFPMM}, and we refer to Subsection~\ref{onBV} for more details. 
Let $f:\RR\times \RR^3\times \RR^{2}\times \RR^{3\times 2}\to[0,+\infty)$ be defined, for $(r,s,\xi,\eta)\in \RR\times \RR^3\times \RR^{2}\times \RR^{3\times 2}$, by
\begin{eqnarray}\label{defoff}
\begin{aligned}
f(r,s,\xi, \eta):=|\xi| + g(|\xi|)|\eta| + |r\eta + s\otimes\xi|,
\end{aligned}%
\end{eqnarray}
where, as above, $g:[0,+\infty)\to (0,1]$ is a non-increasing, Lipschitz continuous function satisfying $g(0)=1$ and $\lim_{t\to+\infty} g(t) = 0$. Notice that
\begin{eqnarray*}
\begin{aligned}
F^{reg}(u_b,u_c) = \intO f(u_b(x), u_c(x),\grad
u_b(x),\grad u_c(x))\,\dx.
\end{aligned}%
\end{eqnarray*}
For $s\in \overline{B(0,1)}$, we have
\(\frac{1}{2}|\xi| +
\frac{|r|}{2}|\eta| \lii \frac{1}{2}|\xi| +
\frac{1}{2}(|r\eta + s \otimes \xi) + |\xi|) \lii
f(r,s,\xi, \eta)\),
and so
\begin{eqnarray}\label{boundsf}
\begin{aligned}
& \frac{1}{2}|\xi| +
\frac{|r|}{2}|\eta| \lii
f(r,s,\xi, \eta) \lii
2|\xi| + (1+ |r|)|\eta|,
\end{aligned}
\end{eqnarray}
where we used the fact that \(g(\cdot) \lii
1\).
 
 For $r\in\RR$, $s\in S^2$, $\xi\in\RR^{2}$, and $\eta\in [T_s(S^2)]^2$, where $T_s(S^2)$ is the tangential space to $S^2$ at $s$, we denote by $\calQ_Tf$ the tangential quasiconvex envelope of  $f$; to
be precise (see \cite{DFMTXCIX}),
{\setlength\arraycolsep{0.5pt}
\begin{eqnarray}\label{defQTf}
\calQ_T f(r,s,\xi,\eta)
:=\inf\bigg\{\int_Q
f(r,s,\xi + \grad\ffi(y),
\eta+\grad\psi(y))\,\dy\!:\,
&&\,\,\ffi\in
W^{1,\infty}_0(Q), \,\,\psi\in
 W^{1,\infty}_0(Q;T_s(S^2))\bigg\}.\nonumber
 \\
\end{eqnarray}}%
The recession function,
$f^\infty$, of $f$ 
is the function defined, for $(r,s,\xi,\eta)\in
\RR\times
\RR^3\times \RR^{2}\times
\RR^{3\times 2}$,
by 
{\setlength\arraycolsep{0.5pt}
\begin{eqnarray}\label{frec}
f^\infty(r,s,\xi,\eta):=&&
\limsup_{t\to+\infty}
{f(r,s,t\xi,t\eta)\over
t}= \limsup_{t\to+\infty}
\big(|\xi| + g(t|\xi|)|\eta|
+ |r\eta + s\otimes\xi|\big) \nonumber\\
=&&|\xi|+ \chi_{_{\{0\}}}(|\xi|)|\eta|+|r\eta + s\otimes\xi|
,
\end{eqnarray}}%
where $\chi(t) :=1$ if $t=0$ and $\chi(t):=0$
if $t\in\RR\backslash \{0\}$,  because $g(0)=1$ and $\displaystyle\lim_{t\to+\infty}
g(t) = 0$. Note that
{\setlength\arraycolsep{0.5pt}
\begin{eqnarray}
&& f^\infty(r,s,\xi,\eta) \lii (3+\beta) |(\xi,\eta)| \label{upborecfct} 
\end{eqnarray}}%
for $r\in[\alpha,\beta]$,
$s\in S^2$, $\xi\in\RR^2$,
and $\eta\in [T_s(S^2)]^2$.
The recession function,
$(Q_Tf)^\infty$,
of $Q_Tf$ is the function defined, for $r\in\RR$,
$s\in S^2$, $\xi\in\RR^2$,
and $\eta\in [T_s(S^2)]^2$, by
\begin{eqnarray*}
\begin{aligned}
(Q_Tf)^\infty(r,s,\xi,\eta):=&
\limsup_{t\to+\infty}
{Q_Tf(r,s,t\xi,t\eta)\over
t}\cdot
\end{aligned}
\end{eqnarray*}
For $a,b\in [\alpha,\beta]\times S^2$ and $\nu\in S^1$, we set
{\setlength\arraycolsep{0.5pt}
\begin{eqnarray}\label{defK}
K(a,b,\nu):= && \inf\bigg\{\int_{Q_\nu}
f^\infty( \ffi(y),\psi(y),
\grad\ffi(y), \grad\psi(y))
\,\dy\!:\, \vartheta=(\ffi,\psi)
\in\calP(a,b,\nu)\bigg\}\nonumber
\\ 
=&&\inf\bigg\{\int_{Q_\nu}
\big(|\grad \ffi(y)|
+ |\grad(\ffi\psi)(y)|
+ \chi_{_{\{0\}}}(|\grad\ffi|)|
\grad\psi|\big)
\,\dy\!:\, \vartheta=(\ffi,\psi)
\in\calP(a,b,\nu)\bigg\},
\end{eqnarray}}%
where $Q_\nu$ is the unit cube in $\RR^2$ centered at the origin and with two faces orthogonal to $\nu$, and
{\setlength\arraycolsep{0.5pt}
\begin{eqnarray}\label{defP}
\calP(a,b,\nu):=\Big\{
\vartheta=(\ffi,\psi)\in
W^{1,1}(Q_\nu; [\alpha,\beta]\times
S^2)\!:\,\,&&\hbox{$\vartheta$ is $1$-periodic
in the  orthogonal direction to $\nu,$} \nonumber\\
 &&\vartheta(y)
= a \hbox{ if } y\cdot\nu=-
\frac{1}{2}, \,\vartheta(y)
= b \hbox{ if } y\cdot\nu=
\frac{1}{2}\Big\}.
\end{eqnarray}}%
Finally, we define the functional $F^{reg, sc^-}:
BV(\Omega;[\alpha,\beta])\times BV(\Omega;S^2)\to\RR$
 as
{\setlength\arraycolsep{0.5pt}
\begin{eqnarray}\label{lscFreg}
F^{reg, sc^-}(u_b,u_c):=
&&\displaystyle \intO
\calQ_Tf(u_b(x), u_c(x),\grad
u_b(x),\grad u_c(x))\,\dx
\nonumber\\
&&\quad + \int_{S_{(u_b,u_c)}}
K\big((u_b,u_c)^+(x),(u_b,u_c)^-(x),\nu_{(u_b,u_c)}(x)\big) \,\d\calH^1(x)
\nonumber\\
&&\quad +\int_{\Omega}
(\calQ_Tf)^\infty(\tilde
u_b(x), \tilde u_c(x),
W^c_b(x), W^c_c(x))\,
\d|D^c(u_b,u_c)|(x),
\end{eqnarray}}%
where \(\tilde u_b(x)\) and \(\tilde u_c(x)\)
are
the approximate limits of \(u_b\) and \(u_c\)
at \(x\),
respectively, and where \(W^c\) is the Radon-Nikodym
derivative of $D^c(u_b,u_c)$ with respect to
its total variation, 
 $W^c_b$ is the first row of ${W^c}$,  and
$W^c_c$ is the $3\times2$ matrix obtained from
${W^c}$ by erasing its first row.
Our main result is the following.
\begin{theorem}\label{mainthm}
Let $\Omega\subset\RR^2$
be an open, bounded domain with Lipschitz
boundary $\partial\Omega$, and
let $\{\epsi_n\}_{n\in\NN}$ and $\{\delta_n\}_{n\in\NN}$
be two arbitrary sequences of positive
numbers
converging to zero. Let $F^{reg}$, $F^{fid}_{\epsi_n}$,
$F^{reg, sc^-}$, and $F^{fid}$ be the
functionals introduced in   \eqref{Freg},
\eqref{penFfid}, \eqref{lscFreg}, and
\eqref{Ffid}, respectively, and let  $X$
be the set 
\begin{eqnarray}\label{defsetX}
\begin{aligned}
 X:=\big\{(u_b,u_c)\in BV(\Omega;[\alpha,\beta])
\times BV(\Omega;S^2)\!:\, &\, u_b-(u_0)_b\in
G(\Omega), \, u_b u_c -
u_0 \in G(\Omega;\RR^3)\big\}.
\end{aligned}
\end{eqnarray}
Then,
{\setlength\arraycolsep{0.5pt}%
\begin{eqnarray*}
&&\min_{(u_b,u_c)\in  X} \big (F^{reg,
sc^-}(u_b,u_c) + F^{fid}(u_b,u_c)\big)
= \lim_{n\to\infty} \inf_{(u_b,u_c)\in
W^{1,1}(\Omega;[\alpha,\beta])\times
W^{1,1}(\Omega;S^2)} \big(F^{reg}(u_b,u_c)+
F^{fid}_{\epsi_n}(u_b,u_c)\big).
\end{eqnarray*}}%
Moreover, if for each $n\in\NN$, 
$( u_b^n,
u_c^n)\in W^{1,1}(\Omega;[\alpha,\beta])\times
W^{1,1}(\Omega;S^2)$ is a $\delta_n$-minimizer
of the functional $\big(F^{reg}+F^{fid}_{\epsi_n}\big)$
 in $W^{1,1}(\Omega;[\alpha,\beta])\times
W^{1,1}(\Omega;S^2)$, that is,
\begin{eqnarray*}
\begin{aligned}
F^{reg}( u_b^n, u_c^n) + F^{fid}_{\epsi_n}(
u_b^n, u_c^n)\lii \inf_{(u_b,u_c)\in W^{1,1}(\Omega;[\alpha,\beta])\times
W^{1,1}(\Omega;S^2)} \big(F^{reg}(u_b,u_c)+
F^{fid}_{\epsi_n}(u_b,u_c)\big) + \delta_n,
\end{aligned}%
\end{eqnarray*}
 then $\{( u_b^n, u_c^n)\}_{n\in\NN}$
is sequentially, relatively compact with
respect
to the weak-$\star$ convergence in $BV(\Omega)\times
BV(\Omega;\RR^3)$; and if $( u_b, u_c)$
is a cluster point of $\{( u_b^n,
u_c^n)\}_{n\in\NN}$, then $( u_b, u_c)\in
X$
is a minimizer of $(F^{reg,sc^-} + F^{fid})$
in $X$
 and
\begin{eqnarray*}
\begin{aligned}
F^{reg, sc^-}( u_b, u_c) + F^{fid}(
u_b, u_c) = \lim_{n\to\infty} \Big(
F^{reg}( u_b^n, u_c^n) + F^{fid}_{\epsi_n}(
u_b^n, u_c^n)\Big).
\end{aligned}%
\end{eqnarray*}
\end{theorem}

The proof of Theorem~\ref{mainthm} relies on the following relaxation result.

\begin{theorem}\label{relaxthm}
Let $\Omega\subset\RR^2$
be an open, bounded domain
with Lipschitz boundary, let  $F^{reg}$  be given by  
 \eqref{Freg},  and consider the functional $ F:L^1(\Omega)\times L^1(\Omega;\RR^3)\to[0,+\infty]$  defined by
\begin{eqnarray*}
F(u_b,u_c):= 
\begin{cases}
\displaystyle F^{reg}(u_b,u_c) & \hbox{if $(u_b,u_c)\in W^{1,1}(\Omega;[\alpha,\beta])\times W^{1,1} (\Omega;S^2)$},\\
+\infty & \hbox{otherwise},
\end{cases}
\end{eqnarray*}
for $(u_b,u_c)\in L^1(\Omega)\times L^1(\Omega;\RR^3)$. Then the lower semicontinuous envelope of $F$, ${\cal F}:L^1(\Omega)\times L^1(\Omega;\RR^3)\to[0,+\infty],$ defined  by
{\setlength\arraycolsep{0.5pt}
\begin{eqnarray*}
\calF(u_b,u_c):=\inf\Big\{\liminf_{n\to+\infty}  F(u_b^n,u_c^n)\!: \ && n\in\NN,\, (u_b^n,u_c^n)\in L^1(\Omega)\times L^1(\Omega;\RR^3),\\
&&u_b^n\to u_b \hbox{ in } L^1(\Omega),\, u_c^n\to u_c \hbox{ in } L^1(\Omega;\RR^3)\Big\},%
\end{eqnarray*}}%
has the  integral representation
\begin{eqnarray}\label{lscF1}
\calF(u_b,u_c)= 
\begin{cases}
F^{reg, sc^-}(u_b,u_c) & \hbox{if $(u_b,u_c)\in BV(\Omega;[\alpha,\beta])\times BV(\Omega;S^2),$ }\\
+\infty & \hbox{otherwise},
\end{cases}
\end{eqnarray}
for $(u_b,u_c)\in L^1(\Omega)\times L^1(\Omega;\RR^3)$, where $F^{reg, sc^-}$ is given by \eqref{lscFreg}.
\end{theorem}

The relaxation result above falls within the general context of studying lower semicontinuity and/or finding integral representations for the lower semicontinuous envelope of functionals of the type
\begin{eqnarray*}
\begin{aligned}
\int_\Omega f(x, u(x),\grad u(x))\,\dx, \quad u\in W^{1,p}(\Omega;\calM),
\end{aligned}
\end{eqnarray*}
where $\Omega\subset\RR^N$ is an open, bounded set, $p\in[1,+\infty)$, and $\calM\subset\RR^d$ is a (sufficiently) smooth, $m$-dimensional manifold. There is a vast literature 
in this framework (see, for instance, \cite{ ACELVII,
BMIX, BMX, BrezisCoronLieb86, CFLXI,DFMTXCIX,
 Ericksen90, Mucci09, Virga94}),  motivated, for
example, by the study of equilibria for liquid
crystals and magnetostrictive materials, where
the class of admissible fields is constrained
to take values on a certain manifold $\calM$ (commonly,
$\calM= S^{d-1}$, the unit sphere in $\RR^{d}$).
As in \cite{ACELVII, BMIX,Mucci09}, the key ingredients
in the proof of Theorem~\ref{relaxthm} are the
density of smooth functions in $W^{1,1}(\Omega;\calM)$
\cite{Bethuel91,BZLXXXVIII,HaLi03} and a projection
technique introduced in \cite{ACELVII,HKL86,HaLi87}.
However, new arguments are required as three main
features of our problem prevent us from using
immediately the  relaxation results concerning
the constraint case in the $BV$ setting \cite{ACELVII,BMIX,Mucci09}:
 unlike \cite{ACELVII,Mucci09}, our starting point
cannot be a tangential quasiconvex function as
the energy density considered here (see \eqref{defoff})
 fail always to satisfy such condition (see Remark~\ref{fnotTqcx});
and
unlike the general setting in the literature,
(i) our manifold, $\calM=[\alpha,\beta]\times
S^2$, {\sl has} boundary, (ii) the recession function
$f^\infty$ in our case (see \eqref{frec}) does
not satisfy a hypothesis of the type $|f(r,s,\xi,\eta)
- f^\infty(r,s,\xi,\eta)| \lii C(1+ |(\xi,\eta)|^{1-m})$
for some $C>0$ and $m\in (0,1)$ (for \aev\ $(r,s)$
and for all $(\xi,\eta)$). We anticipate that
our arguments may be used to treat more general
manifolds with boundary and  more general integrands.

This paper is organized as follows. 
In Section~\ref{notprel}, we collect the
 notation,  we recall   properties of
the space $G$ introduced by Meyer \cite{Meyer01},
and we also recall    properties of functions
of bounded variation and  sets of finite
perimeter. We also make some considerations
on quasiconvexity. In Section~\ref{proofmainthm},
we prove Theorem~\ref{mainthm}.
Finally, in Section~\ref{proofrelaxthm},
we establish Theorem~\ref{relaxthm}.

\section{Notation and Preliminaries}\label{notprel}

\subsection{Notation}\label{not}

Let $N, d\in\NN$. If $x,y\in\RR^N$, then $x\cdot y$ stands for the Euclidean inner product of $x$ and $y$, and $|x|:=\sqrt{x\cdot x}$ for the Euclidean norm of $x$. The space of $d\times N$-dimensional matrices is identified with $\RR^{dN}$, and we write $\RR^{d\times N}$. We define $S^{N-1}:=\{x\in\RR^N\!: |
x| =1\}$, $Q:=(-\frac{1}{2}, \frac{1}{2})^N$, $Q(x,\delta):= x + \delta Q$, and $B(x,\delta):=\{y\in\RR^N\!:
|y- x| <\delta\}$, where $\delta>0$. Given $\nu\in
S^{N-1}$ and a rotation  $R_\nu$ such that $R_\nu e_N=\nu$, we set $Q_\nu:= R_\nu Q$, $Q_\nu (x,\delta):= x + \delta Q_\nu$, $B^+_\nu(x,\delta):=\{y\in
B(x,\delta)\!: (y-x)\cdot\nu>0\},$ and $B^-_\nu(x,\delta):=\{y\in
B(x,\delta)\!: (y-x)\cdot\nu<0\}$.

Let $\Omega\subset \RR^N$ be an open set. We represent by $\calA(\Omega)$ the family of all open subsets of $\Omega$, and by $\calA_\infty(\Omega)$ the family of all sets in $\calA(\Omega)$ with Lipschitz boundary. 
The Borel $\sigma$-algebra on $\Omega$ is denoted by ${\cal B}(\Omega)$, and ${\cal M}(\Omega;\RR^d)$ is the Banach space of all bounded Radon measures $\mu:{\cal B}(\Omega)\to\RR^d$ endowed with the total variation norm $|\cdot|$. If $\mu\in\calM(\Omega;\RR^+_0)$ is a nonnegative Radon measure and
$v:\Omega\to\RR^d$ is a $\mu$-measurable function, then $\dashint_\Omega v(x) \, \d\mu(x)$ stands for $\frac{1}{\mu(\Omega)} \int_\Omega v(x) \, \d\mu(x)$.
The $N$-dimensional Lebesgue measure is denoted by ${\cal L}^N$ and the $(N-1)$-dimensional Hausdorff measure is designated
by ${\cal H}^{N-1}$. Also, ``\aev\ in $\Omega$" stands for  ``almost everywhere in $\Omega$ with respect to $\calL^N$".

Let $\calM$ be an $m$-dimensional manifold, $m\in\NN$, in $\RR^d$. 
The tangent space of $\calM$ at $z\in \calM$ is represented by $T_z(\calM)$. 
Given a Banach space $X(\Omega;\RR^d)$ of functions  $\vartheta:\Omega  \to \RR^d$, 
we denote by $X(\Omega;\calM)$ the set 
{\setlength\arraycolsep{0.5pt}%
\begin{eqnarray*}
\begin{aligned}
X(\Omega;\calM):= \big\{ \vartheta\in X(\Omega;\RR^d)\!:\, \vartheta(\cdot) \in \calM \hbox{  \aev\ in } \Omega \big\}.
\end{aligned}
\end{eqnarray*}}%
To simplify the notation, if $\calM_1$  is an
$m_1$-dimensional manifold in $\RR^{d_1}$, $\calM_2$
 is an $m_2$-dimensional
manifold in $\RR^{d_2}$,
   $u\in X(\Omega;\calM_1)$,  $v\in X(\Omega;\calM_2)$,
and $w:=(u,v)$, we write $w\in X(\Omega; \calM_1\times
\calM_2)$.

\subsection{Meyer's $G$-space}\label{Gspa}

In this section, we recall the definition of the space $G(\Omega)$ introduced in \cite{Meyer01} for $\Omega=\RR^2$ and generalized in \cite{AuAuV}  for bounded domains $\Omega\subset\RR^2$. The vectorial case, which allows modeling
textures in color  images, has been treated in  \cite{DuvalAujoVese10}. Below we  collect the main properties of the space $G$  used in this paper. For the proofs and  for more considerations on the space $G$, we refer to \cite{AuAuV,DuvalAujoVese10,Meyer01}.

\begin{definition}
Let
$\Omega\subset\RR^2$ be an open, bounded domain with Lipschitz boundary, and let $d\in\NN$. We
define 
\begin{eqnarray*}
\begin{aligned}
G(\Omega;\RR^d):=\big\{v\in L^2(\Omega;\RR^d)\!:\, v_i=
{\rm div}\, \xi_i,\, \xi\in L^\infty(\Omega;(\RR^2)^d),
\, \xi_i\cdot n =0 {\hbox{ on }} \partial\Omega, \, i\in \{ 1, ..., d\}\big\},
\end{aligned}
\end{eqnarray*}
where $n$ is the outward unit normal to $\partial
\Omega$. We endow $G(\Omega;\RR^d)$ with the norm
\begin{eqnarray*}
\begin{aligned}
\Vert v\Vert_{G(\Omega;\RR^d)}:=\inf\big\{\Vert
\xi\Vert_{L^\infty(\RR^2;(\RR^2)^d)}\!:\, v_i={\rm div}\,
\xi_i , \, \xi_i\cdot n =0 {\hbox{ on }} 
\partial\Omega,
\, i\in \{ 1, ..., d\} \big\}.
\end{aligned}
\end{eqnarray*}
\end{definition}

  $G(\Omega;\RR^d)$ is a Banach space, and  when
  \(N=2\), it admits the following characterization.

\begin{proposition}\label{Galtern}
Let
$\Omega\subset\RR^2$ be an open, bounded domain with Lipschitz boundary. Then, 
\begin{eqnarray*}
\begin{aligned}
G(\Omega;\RR^d) = \bigg\{v\in L^2(\Omega;\RR^d)\!:\,
\intO v(x)\,\dx=0\bigg\}.
\end{aligned}
\end{eqnarray*}
\end{proposition}

The topology induced by the $G$-norm is coarser than the one induced by the $L^2$-norm as there are sequences that converge to zero in the $G$-norm but not in the $L^2$-norm.
More generally, the following result  shows that   the $G$-norm is well adapted to capture oscillations of a function in an energy minimization method.
\begin{proposition}\label{normGto0}
Let
$\Omega\subset\RR^2$ be an open, bounded domain
with Lipschitz boundary and let \(p>2\). If
$\{ v_n\}_{n\in\NN}\subset G(\Omega;\RR^d)$ is
such
that $v_n\weakly 0$ weakly in $L^p(\Omega;\RR^d)$
as $n\to\infty$, then $\lim_{n\to\infty} \Vert v_n\Vert_{G(\Omega;\RR^d)}=0$.
\end{proposition}

\subsection{The Space $BV$ of Functions of Bounded Variation and Sets of Finite
Perimeter}\label{onBV}

We will adopt the notations of \cite{AFPMM} regarding
functions of bounded
variation and sets of finite
perimeter.

In what follows, $N,d\in\NN$,  and $\Omega\subset\RR^N$
is an open set. 
Let $\rho\in C^\infty_c(\RR^n)$ be a nonnegative
function such that
\begin{eqnarray*}
\int_{\RR^N}\rho(x)\,\dx = 1,\quad \supp\rho
= \overline{B(0,1)},\quad \rho(x)= \rho(-x)
\enspace \hbox{for all $x\in\RR^N$.}
\end{eqnarray*}
For $u\in L^1_{\rm loc}(\Omega;\RR^d)$ and
$\delta>0$, we set 
\begin{eqnarray}\label{smoothmoll}
\rho_\delta(x):=\frac{1}{\delta^N}\rho\Big(\frac{x}{\delta}\Big),\enspace
x\in\RR^N,
\end{eqnarray}
 and
\begin{eqnarray}\label{defmollfct}
u_\delta(x):=(u*\rho_\delta)(x)=\int_{\Omega}
u(y)\rho_\delta(x-y)\,\dy,\enspace  x\in \Omega_{\delta}:=\{x\in\Omega\!:\,
\dist(x,\partial\Omega)>\delta\}.
\end{eqnarray}

We observe that $\supp\rho_\delta\subset 
\overline{B(0,\delta)}$,
and we recall that
$u_\delta\in C^\infty(\Omega_\delta)$ and $\Vert
u_\delta\Vert_{L^1(\Omega_\delta;\RR^d)}\lii
\Vert u\Vert_{L^1(\Omega;\RR^d)}$.

\begin{definition}\label{appnotions}
Let $u\in L^1_{\rm loc}(\Omega;\RR^d)$.
\begin{itemize}
\item[(a)] We define the set $A_u$ as the set
of points $x\in\Omega$ for which there exists
a vector $z\in\RR^d$ such that
\begin{eqnarray}\label{applim}
\lim_{\delta\to0^+} \dashint_{B(x,\delta)}
|u(y) - z|\,\dy =0,
\end{eqnarray}
in which case we say that $u$ has an approximate
limit at $x$, and the vector $z$, uniquely
determined by \eqref{applim},  is represented
by $\tilde u(x)$. The set $S_u:=\Omega\setminus
A_u$ is called the approximate discontinuity
set. We say that $u$ is approximately continuous
at $x$ if $x\in A_u$ and $\tilde u(x) = u(x)$,
i.e., $x$ is a Lebesgue point of $u$.

\item[(b)] We define the set of approximate
jump points of $u$, represented by $J_u$, as
the set of points $x\in\Omega$ for which there
exist vectors $a,b\in\RR^d$, $a\not= b$, and
$\nu\in S^{N-1}$ such that
\begin{eqnarray}\label{appjp}
\lim_{\delta\to0^+} \dashint_{B^+_\nu(x,\delta)}
|u(y) - a|\,\dy =0, \quad \lim_{\delta\to0^+}
\dashint_{B^-_\nu(x,\delta)} |u(y) - b|\,\dy
=0.
\end{eqnarray}
A point $x\in J_u$ is called an approximate
jump point of $u$; the associated triplet $(a,b,\nu)$,
uniquely determined by \eqref{appjp} up to
a permutation of $(a,b)$ and a change of sign
of $\nu$, is denoted by $(u^+(x), u^-(x),\nu_u(x))$.

\item[(c)] We say that u is approximately differentiable
at $x\in A_u$ if there exists a $d\times N$
 matrix $L$ such that
\begin{eqnarray}\label{appdiff}
\lim_{\delta\to0^+}\dashint_{B(x,\delta)} \frac{|u(y)
- \tilde u(x) - L(y-x)|}{\delta}\, \dy =0,
\end{eqnarray}
in which case we denote the matrix $L$, uniquely
determined by \eqref{appdiff}, by $\grad u(x)$.
The set of approximate differentiability points
is denoted by ${\cal D}_u$.
\end{itemize}

\end{definition}

\begin{remark}\label{applim&ec}
The set $A_u$ does not depend on the representative
in the equivalent class of $u$, i.e., if $v=u$
$\calL^N$-\aev\ in $\Omega$ then $A_v= A_u=:A$
and $\tilde v (x) = \tilde u(x)$ for all $x\in
A$. In contrast, the property of being approximately
continuous at $x$ depends on the value of $u$
at $x$, thus on the representative in the equivalent
class of $u$.
\end{remark}

The proof of the following  result may be found
in  \cite[Prop.~{3.64}, Prop.~{3.69}, Prop.~{3.71}]{AFPMM}.

\begin{proposition}\label{appppties}
Let $u\in L^1_{\rm loc}(\Omega;\RR^d)$, let
$u_\delta\in C^\infty(\Omega_\delta)$ be given
by \eqref{defmollfct}, let $\phi:\RR^d\to\RR^m$
be a Lipschitz map, and let  $v:=\phi\circ
u$. Then
\begin{itemize}
\item [(a)]
\begin{itemize}
\item[i)] $S_u$ is a $\calL^N$-negligible Borel
set and $\tilde u: A_u\to\RR^d$ is a Borel
function, coinciding $\calL^N$-\aev\ in $A_u$
with $u$;

\item[ii)] $\lim_{\delta\to0^+} u_\delta(x)=\tilde
u(x)$ for all $x\in A_u$;

\item[iii)]  $S_v\subset S_u$ and $\tilde v(x)
= \phi(\tilde u(x))$ for all $x\in A_u$.
\end{itemize}

\item [(b)]
\begin{itemize}
\item[i)] $J_u$ is a Borel subset of $S_u$
and there exist Borel functions $(u^+, u^-,\nu_u):
J_u\to\RR^d\times \RR^d\times S^{N-1}$ such
that \eqref{appjp} holds for all $x\in J_u$;

\item[ii)] $\lim_{\delta\to0^+} u_\delta(x)=\frac{u^+(x)
+ u^-(x)}{2}$ for all $x\in J_u$;

\item[iii)] if  $x\in J_u$, then $x\in J_v$
if and only if $\phi(u^+(x))\not=\phi(u^-(x))$,
in which case  $(v^+(x), v^-(x),\nu_v(x))=(\phi(u^+(x)),
\phi(u^-(x)),\nu_u(x))$; otherwise, $x\in A_v$
and $\tilde v(x)= \phi(u^+(x))= \phi(u^-(x))$.
\end{itemize}

\item [(c)]
\begin{itemize}
\item[i)] ${\cal D}_u$ is a Borel subset and
$\grad u: {\cal D}_u\to\RR^{d\times N}$ is
a Borel map;

\item[ii)] if $x\in {\cal D}_u$ and, in addition,
 $\ffi$ has linear growth at infinity and is
differentiable at $\tilde u(x)$, then $v$ is
approximately differentiable at $x$ and $\grad
v(x) = \grad \phi(\tilde u(x))\grad u(x)$.

\end{itemize}

\end{itemize}
\end{proposition}

A function $u:\Omega\to\RR^d$ is said to be
a function of \textit{bounded variation}, and we write
$u\in BV\big(\Omega;\RR^d\big)$, if $u\in L^1\big(\Omega;\RR^d\big)$
and its distributional derivative, $Du$, belongs
to ${\cal M}\big(\Omega;\RR^{d\times N}\big)$;
that is, if there exists a measure $Du\in{\cal
M}\big(\Omega;\RR^{d\times N}\big)$ such that
for all $\phi\in C_c(\Omega)$, $j\in\{1,\cdots,d\}$
and $i\in\{1,\cdots,N\}$, one has
\begin{eqnarray*}
\intO u_j(x){\partial\phi\over\partial
x_i}(x)\,\dx = -\intO\phi(x)\,\d
D_i u_j(x),
\end{eqnarray*}
where $u=(u_1,\cdots,u_d)$ and $Du_j=(D_1 u_j,\cdots,D_N
u_j)$. 
The space  $BV\big(\Omega;\RR^d\big)$ is a
Banach space when endowed with the norm $\Vert
u\Vert_{BV(\Omega;\RR^d)}:=\Vert u\Vert_{L^1(\Omega;\RR^d)}
+ |Du|(\Omega)$.

We recall that
$\{u_j\}_{j\in\NN}\subset BV\big(\Omega;\RR^d\big)$
is said to weakly-$\star$ converge in $BV\big(\Omega;\RR^d\big)$
to some $u\in BV\big(\Omega;\RR^d\big)$ if
$u_j\to u$ (strongly) in $L^1\big(\Omega;\RR^d)$
and $Du_j\weaklystar Du$ weakly-$\star$ in
${\cal M}\big(\Omega;\RR^{d\times N}\big)$.

The proof of the following  result may be found
in  \cite[Prop.~{3.7}]{AFPMM} and \cite[Lemma~{4.5}]{AMTXCI}.

\begin{lemma}\label{lemmaAMT}
Let  $u\in BV(\Omega;\RR^d)$ and let $u_\delta\in
C^\infty(\Omega_\delta)$ be given by \eqref{defmollfct}.
Then,
\begin{itemize}
\item[i)] $u_\delta\weaklystar u$ weakly-$\star$
in $BV(\Omega';\RR^d)$ as $\delta\to0^+$, for
all $\Omega'\subset\subset\Omega$. Moreover,
if $\calL^N(\partial\Omega')=0$, then
\begin{eqnarray*}
\lim_{\delta\to0^+}|D u_\delta|(\Omega') =
\lim_{\delta\to0^+}\int_{\Omega'} |\grad u_\delta(x)|\,\dx
= |Du|(\Omega');
\end{eqnarray*}
\item[ii)] 
\begin{eqnarray*}
\int_{B(x_0,\epsi)} \mathfrak{h}(x) |\grad u_\delta(x)|\,\dx
\lii \int_{B(x_0,\epsi + \delta)} (\mathfrak{h}*\rho_\delta)(x)\,\d|Du|(x)
\end{eqnarray*}
whenever $\dist(x_0,\partial\Omega)>\epsi +
\delta$ and $\mathfrak{h}$ is a nonnegative Borel function;

\item[iii)] 
\begin{eqnarray*}
\lim_{\delta\to0^+}\int_{B(x_0,\epsi)} \theta(\grad
u_\delta(x))\,\dx =  \int_{B(x_0,\epsi)} \theta\Big({\d
Du\over \d |Du|}(x)\Big)\,\d|Du|(x)
\end{eqnarray*}
for every positively 1-homogeneous continuous
function $\theta$ and for every $\epsi\in (0,\dist(x_0,\partial\Omega))$
such that $|Du|(\partial B(x_0,\epsi)) = 0$;

\item[iv)] 
\begin{eqnarray*}
\lim_{\delta\to0^+}(|u_\delta - u|*\rho_\delta)(x)
= 0\enspace \hbox{for all } x\in A_u
\end{eqnarray*}
if, in addition,  $u\in L^\infty(\Omega;\RR^d)$.
\end{itemize}
\end{lemma}

In what follows, $Du=D^au + D^su$ is the Radon-Nikodym
decomposition of $Du$ in absolutely continuous
and singular parts with respect to 
${{\cal L}^N}_{\!\lfloor\Omega}$.
The proof of the following  results may be
found in   \cite[Rmk.~{3.93}, Thm.~{3.83},
Thm.~{3.78}]{AFPMM}.

\begin{lemma}\label{loctyDu}
Let $u_1, u_2\in BV(\Omega;\RR^d)$ and $A:=\{x\in
A_{u_1}\cap A_{u_2}\!: \, \tilde u_1(x) = \tilde
u_2(x)\}$. Then ${D{u_1}}_{\lfloor A}={D{u_2}}_{\lfloor
A}$.
\end{lemma}

\begin{theorem}\label{appdiffisac}
Let $u\in BV(\Omega;\RR^d)$. Then,
\begin{itemize}
\item [(a)] $u$ is approximately differentiable
at $\calL^N$-\aev\ point of $\Omega$, and the
approximate differential $\grad u$ is the density
of the absolutely continuous part of  $Du$
with respect to ${{\cal L}^N}_{\!\lfloor\Omega}$; that is, $D^a u=\grad
u{{\cal L}^N}_{\!\lfloor\Omega}$;

\item[(b)] the set $S_u$ is countably $\calH^{N-1}$-rectifiable
and $\calH^{N-1}(S_u\setminus J_u)=0$. Moreover,
$Du_{\lfloor J_u} = (u^+-u^-)\otimes \nu_u
{\calH^{N-1}}_{\lfloor J_u}$.
\end{itemize}
\end{theorem}

\begin{definition}\label{DefDjDc}
Given $u\in BV(\Omega;\RR^d)$, the measures
\begin{eqnarray*}
D^j u:= D^su_{\lfloor J_u} \enspace{ and }\enspace
D^cu:= D^su_{\lfloor A_u}
\end{eqnarray*}
are called the jump part of the derivative
and the Cantor part of the derivative, respectively.
The sum $D^a u + D^cu$ is called the diffuse
part of the derivative and is denoted by $\tilde
Du$.
\end{definition}

\begin{remark}\label{decforDu}
It can be proved that $D^ju = Du_{\lfloor J_u}$
(see \cite[Prop.~{3.92}]{AFPMM}). In view
of Definition~\ref{DefDjDc} and Theorem~\ref{appdiffisac},
we have the following decompositions for $Du$:
\begin{eqnarray*}
Du= D^a u + D^su = \grad u{\calL^N}_{\!\lfloor
\Omega} + (u^+-u^-)\otimes
\nu_u {\calH^{N-1}}_{\lfloor J_u} + D^cu =\tilde
Du + D^cu.
\end{eqnarray*}
\end{remark}

The next result is due to G. Alberti (see \cite{AlXCIII}).

\begin{theorem}\label{rank1Alberti}  If $u\in
BV(\Omega;\RR^d)$ and $Du = h|Du|$, then $h$
has rank one for $(|D^ju| + |D^cu|)$-\aev\
point of $\Omega$.
\end{theorem}

We now state a result regarding the chain rule
in $BV$, which proof may be found in \cite[Thm.~{3.96}]{AFPMM}.

\begin{theorem}\label{chainruleBV}
Let $u\in BV(\Omega;\RR^d)$ and $\phi\in C^1(\RR^d;\RR^m)$
be a Lipschitz function satisfying $\phi(0)=0$
if $\calL^N(\Omega)=+\infty$. Then $v:= \phi\circ
u$ belongs to $BV(\Omega;\RR^m)$, and 
\begin{eqnarray}\label{chruBV}
\begin{aligned}
& \tilde D v = \grad \phi(u)\grad u{\calL^N}_{\!\lfloor
\Omega}
+ \grad \phi(\tilde u) D^cu = \grad \phi(\tilde
u) \tilde D u, \qquad D^j v = \big(\phi(u^+)
- \phi(u^-)\big)\otimes \nu_u{\calH^{N-1}}_{\lfloor
J_u}.
\end{aligned}%
\end{eqnarray}
\end{theorem}

As a consequence of Lebesgue-Besicovitch Differentiation Theorem, we have the following.

\begin{theorem}\label{conseqLBDT}
If $\mu$ is a nonnegative Radon measure and if $v\in L^1_{\rm loc}(\Omega,\mu;\RR^d)$, then
\begin{eqnarray*}
\lim_{\epsilon\to0^+}\dashint_{x_0 + \epsilon C} |v(y) -v(x_0)|\,\d\mu(y)=0
\end{eqnarray*}
for $\mu$-\aev\ $x_0\in\Omega$ and for every bounded, convex, open set $C$ containing the origin.
\end{theorem}

In the remainder of this subsection  $\Omega$ denotes an open
subset of $\RR^N$ and
$E$ a $\calL^N$-measurable subset of $\RR^N$. 

\begin{definition}\label{Perimeter}
The perimeter of $E$ in
$\Omega$ is represented
by ${\rm Per}_\Omega(E)$
and  defined by
\begin{eqnarray*}
\begin{aligned}
& {\rm Per}_\Omega(E):=
\sup \bigg\{ \int_E {\rm
div}\,\ffi\,\dx\!:\,
\ffi\in C^1_c(\Omega;\RR^N),\,
\Vert\ffi\Vert_{L^\infty(\Omega)}\lii
1\bigg\}. 
\end{aligned}
\end{eqnarray*}
We say that $E$ is a
set of finite perimeter
in $\Omega$ if ${\rm
Per}_\Omega(E)<+\infty$.
\end{definition}

The proof of the following
 result may be found
in  \cite[Thm.~{3.36}]{AFPMM}.

\begin{theorem}\label{ThmSFP}
Assume that $E$ is a set of finite perimeter
in $\Omega$. Then the
distributional derivative
of $\chi_E$, $D\chi_E$, belongs to $\calM(\Omega;\RR^N)$
and $|D\chi_E| (\Omega)=
{\rm Per}_\Omega(E)$.
Moreover, the following
generalized Gauss--Green
formula holds
\begin{eqnarray*}
\begin{aligned}
& \int_E {\rm div}\,
\ffi \,\dx = - \int_\Omega
 \nu_E\cdot\ffi\,
\d|D\chi_E|,\quad \hbox{for all }\ffi\in
C^1_c(\Omega;\RR^N),
\end{aligned}
\end{eqnarray*}
where $D\chi_E = \nu_E
|D\chi_E|$ is the polar
decomposition of $D\chi_E$.
\end{theorem}

\begin{definition}\label{RedBound}
Let $E$ be a set of finite
perimeter in $\Omega$.
The reduced boundary
of $E$, denoted by $\calF
^*E$, is the set of all
points $x\in \Omega$
such that for all $\epsilon>0$,
\begin{eqnarray*}
\begin{aligned}
& |D\chi_E|(B(x,\epsilon)\cap
\Omega)>0,
\end{aligned}
\end{eqnarray*}
and such that the limit
\begin{eqnarray*}
\begin{aligned}
& \nu_E(x):=\lim_{\epsilon\to0^+}
\frac{D\chi_E (B(x,\epsilon))}{|D\chi_E| (B(x,\epsilon))}
\end{aligned}
\end{eqnarray*}
exists in $\RR^N$ and
satisfies $|\nu_E(x)|=1$.
The function $\nu_E:\calF^*E
\to S^{N-1}$ is called
the generalised inner
normal to $E$.
\end{definition}

\begin{definition}\label{tDensPts}
Given $t\in[0,1]$, we
represent by $E^t$ the
set of all points where
$E$ has density $t$,
i.e.,
\begin{eqnarray*}
\begin{aligned}
& E^t:= \bigg \{ x\in\RR^N\!:\,
\lim_{\epsilon\to0^+}
\frac{\calL^N (E\cap
B(x,\epsilon))}{\calL^N
(B(x,\epsilon))}=t\bigg\}.
\end{aligned}
\end{eqnarray*}
The set $\partial^*E:=
\RR^N\backslash (E^0
\cup E^1)$ is called
the essential boundary
of $E$.
\end{definition}

The proof of the following
 theorem may be found
in  \cite[Thm.~{3.59,
Thm.~{3.61}, Example~{3.68}}]{AFPMM}.

\begin{theorem}\label{DG&FE}
Let $E$ be a set of finite
perimeter in $\Omega$.
\begin{itemize}
\item [(i)] (De Giorgi)
The set $\calF^*E$ is
contained, up to $\calH^{N-1}$
negligible sets, in a
countable union of $C^1$
hypersurfaces, and
\begin{eqnarray*}
\begin{aligned}
& D\chi_E = \nu_E \,{\calH^{N-1}}_{\lfloor
\calF^*E},\quad  |D\chi_E| = {\calH^{N-1}}_{\lfloor
\calF^*E},
\end{aligned}
\end{eqnarray*}
where $\nu_E$ is the
generalised inner
normal to $E$.

\item [(ii)] (Federer)
It holds
\begin{eqnarray*}
\begin{aligned}
& \calF^*E \subset E^{1/2}
\subset \partial^*E,\quad
\calH^{N-1}(\Omega\backslash
(E^0 \cup \calF^*E \cup
E^1)) =0.
\end{aligned}
\end{eqnarray*}
In particular, $E$ has
density
either $0$ or $1/2$ or
$1$ at $\calH^{N-1}$-\aev\
$x\in\Omega$ and $\calH^{N-1}$-\aev\
$x\in \partial^*E \cap \Omega$ belongs to $\calF^*E$.

\item [(iii)] Setting
$u:=\chi_E$, then $u\in
BV(\Omega)$ with $S_u
= \partial^*E \cap \Omega$, $\calF^*E
\subset J_u \subset E^{1/2}$,
and $\{u^+(x), u^-(x)\}=
\{0,1\}$ for all $x\in
J_u$.
\end{itemize}
\end{theorem}

\begin{remark}\label{AppByPoly}
Another property of 
sets of finite perimeter
in $\RR^N$, which is
due to De Giorgi \cite{DeGiorgi54}, is
the following. If $E$
is a set of finite perimeter
in $\Omega$, then there
exists a sequence of
open sets $\{E_n\}_{n\in\NN}$
such that each set $\partial
E_n$ is contained in
a finite number of hyperplanes
and 
\begin{eqnarray*}
\begin{aligned}
& \lim_{n\to\infty} \calL^N(
E_n \Delta E) =0, \quad
\lim_{n\to\infty} |D\chi_{E_n}|(\Omega)
= |D\chi_{E}|(\Omega).
\end{aligned}
\end{eqnarray*}
\end{remark}

\subsection{Quasiconvex  Functions
}\label{qqcxfct}

We say that a Borel function $h:\RR^N\times \RR^{d\times N}\to\RR$ is quasiconvex if for all $(\xi,\eta)\in\RR^N\times\RR^{d\times N}$, $\ffi\in W^{1,\infty}_0(Q)$, and $\psi\in
W^{1,\infty}_0(Q;\RR^d)$, we have
{\setlength\arraycolsep{0.5pt}%
\begin{eqnarray*}
\begin{aligned}
h(\xi,\eta)\lii\int_Q h(\xi + \grad\ffi(x),\eta
+ \grad\psi(x))\,\dx.
\end{aligned}
\end{eqnarray*}}%
\begin{remark}\label{relusualqcx}
Consider the mapping that to each matrix $A=(a_{ij})_{1\lii i\lii d+1, 1\lii j\lii N}\in \RR^{(d+1)\times N}$ associates the pair $(\xi_A,\eta_A)\in \RR^N\times \RR^{d\times N}$, where $\xi_A:=(a_{(d+1)j})_{1\lii j\lii N}$ is the last row of $A$, and $\eta_A:=(a_{ij})_{1\lii i\lii d, 1\lii j\lii N}$ is obtained from $A$ by erasing its last row. Then, to a Borel function $h:\RR^N\times \RR^{d\times N}\to\RR$ we may associate the Borel function $\bar h:\RR^{(d+1)\times N}\to\RR$ defined by
{\setlength\arraycolsep{0.5pt}%
\begin{eqnarray*}
\begin{aligned}
\bar h(A):=h(\xi_A,\eta_A), \quad A\in \RR^{(d+1)\times N}.
\end{aligned}
\end{eqnarray*}}%
In this setting, $h$ is a quasiconvex  function if and only if $\bar h$ is a quasiconvex function in the \textit{usual} sense; that is, for all $A \in \RR^{(d+1)\times N}$ and $\vartheta\in W^{1,\infty}_0(Q;\RR^{(d+1)})$,
{\setlength\arraycolsep{0.5pt}%
\begin{eqnarray*}
\begin{aligned}
\bar h(A)\lii \int_Q \bar h(A+\grad \vartheta(x))\,\dx.
\end{aligned}
\end{eqnarray*}}%
\end{remark}
In view of Remark~\ref{relusualqcx} and well-know
results concerning the \textit{usual} notion
of quasiconvexity, if $h:\RR^N\times
\RR^{d\times N}\to\RR$ is a quasiconvex for
which there exists  a positive constant $C$ such that
for
all $(\xi,\eta)\in \RR^N \times \RR^{d\times N}$,
 \begin{eqnarray*}
0\lii h(\xi,\eta)\lii C(1+|(\xi,\eta)|),
\end{eqnarray*}
then \(h\) is Lipschitz; i.e., there exists a constant $L>0$, only depending on $C$, $N$, and $d$, such that
for all $(\xi,\eta), (\xi',\eta') \in  \RR^N \times \RR^{d\times
N}$,
\begin{eqnarray}\label{qcxisLip}
|h(\xi,\eta) - h(\xi',\eta')|\lii L|(\xi,\eta) -(\xi',\eta')|.
\end{eqnarray}

\begin{remark}
We finish this subsection by noting that for all
$(r,s)\in [\alpha,\beta] \times S^2$, the function
$f(r, s, \cdot, \cdot)$ introduced in \eqref{defoff}
is not quasiconvex in $\RR^2 \times \RR^{3 \times
2}$. In fact, if it were, then so would be its
recession function $f^\infty(r,s,\cdot,\cdot)$
introduced in \eqref{frec} (see \cite[Rmk.~2.2]{FMXCIII}).
In turn, by \eqref{qcxisLip},
$f^\infty(r,s,\cdot,\cdot)$ would be 
continuous in $\RR^2\times\RR^{3\times2}$.
However, taking $\xi_n\in\RR^2
\backslash \{0\}$
 and $\eta\in\RR^{3\times2}\backslash
\{0\}$
 such that $\xi_n\to0$
 as $n\to\infty$, we
 conclude that $\lim_{n\to\infty}
 f^\infty(r,s,\xi_n,\eta)
 = |r\eta| \not= |r\eta|
 +|\eta| =f^\infty(r,s,0,\eta)$. Thus, neither
$f(r, s, \cdot, \cdot)$ nor $f^\infty(r, s, \cdot,
\cdot)$ are quasiconvex functions.
\end{remark}

\section{Proof of Theorem~\ref{mainthm}}\label{proofmainthm}

This subsection is devoted to the proof of 
Theorem~\ref{mainthm}, under the assumption that
Theorem~\ref{relaxthm} holds. We start by observing
that there are admissible fields as introduced
in   \eqref{defsetX}.

\begin{lemma}\label{Xnotempty}
The set 
\begin{eqnarray*}
\begin{aligned}
 X=\big\{(u_b,u_c)\in BV(\Omega;[\alpha,\beta])
\times BV(\Omega;S^2)\!:\, &\,u_b-(u_0)_b\in
G(\Omega), \, u_b u_c -
u_0 \in G(\Omega;\RR^3)\big\}
\end{aligned}
\end{eqnarray*}
is nonempty.
\end{lemma}

\begin{proof}
Recall that
\aev\ in $\Omega$,
\begin{eqnarray*}
\begin{aligned}
(u_0)_b = |u_0|\in [\alpha,\beta], \quad (u_0)_c =  {u_0\over|u_0|}= {u_0\over (u_0)_b}\in S^2, \quad u_0 = (u_0)_b (u_0)_c,
\end{aligned}
\end{eqnarray*}
 and $0<\alpha\lii\beta$.

Let $\displaystyle c_0:=\dashint_\Omega (u_0)_b\,\dx$ and set
\begin{eqnarray*}
\begin{aligned}
u_b(x):= c_0 \quad \hbox{for all } x \in \Omega.
\end{aligned}
\end{eqnarray*}
Clearly, $u_b \in BV(\Omega; [\alpha,\beta])$ and $
\int_\Omega u_b\,\dx = \int_\Omega (u_0)_b\,\dx$;
 since $u_b - (u_0)_b \in L^\infty (\Omega) \subset L^2(\Omega)$,  it follows that $u_b-(u_0)_b\in
G(\Omega)$ by Proposition~\ref{Galtern}.
 
Because
\begin{eqnarray*}
\begin{aligned}
\bigg|\dashint_\Omega (u_0)_b (u_0)_c\,\dx \bigg| \lii \dashint_\Omega (u_0)_b \, \dx = c_0,
\end{aligned}
\end{eqnarray*}
we have $\dashint_\Omega (u_0)_b (u_0)_c\,\dx \in \overline{B(0, c_0)}\subset \RR^3$; thus, there exist $\theta \in [0,1]$ and $s_1$, $s_2 \in \partial B(0,c_0)$ such that
\begin{eqnarray*}
\begin{aligned}
\dashint_\Omega (u_0)_b (u_0)_c\,\dx = \theta s_1 + (1-\theta) s_2.
\end{aligned}
\end{eqnarray*}
Let $\{\Omega_1,\Omega_2\}$ be a Lipschitz partition of $\Omega$ satisfying  $\calL^2(\Omega_1) = \theta \calL^2(\Omega)$, $\calL^2(\Omega_2)
= (1-\theta) \calL^2(\Omega)$, and consider the function $u_c$ defined, for $x\in\Omega$,
by
\begin{eqnarray*}
\begin{aligned}
u_c(x):= \begin{cases}
\displaystyle\frac{s_1}{c_0} & \hbox{if } x\in \Omega_1,\\ \\
 \displaystyle\frac{s_2}{c_0} & \hbox{if } x\in
\Omega_2.
\end{cases}
\end{aligned}
\end{eqnarray*}
Then, $u_c \in BV(\Omega; S^2)$, $u_b u_c - u_0 \in L^\infty(\Omega;\RR^3)$, and  
$\int_\Omega u_b u_c \, \dx = \int_\Omega (u_0)_b (u_0)_c\,\dx = \int_\Omega u_0\,\dx$. Thus, $u_b u_c -
u_0 \in G(\Omega;\RR^3)$, and this completes the proof.  
\end{proof}

\begin{proof}[Proof of Theorem~\ref{mainthm}] Fix sequences $\{\epsi_n\}_{n\in\NN}$ and 
$\{\delta_n\}_{n\in\NN}$
  as in the statement.
Let $G_n$,
$G_0: L^1(\Omega) \times L^1(\Omega;\RR^3)
\to [0,+\infty]$ be the functionals defined, for $(u_b,u_c) \in
L^1(\Omega) \times L^1(\Omega;\RR^3),$ by
\begin{eqnarray*}
\begin{aligned}
G_n (u_b, u_c):= 
\begin{cases}
F^{reg} (u_b, u_c) + F^{fid}_{\epsi_n} (u_b, u_c)
 & \hbox{if } (u_b, u_c) \in W^{1,1}(\Omega;
 [\alpha,\beta]) \times  W^{1,1}(\Omega; S^2),\\
 +\infty & \hbox{otherwise},
\end{cases}
\end{aligned}
\end{eqnarray*}
and
\begin{eqnarray*}
\begin{aligned}
G_0 (u_b, u_c):= 
\begin{cases}
F^{reg,sc^-} (u_b, u_c) + F^{fid} (u_b, u_c)
 & \hbox{if } (u_b, u_c) \in X,\\
 +\infty & \hbox{otherwise},
\end{cases}
\end{aligned}
\end{eqnarray*}
respectively.

We claim that $\{G_n \}_{n\in\NN}$ $\Gamma$-converges to $G_0$ in $L^1(\Omega) \times L^1(\Omega;\RR^3)
$.
Invoking \cite[Prop.~8.1]{DMXCIII}, this claim follows from Steps~1 and 2 below.

\underbar{Step 1.} \textit{(liminf inequality)} Fix $(u_b, u_c) \in L^1(\Omega) \times L^1(\Omega;\RR^3),$
and let $\{( u_b^n, u_c^n)\}_{n\in\NN}
\subset L^1(\Omega) \times L^1(\Omega;\RR^3)$
be an arbitrary sequence converging to $(u_b, u_c)$ in $L^1(\Omega)
\times L^1(\Omega;\RR^3)$. We claim that
\begin{eqnarray}\label{liineq}
\begin{aligned}
G_0(u_b, u_c) \lii \liminf_{n\to\infty} G_n
( u_b^n, u_c^n).
\end{aligned}
\end{eqnarray}
Without loss of generality, we assume that the limit inferior on the right-hand side of \eqref{liineq} is  a limit and  is finite, with $\sup_{n\in\NN} G_n( u_b^n, u_c^n) <+\infty$. Then, $\{ (u_b^n, u_c^n)\}_{n\in\NN} \subset W^{1,1}(\Omega;
 [\alpha,\beta]) \times  W^{1,1}(\Omega; S^2)$ and there exists a positive constant, $C$, independent of $n\in\NN$, such that
\begin{eqnarray*}
\begin{aligned}
& C \gii F^{reg} (u_b^n, u_c^n) = \int_\Omega
f(u_b^n, u_c^n, \grad u_b^n, \grad u_c^n) \,
\dx \gii \frac{1}{2}\int_\Omega |\grad u_b^n|\,\dx
+ \frac{\alpha}{2}\int_\Omega |\grad u_c^n|\,\dx, \\
&C \gii F^{fid} _{\epsi_n}(u_b^n, u_c^n)  \gii \frac{1}{\epsi_n} \bigg| \int_\Omega
(u_b^n u_c^n - u_0)\,\dx\bigg| + \frac{1}{\epsi_n}
\bigg| \int_\Omega
(u_b^n  - (u_0)_b)\,\dx\bigg|,
\end{aligned}
\end{eqnarray*}
where we used \eqref{boundsf} together with the fact that $u_b^n \gii \alpha$ and $u_c^n\in
S^2$ \aev\ in $\Omega$. Consequently,
{\setlength\arraycolsep{0.5pt}
\begin{eqnarray}
&&u_b^n \weaklystar u_b \hbox{ weakly-$\star$  in $BV(\Omega)$}, \quad u_c^n \weaklystar u_c \hbox{ weakly-$\star$
 in $BV(\Omega;\RR^3)$}, \quad \hbox{as $n\to\infty$,}\nonumber \\
&&\lim_{n\to\infty} \int_\Omega u_b^n u_c^n \, \dx = \int_\Omega u_0 \, \dx,   \quad \lim_{n\to\infty} \int_\Omega u_b^n 
\, \dx = \int_\Omega (u_0)_b \, \dx,\label{=avgu0}
\end{eqnarray}}%
and, up to a (not relabeled) subsequence,
\begin{eqnarray*}
\begin{aligned}
u^n_b \to u_b \hbox{ \aev\ in $\Omega$},\quad u_c^n \to u_c \hbox{ \aev\ in $\Omega$}, \quad \hbox{as $n\to\infty$.}
\end{aligned}
\end{eqnarray*}
These two last convergences yield $u_b\in[\alpha,\beta]$ and $u_c \in S^2$ \aev\ in $\Omega$. By Lebesgue's
Dominated Convergence Theorem, we have
{\setlength\arraycolsep{0.5pt}
\begin{eqnarray*}
\begin{aligned}
\lim_{n\to\infty} \int_\Omega u_b^n u_c^n
\, \dx = \int_\Omega u_b u_c\dx,   \quad \lim_{n\to\infty}
\int_\Omega u_b^n 
\, \dx = \int_\Omega u_b\dx,
\end{aligned}
\end{eqnarray*}}%
which, together with \eqref{=avgu0} and in view of Proposition~\ref{Galtern}, implies $(u_b , u_c) \in X$. 
Furthermore, by Theorem~\ref{relaxthm},
{\setlength\arraycolsep{0.5pt}%
\begin{eqnarray}\label{liminfFreg}
\begin{aligned}
F^{reg, sc^-}(u_b, u_c) \lii \liminf_{n\to\infty} F^{reg} ( u_b^n, u_c^n).
\end{aligned}
\end{eqnarray}}%
Finally, we prove that
{\setlength\arraycolsep{0.5pt}%
\begin{eqnarray}\label{liminfFfid}
\begin{aligned}
F^{fid}(u_b, u_c) \lii \liminf_{n\to\infty}
F^{fid}_{\epsi_n} ( u_b^n, u_c^n),
\end{aligned}
\end{eqnarray}}%
which, together with \eqref{liminfFreg}, yields \eqref{liineq}.

The sequence $\{ u_b^n u_c^n - u_0 - \dashint_\Omega
(u_b^n u_c^n \,  - u_0 ) \, \dx \}_{n\in\NN}\subset
G(\Omega;\RR^3)$   converges \aev\ in $\Omega$
 to $u_b u_c- u_0$ and is  bounded  in
 $L^\infty(\Omega;\RR^3)$; hence, the convergence
 holds weakly in
\(L^p(\Omega)\) for any \(p>2\). Consequently,
by Proposition~\ref{normGto0},
we have
{\setlength\arraycolsep{0.5pt}%
\begin{eqnarray*}
\begin{aligned}
\lim_{n\to\infty}  \bigg\Vert
u_b^n u_c^n -
u_0  - \dashint_\Omega
(u_b^n u_c^n - u_0)\,\dx \bigg \Vert_{G(\Omega;\RR^3)} =  \Vert
u_b u_c -
u_0  \Vert_{G(\Omega;\RR^3)}.
\end{aligned}
\end{eqnarray*}}%
Similarly,
{\setlength\arraycolsep{0.5pt}%
\begin{eqnarray*}
\begin{aligned}
\lim_{n\to\infty}  \bigg\Vert
u_b^n  -
(u_0)_b - \dashint_\Omega
(u_b^n  - (u_0)_b)\,\dx \bigg \Vert_{G(\Omega)}
=  \Vert
u_b  -
(u_0)_b  \Vert_{G(\Omega)}.
\end{aligned}
\end{eqnarray*}}%
Moreover, because $u_c^n, \,(u_0)_c \in S^2$ and $u_c^n \to u_c$ \aev\ in $\Omega$, as $n\to \infty$, 
{\setlength\arraycolsep{0.5pt}%
\begin{eqnarray*}
\begin{aligned}
\lim_{n\to\infty} \int_\Omega | u_c^n - (u_0)_c|^2\,\dx = \int_\Omega | u_c - (u_0)_c|^2\,\dx.
\end{aligned}
\end{eqnarray*}}%
Thus,
{\setlength\arraycolsep{0.5pt}%
\begin{eqnarray*}
\begin{aligned}
\liminf_{n\to\infty}
F^{fid}_{\epsi_n} ( u_b^n, u_c^n) &\gii \liminf_{n\to\infty} \bigg(\lambda_v \bigg\Vert
u_b^n u_c^n -
u_0  - \dashint_\Omega
(u_b^n u_c^n - u_0)\,\dx \bigg \Vert_{G(\Omega;\RR^3)}\\
&\qquad + \lambda_b \bigg\Vert
u_b^n  -
(u_0)_b - \dashint_\Omega
(u_b^n  - (u_0)_b)\,\dx \bigg \Vert_{G(\Omega)} +\lambda_c  \int_\Omega | u_c^n - (u_0)_c|^2\,\dx\bigg)\\
&= F^{fid} (u_b,u_c),
\end{aligned}
\end{eqnarray*}}%
which proves \eqref{liminfFfid}. This concludes Step~1.

\underbar{Step 2.} \textit{(limsup inequality)}
 Fix $(u_b, u_c) \in L^1(\Omega) \times
  L^1(\Omega;\RR^3)$.
We claim that
there exists a sequence $\{( u_b^n, u_c^n)\}_{n\in\NN}
\subset L^1(\Omega) \times L^1(\Omega;\RR^3)$
converging to $(u_b, u_c)$ in $L^1(\Omega)
\times L^1(\Omega;\RR^3)$ that satisfies
\begin{eqnarray}\label{lsineq}
\begin{aligned}
G_0(u_b, u_c) \gii \limsup_{n\to\infty} G_n
( u_b^n, u_c^n).
\end{aligned}
\end{eqnarray}
Without loss of generality, we may assume that $(u_b,u_c) \in X$. By Theorem~\ref{relaxthm}, and recalling the bounds \eqref{boundsf}
 for $f$ (see also \eqref{Freg}), we can find a sequence $\{(u_b^j, u_c^j)\}_{j\in \NN} \subset W^{1,1}(\Omega;[\alpha,\beta]) \times W^{1,1}(\Omega; S^2)$ such that
{\setlength\arraycolsep{0.5pt}%
\begin{eqnarray}
&& u_b^j \weaklystar u_b \hbox{ weakly-$\star$
 in $BV(\Omega)$}, \enspace u_c^j \weaklystar
u_c \enspace \hbox{and}\enspace  u_b^j u_c^j \weaklystar
u_b u_c \enspace\hbox{weakly-$\star$
 in $BV(\Omega;\RR^3)$}, \quad \hbox{as $j\to\infty$;} \nonumber\\
&& u^j_b \to u_b,\enspace
u_c^j \to u_c,  \enspace \hbox{and} \enspace\ u_b^j u_c^j \to u_b u_c \enspace \hbox{\aev\ in $\Omega$}, \quad
\hbox{as $j\to\infty$;} \nonumber\\
&& F^{reg, sc^-}(u_b, u_c) = \lim_{j\to\infty}
F^{reg} ( u_b^j, u_c^j). \label{limFreg}
\end{eqnarray}}%
In particular, arguing as in Step~1,
{\setlength\arraycolsep{0.5pt}%
\begin{eqnarray}
F^{fid} (u_b,u_c) &&=  \lim_{j\to\infty}
\bigg(\lambda_v \bigg\Vert
u_b^j u_c^j -
u_0  -  \dashint_\Omega
(u_b^j u_c^j - u_0)\,\dx \bigg \Vert_{G(\Omega;\RR^3)} \nonumber\\
&&\qquad\qquad + \lambda_b \bigg\Vert
u_b^j  -
(u_0)_b - \dashint_\Omega
(u_b^j  - (u_0)_b)\,\dx \bigg \Vert_{G(\Omega)}
+\lambda_c  \int_\Omega | u_c^j - (u_0)_c|^2\,\dx\bigg). \label{limFfid1}
\end{eqnarray}}%
Moreover, recalling Proposition~\ref{Galtern} and the fact that $(u_b,u_c)\in X$,
{\setlength\arraycolsep{0.5pt}%
\begin{eqnarray*}
\begin{aligned}
\lim_{j\to\infty}  \int_\Omega u_b^j u_c^j
\, \dx = \int_\Omega u_b u_c\,\dx = \int_\Omega u_0 \, \dx ,   \quad \lim_{j\to\infty}
\int_\Omega u_b^j 
\, \dx = \int_\Omega u_b\,\dx = \int_\Omega (u_0 \, )_b \,\dx. 
\end{aligned}
\end{eqnarray*}}%
Hence,  we can find a subsequence $j_n\preceq j$ such that
{\setlength\arraycolsep{0.5pt}%
\begin{eqnarray}\label{limFfid2}
\begin{aligned}
\bigg|\int_\Omega u^{j_n}_b u^{j_n}_c\, \dx -   \int_\Omega
u_0 \, \dx \bigg| \lii \epsi_n^2, \quad \bigg|\int_\Omega u^{j_n}_b\, \dx
-   \int_\Omega
(u_0)_b \, \dx \bigg| \lii \epsi_n^2.
\end{aligned}
\end{eqnarray}}%
From \eqref{limFreg}, \eqref{limFfid1}, and \eqref{limFfid2}, we obtain \eqref{lsineq} for $\{(u_b^{j_n}, u_c^{j_n})\}_{n\in\NN}$, which concludes Step~2.

We now observe that $\{G_n\}_{n\in\NN}$ is equi-coercive in $L^1(\Omega) \times L^1(\Omega;\RR^3)$. In fact, arguing as in Step~1 above, given $C\in\RR$ we can find a positive constant  $c = c(C, \alpha)$ such that
for all $n\in\NN$,
{\setlength\arraycolsep{0.5pt}%
\begin{eqnarray}
&&\big\{ (u_b, u_c) \in L^1(\Omega) \times L^1(\Omega;\RR^3) \!: \, G_n(u_b, u_c) \lii C \big\} \nonumber \\ 
&&\qquad \subset \big\{ (u_b, u_c) \in W^{1,1}(\Omega) \times W^{1,1}(\Omega;\RR^3)
\!: \,\Vert (u_b, u_c)\Vert_{ W^{1,1}(\Omega) \times
W^{1,1}(\Omega;\RR^3)} \lii c \big\}, \label{Gnequic}
\end{eqnarray}}%
which, together with the compact injection of $ W^{1,1}(\Omega)
\times
W^{1,1}(\Omega;\RR^3)$ into $L^1(\Omega) \times L^1(\Omega;\RR^3)$, yields the conclusion.

We have just proved that $\{G_n\}_{n\in\NN}$ is an
equi-coercive
sequence that $\Gamma$-converges to $G_0$ in $L^1(\Omega) \times L^1(\Omega;\RR^3)$. Therefore, by \cite[Thm~7.8, Cor.~7.20]{DMXCIII}, we have
{\setlength\arraycolsep{0.5pt}%
\begin{eqnarray}\label{convinfGn}
\begin{aligned}
\min_{ (u_b, u_c) \in L^1(\Omega) \times
L^1(\Omega;\RR^3)} G_0 (u_b, u_c) = \lim_{n\to\infty } \inf_{ (u_b, u_c) \in L^1(\Omega) \times
L^1(\Omega;\RR^3)} G_n (u_b, u_c).
\end{aligned}
\end{eqnarray}}%
Note that by Lemma~\ref{Xnotempty} and
\eqref{boundsf},
the minimum on the left-hand side of \eqref{convinfGn}
is
finite, and \eqref{convinfGn} is equivalent
to saying
that
\begin{equation}
\label{convinf}
\begin{aligned}
&\min_{(u_b,u_c)\in  X} \big (F^{reg,
sc^-}(u_b,u_c) + F^{fid}(u_b,u_c)\big)
= \lim_{n\to\infty} \inf_{(u_b,u_c)\in
W^{1,1}(\Omega;[\alpha,\beta])\times
W^{1,1}(\Omega;S^2)} \big(F^{reg}(u_b,u_c)+
F^{fid}_{\epsi_n}(u_b,u_c)\big).
\end{aligned}
\end{equation}
Let
 $( u_b^n,
u_c^n)\in W^{1,1}(\Omega;[\alpha,\beta])\times
W^{1,1}(\Omega;S^2)$ be a $\delta_n$-minimizer
of the functional $\big(F^{reg}+F^{fid}_{\epsi_n}\big)$
 in $W^{1,1}(\Omega;[\alpha,\beta])\times
W^{1,1}(\Omega;S^2)$. Observe that \((
u_b^n,
u_c^n)\)  is also a $\delta_n$-minimizer
of $G_n$ in $L^1(\Omega) \times
L^1(\Omega;\RR^3)$,  $G_n( u_b^n,
u_c^n) = F^{reg}( u_b^n,
u_c^n)\allowbreak + F^{fid}_{\epsi_n}(
u_b^n,
u_c^n)$, and, by \eqref{convinf},
\begin{equation}\label{minseqsc-}
\begin{aligned}
\min_{(u_b,u_c)\in  X} \big (F^{reg,
sc^-}(u_b,u_c) + F^{fid}(u_b,u_c)\big)
= \lim_{n\to\infty} \Big(
F^{reg}( u_b^n, u_c^n) + F^{fid}_{\epsi_n}(
u_b^n, u_c^n)\Big). 
\end{aligned}
\end{equation}
Using \eqref{Gnequic} and   the fact that
the minimum on the left-hand side of 
\eqref{minseqsc-} is finite, we deduce
that $\{( u_b^n, u_c^n)\}_{n\in\NN}$
is sequentially, relatively compact with
respect
to the weak-$\star$ convergence in $BV(\Omega)\times
BV(\Omega;\RR^3)$ and its cluster points
belong to \(X\).
Let  $(\bar u_b, \bar u_c)\in X $ be a
cluster point of $\{ (u_b^n,
u_c^n) \}_{n\in\NN}$. Then, by Step~1
and \eqref{minseqsc-},
\begin{equation*}
\begin{aligned}
F^{reg,
sc^-}(\bar u_b, \bar u_c) + F^{fid}(\bar
u_b, \bar u_c) &= G_0(\bar u_b, \bar u_c)
 \lii
\liminf_{n\to\infty} G_n ( u_b^n,
u_c^n) = \liminf_{n\to\infty} \big( F^{reg}(
u_b^n,
u_c^n)+F^{fid}_{\epsi_n}( u_b^n,
u_c^n) \big)\\
&= \min_{(u_b,u_c)\in  X} \big (F^{reg,
sc^-}(u_b,u_c) + F^{fid}(u_b,u_c)\big)
\lii F^{reg,
sc^-}(\bar u_b, \bar u_c) + F^{fid}(\bar
u_b, \bar u_c).
\end{aligned}
\end{equation*}
Thus,   $(\bar u_b, \bar u_c) $
is a minimizer of $(F^{reg,sc^-} + F^{fid})$
in $X$ and
\(\displaystyle F^{reg, sc^-}(\bar u_b,
\bar u_c) + F^{fid}(\bar u_b, \bar u_c)
= \lim_{n\to\infty} \big(
F^{reg}( u_b^n, u_c^n) + F^{fid}_{\epsi_n}(
u_b^n, u_c^n)\big)\). This concludes the
proof.
\end{proof}

\section{Proof of Theorem~\ref{relaxthm}}
\label{proofrelaxthm}

This section is devoted to the proof of Theorem~\ref{relaxthm} and is organized as follows. In Subsection~\ref{integrands}, we state some properties concerning the densities $Q_Tf$ and $K$ characterizing the functional in \eqref{lscFreg}. In Subsection~\ref{auxlem}, we collect several auxiliary results, which will be used to establish the integral representation for $\calF$ stated in Theorem~\ref{relaxthm}. A lower bound for the latter is proved in Subsection~\ref{lowerbound} and an upper bound in Subsection~\ref{upperbound}.

To simplify the notation, throughout the present section we will drop the indices $b$ and $c$, referring to brightness and chromaticity, respectively, and we replace $u_b$ by $u$, $u_c$ by $v$, $W^c_b$ by $W^c_u$, and  $W^c_c$ by $W^c_v$. Also, we recall that $\Omega\subset\RR^2$ is an open, bounded set with Lipschitz boundary $\partial\Omega$.

\subsection{Properties of $Q_Tf$ and $K$ }\label{integrands}

We start this subsection by  proving some properties of $\calQ_T f$ (see \eqref{defQTf}). Given $s\in S^2$ and $\eta\in\RR^{3\times 2}$ (respectively, $\eta\in\RR^3$),  set (cf. \cite{DFMTXCIX})
\begin{eqnarray}\label{defPDetall}
P_s \eta:=(\II_{3\times3}
- s\otimes s)\eta,
\end{eqnarray}
which defines a projection of $\RR^{3\times 2}$ onto $[T_s(S^2)]^2$ (respectively,
of  $\RR^3$
onto $T_s(S^2)$). Note that if
 $\eta\in T_s(S^2) \cup [T_s(S^2)]^2$, then $P_s\eta=\eta$ and
$|P_s\eta|\lii \sqrt2
|\eta|$.
Let $\tilde f:\RR\times \RR^3\times \RR^2\times \RR^{3\times 2}\to [0,\infty)$ be the function defined, for $(r,s,\xi,\eta)\in \RR\times \RR^3\times \RR^2\times \RR^{3\times 2}$, by
\begin{eqnarray}\label{deftildef}
\begin{aligned}
& \tilde f(r,s,\xi,\eta):=
\begin{cases}
f(\tilde r, \tilde s, \xi, P_{\tilde s}\eta)\,\phi(|s|) & \hbox{if } s\in \RR^3\backslash\{0\},\\
0 & \hbox{otherwise},
\end{cases}
\end{aligned}
\end{eqnarray}
where
\begin{eqnarray}\label{deftilders}
\begin{aligned}
&\tilde r := 
\begin{cases}
\alpha & \hbox{if } r\lii \alpha,\\
r & \hbox{if } \alpha\lii r\lii \beta,\\
\beta & \hbox{if } r\gii \beta,
\end{cases}
\qquad \tilde s := \frac{s}{|s|},
\end{aligned}
\end{eqnarray}
and $\phi\in C^\infty(\RR;[0,1])$ is a cut-off function such that
\begin{eqnarray}\label{defcutofftf}
\begin{aligned}
&\phi(t) = 1 \hbox{ if } t\gii 1,\qquad \phi(t) = 0 \hbox{ if } t\lii \frac{3}{4}.
\end{aligned}
\end{eqnarray}
Note that for all $r\in[\alpha,\beta]$, $s\in S^2$, $\xi\in\RR^2$, and $\eta\in [T_s(S^2)]^2$, we have that
\begin{eqnarray}\label{fandtildef}
\begin{aligned}
\tilde f(r,s,\xi,\eta) = f(r,s,\xi,\eta).
\end{aligned}
\end{eqnarray}
We observe also that $\tilde f_{|[\alpha,\beta]\times S^2\times \RR^2\times \RR^{3\times 2}}$ plays the role of the function introduced in  
\cite[(1.4)]{DFMTXCIX} and, as stated next, an analogous result to \cite[Prop.~2.2 (ii)]{DFMTXCIX} providing an alternative characterization of $\calQ_T f$ holds.

\begin{lemma}\label{onQQf}
For all $r\in[\alpha,\beta]$, $s\in S^2$, $\xi\in\RR^2$, and $\eta\in [T_s(S^2)]^2$, we have that
\begin{eqnarray}\label{qcxfqcxbarf}
\calQ_T f(r,s,\xi,\eta) = \calQ \tilde f(r,s,\xi,\eta),
\end{eqnarray}
where
\begin{eqnarray*}
\begin{aligned}
\calQ \tilde f(r,s,\xi,\eta):=\inf\bigg\{\int_Q \tilde f(r,s,\xi + \grad\ffi(y),\eta+\grad\psi(y))\,\dy\!:\,&\,\,\ffi\in W^{1,\infty}_0(Q), \,\,\psi\in W^{1,\infty}_0(Q;\RR^3)\bigg\}.\end{aligned}%
\end{eqnarray*}
\end{lemma}

\begin{proof} Fix $r\in[\alpha,\beta]$, $s\in S^2$, $\xi\in\RR^2$, and $\eta\in [T_s(S^2)]^2$.

Let $\ffi \in W^{1,\infty}_0(Q)$ and  $\psi\in W^{1,\infty}_0(Q;T_s(S^2))$ be given. Then, in particular,  $\psi\in W^{1,\infty}_0(Q;\RR^3)$ and $P_s\circ\grad\psi =\grad\psi$; hence
{\setlength\arraycolsep{0.5pt}
\begin{eqnarray*}
\int_Q f(r,s,\xi + \grad
\ffi,\eta+ \grad\psi)\,\dy\,
&&= \int_Q f(r,s,\xi +
\grad \ffi,P_s\circ(\eta+
\grad\psi))\,\dy\\
&&= \int_Q \tilde f(r,s,\xi
+ \grad \ffi,\eta+ \grad\psi)\,\dy
\gii \calQ\tilde f(r,s,\xi,\eta).
\end{eqnarray*}}%
Taking the infimum over all admissible $\psi$ and $\ffi$, we get $\calQ_T f(r,s,\xi,\eta) \gii \calQ \tilde f(r,s,\xi,\eta)$.

Conversely, let $\ffi \in W^{1,\infty}_0(Q)$ and  $\bar\psi\in W^{1,\infty}_0(Q;\RR^3)$ be given. Define $\psi:=P_s \circ\bar \psi$. Then, $\psi\in W^{1,\infty}_0(Q;T_s(S^2))$ and $\grad \psi = P_s \circ \grad \bar \psi$. Thus, 
{\setlength\arraycolsep{0.5pt}
\begin{eqnarray*}
\int_Q \tilde f(r,s,\xi
+ \grad \ffi,\eta+ \grad\bar\psi)
\,\dy\,
&&= \int_Q f(r,s,\xi +
\grad \ffi,P_s\circ(\eta+
\grad\bar\psi))\,\dy\\
&&= \int_Q  f(r,s,\xi
+ \grad \ffi,\eta+ \grad\psi)\,\dy
\gii \calQ_T f(r,s,\xi,\eta).
\end{eqnarray*}}%
Taking the infimum over all such  $\ffi$ and $\bar \psi$, we get $\calQ \tilde f(r,s,\xi,\eta) \gii \calQ_T f(r,s,\xi,\eta)$, which concludes the proof of \eqref{qcxfqcxbarf}.
\end{proof}

\begin{remark}\label{fnotTqcx}
Arguing exactly as at the end of Subsection~\ref{qqcxfct}, there does not exist $(r,s)\in [\alpha, \beta] \times S^2$ for which
{\setlength\arraycolsep{0.5pt}%
\begin{eqnarray*}
\begin{aligned}
(\xi,\eta) \in \RR^2 \times \RR^{3\times 2} \mapsto \tilde f (r,s, \xi, \eta)\in\RR^+
\end{aligned}
\end{eqnarray*}}%
is quasiconvex. Consequently, given $(r,s)\in [\alpha, \beta]
\times S^2$, 
{\setlength\arraycolsep{0.5pt}%
\begin{eqnarray*}
\begin{aligned}
(\xi,\eta) \in \RR^2 \times  [T_s(S^2)]^2
\mapsto  f (r,s, \xi, \eta)\in\RR^+
\end{aligned}
\end{eqnarray*}}%
is not tangential quasiconvex; that is, there exists $(\bar \xi,\bar \eta) \in \RR^2 \times  [T_s(S^2)]^2
$ such that $f(r,s,\bar\xi,\bar\eta) \not = \calQ_T f(r,s,\bar \xi,\bar \eta)$. In fact, if $(\xi,\eta) \in \RR^2 \times \RR^{3\times 2}$ and $(\ffi,\psi) \in W^{1,\infty}_0(Q)
 \times W^{1,\infty}_0(Q;\RR^3)$ are such that $\tilde f(r,s,\xi,\eta) > \int_Q \tilde f (r,s,\xi + \grad \ffi(y), \eta + \grad \psi(y))\,\dy$, 
with \((r,s) \in [\alpha,\beta] \times S^2\), then $(\bar\xi,\bar\eta):= (\xi, P_s \eta) \in \RR^2 \times\ [T_s(S^2)]^2$ and $(\bar\ffi,\bar\psi):=(\ffi,P_s \circ\psi) \in W^{1,\infty}_0(Q)
 \times W^{1,\infty}_0(Q;T_s(S^2))$ are such that
$ f(r,s,\bar\xi,\bar\eta) > \int_Q  f (r,s,\bar\xi
+ \grad \bar\ffi(y), \bar\eta + \grad \bar\psi(y))\,\dy$.
\end{remark}

We now establish some properties of  $\tilde f$, $\calQ\tilde f$, and $(\calQ\tilde f)^\infty$ that will be useful in what follows, where
\begin{eqnarray*}
\begin{aligned}
(\calQ\tilde f)^\infty(r,s,\xi,\eta):=&\limsup_{t\to+\infty} {\calQ \tilde f(r,s,t\xi,t\eta)\over t}
\end{aligned}
\end{eqnarray*}
for $(r,s,\xi,\eta)\in \RR\times \RR^3\times \RR^{2}\times \RR^{3\times 2}$. We first observe that since the application $s\in\RR^3\mapsto s\otimes s\in\RR^{3\times 3}$ is locally Lipschitz, there exists a positive constant, $c_\otimes$, such that for all $s,\bar s\in \overline{B(0,1)}$, it holds
\begin{eqnarray}\label{Lipotimes}
\begin{aligned}
& |s\otimes s - \bar
s\otimes\bar s|\lii c_\otimes|s-\bar
s|.
\end{aligned}
\end{eqnarray}

\begin{lemma}\label{onbarfQQbarf}
For all $(r,s,\xi,\eta)\in
[\alpha,\beta]\times
S^2\times \RR^2\times
\RR^{3\times 2}$, we
have that
{\setlength\arraycolsep{0.5pt}
\begin{eqnarray}
&& {1\over 2}|\xi| +{\alpha\over
2} |P_s\eta|  \lii \tilde
f(r,s,\xi,\eta) \lii
2|\xi| +
\sqrt2(1+\beta)|\eta|,\label{boundsbarf}
\\
&& {1\over 2}|\xi| +{\alpha\over
2} |P_s\eta| \lii \calQ\tilde
f(r,s,\xi,\eta) \lii
2|\xi| +
\sqrt2(1+\beta)|\eta|.\label{boundsQQbarf}
\end{eqnarray}}%
Moreover, there exists
a positive constant, $c$,
 depending only on $\alpha
$, $\beta$, $c_\otimes$,
and $Lip(g)$,  such that
for all $r,\, \bar r\in
[\alpha,\beta]$, $s,\,\bar
s\in S^2$, $\xi,\,\bar\xi\in\RR^2$,
 $\eta\in T_s(S^2)$,
and $\bar\eta\in T_{\bar
s}(S^2)$, one has
{\setlength\arraycolsep{0.5pt}
\begin{eqnarray}
&&\calQ\tilde
f( r, s,\xi,\eta)
\lii \calQ \tilde f(\bar
r,\bar s,\xi,\eta)   + c(|r-\bar
r| +|s-\bar s | )(|\xi|+|\eta|),
\label{QQbarfrbarrsbars}
\\
&& \big | \calQ\tilde
f(r,s,\xi,\eta) -
\calQ \tilde f(\bar
r,\bar s,\bar\xi,\bar\eta)\big|\nonumber\\
&&\hskip20mm\lii c|\eta-\bar
\eta| + c(|r-\bar
r| +|s-\bar s |+|\xi-\bar
\xi |  )(1+|\xi|+|\bar \xi|+|\eta|+|\bar
\eta|),\quad\qquad\label{contQtildef}
\\
&&(\calQ\tilde
f)^\infty( r, s,\xi,\eta)
\lii (\calQ\tilde f)^\infty
(\bar
r,\bar s,\xi,\eta) 
 +  c(|r-\bar
r| +|s-\bar s |)(|\xi|+|\bar\eta|).
\qquad\label{contrecQtildef}
\end{eqnarray}}%
\end{lemma}

\begin{proof} Fix $(r,s,\xi,\eta)\in
[\alpha,\beta]\times
S^2\times \RR^2\times
\RR^{3\times 2}$. We
have that
{\setlength\arraycolsep{0.5pt}
\begin{eqnarray*}
\tilde f(r,s,\xi,\eta)&&=|\xi|
+ g(|\xi|)|P_s\eta| + |rP_s\eta
+ s\otimes\xi|\lii |\xi|
+ 
\sqrt2|\eta| + 
\sqrt2|r||\eta|
+ |\xi|\lii 2|\xi| +
\sqrt2(1+\beta)|\eta|,
\end{eqnarray*}}%
where we used the fact
that $g\lii1$. On the
other hand,
\begin{eqnarray*}
\begin{aligned}
&\alpha|P_s\eta|\lii
|rP_s\eta + s\otimes\xi|
+ |s\otimes\xi|= |rP_s\eta
+ s\otimes\xi| + |\xi|,
\end{aligned}
\end{eqnarray*}
which, together with
the fact that $g>0$,
yields 
\begin{eqnarray*}
\begin{aligned}
&\tilde f (r,s,\xi,\eta)\gii
 |\xi| +  |rP_s\eta +
s\otimes\xi|\gii {\alpha\over
2} |P_s\eta| + {1\over
2}|\xi|.
\end{aligned}
\end{eqnarray*}
This concludes the
proof of \eqref{boundsbarf}.
Then, \eqref{boundsQQbarf}
follows from \eqref{boundsbarf}
taking into account that
the lower and upper bounds
for $\tilde f$ in \eqref{boundsbarf}
are quasiconvex functions
(with respect to the
pair $(\xi,\eta)$).

Next, we establish 
\eqref{QQbarfrbarrsbars}--\eqref{contrecQtildef}.
Let $r,\, \bar r\in [\alpha,\beta]$,
$s,\,\bar s\in S^2$,
$\xi,\,\bar\xi\in\RR^2$,
 $\tilde \eta \in\RR^{3\times
2}$, $\eta\in T_s(S^2)$,
and $\bar\eta\in T_{\bar
s}(S^2)$  be given.
 To simplify the notation,
\(c\) represents a positive constant that depends
only on $\alpha
$, $\beta$, $c_\otimes$,
and $Lip(g)$ and whose value may change from one
instance to another. We divide the proof
into three steps.

{\sl Step 1.} We show
that 
{\setlength\arraycolsep{0.5pt}
\begin{eqnarray}
&& \big | \calQ\tilde
f(r,s,\xi,P_s\tilde\eta) -
\calQ \tilde f( r, s,\xi,P_{\bar
s}\tilde\eta)\big| \lii c|s-\bar
s ||\tilde\eta|.\label{contQtildef0}
\end{eqnarray}}%

Fix $\epsi>0,$ and let
$\ffi_\epsi\in W^{1,\infty}_0(Q)$
and $\psi_\epsi\in W^{1,\infty}_0(Q;\RR^3)$
be such that
\begin{eqnarray*}
\calQ \tilde f(r,s,\xi,P_{\bar
s}\tilde\eta) + \epsi \gii
\int_Q \tilde f(r,s,\xi
+ \grad \ffi_\epsi,P_{\bar
s}\tilde\eta+ \grad\psi_\epsi)\,\dy.\end{eqnarray*}
Using the facts
that $0< g(\cdot)\lii1$,
 $0<\alpha\lii r \lii
\beta$, the linearity
of $P_s\cdot$, and the
estimate $|P_s\eta|\lii
\sqrt2|\eta|$,  
in this order, we get
{\setlength\arraycolsep{0.5pt}
\begin{eqnarray*}
&&\calQ \tilde f(r,s,\xi,P_s\tilde\eta)
- \calQ \tilde f(r,s,\xi,P_{\bar
s}\tilde\eta)\\
&&\quad \lii \int_Q \tilde
f(r,s,\xi + \grad \ffi_\epsi,P_s\tilde\eta+
\grad\psi_\epsi)\,\dy
- \int_Q \tilde f(r,s,\xi
+ \grad \ffi_\epsi,P_{\bar
s}\tilde\eta+ \grad\psi_\epsi)\,\dy
+ \epsi\\
&& \quad =  \int_Q  \Big[f(r,s,\xi
+ \grad \ffi_\epsi,P_s\circ(P_s\tilde\eta+
\grad\psi_\epsi)) - f(r,s,\xi
+ \grad \ffi_\epsi, P_s\circ(P_{\bar
s}\tilde\eta+ \grad\psi_\epsi)) \Big]\,\dy
+ \epsi\\
&& \quad \lii \int_Q (1+\beta)\big|P_s\circ(P_s\tilde\eta+
\grad\psi_\epsi)) - P_s\circ(P_{\bar
s}\tilde\eta+ \grad\psi_\epsi))\big|
\,\dy +\epsi = \int_Q (1+\beta)\big|P_s(P_s\tilde\eta
- P_{\bar s}\tilde\eta)\big|
\,\dy +\epsi\\
&& \quad \lii 
\sqrt2(1+\beta)
|P_s\tilde\eta - P_{\bar
s}\tilde\eta|
+\epsi = 
\sqrt2(1+\beta) |(s\otimes
s - \bar s\otimes \bar
s)\tilde\eta| +\epsi
\lii
\sqrt2 c_\otimes(1+\beta)|s-\bar
s||\tilde\eta| + \epsi.
\end{eqnarray*}}%
Letting $\epsi\to0^+ $ first
and then interchanging the roles of \(s\) and \(\bar
s\), we conclude  \eqref{contQtildef0}.


{\sl Step 2.} We establish
\eqref{QQbarfrbarrsbars}
and \eqref{contrecQtildef}.

By Step~1, applied to
$\tilde\eta := \eta =
P_s\eta$,
\begin{eqnarray}\label{QQbarfPsPbars00}
\begin{aligned}
&\calQ\tilde
f( r, s,\xi,\eta)
\lii \calQ \tilde f(
r, s,\xi,P_{\bar s}\eta)
+ c|s-\bar
s ||\eta|. \end{aligned}
\end{eqnarray}
Next, we estimate $\calQ \tilde f(
r, s,\xi,P_{\bar s}\eta)$
in terms of $\calQ \tilde f(\bar r,\bar s,\xi,P_{\bar s}\eta)$.
 Using \eqref{qcxfqcxbarf},
for all $\epsi>0$, we
can find $\ffi_\epsi\in
W^{1,\infty}_0(Q)$ and
$\psi_\epsi\in W^{1,\infty}_0(Q;T_{\bar
s}(S^2))$ such that
{\setlength\arraycolsep{0.5pt}
\begin{eqnarray*}
\calQ \tilde f(\bar
r,\bar s,\xi,P_{\bar
s}\eta) + \epsi &&= \calQ_T
f(\bar
r,\bar s,\xi,P_{\bar
s}\eta) + \epsi \gii
\int_Q f(\bar r,\bar
s,\xi + \grad \ffi_\epsi,P_{\bar
s}\eta+ \grad\psi_\epsi)\,\dy\\
&&=\int_Q \tilde f(\bar
r,\bar s,\xi + \grad
\ffi_\epsi,P_{\bar s}\eta+
\grad\psi_\epsi)\,\dy,
\end{eqnarray*}}%
where in the last equality
we used the fact that
$P_{\bar s}\circ (P_{\bar
s}\eta + \grad \psi_\epsi)=
P_{\bar s}\eta + \grad
\psi_\epsi$.
In particular, in view
of  \eqref{boundsbarf},
 \eqref{boundsQQbarf},
and the inequality $|P_{\bar s}\eta|\lii
\sqrt2|\eta|$, we get
\begin{eqnarray*}
\begin{aligned}
& 2|\xi| + \sqrt2(1+\beta)|\eta|
+ \epsi\gii \int_Q {1\over
2} |\xi + \grad \ffi_\epsi|
+{\alpha\over 2}|P_{\bar
s}\eta + \grad \psi_\epsi|
\,\dy.
\end{aligned}
\end{eqnarray*}
Thus,
\begin{eqnarray}\label{psifficompact}
\max\bigg\{\int_Q |\xi
+ \grad \ffi_\epsi|\,\dy,\int_Q
|P_{\bar s}\eta + \grad
\psi_\epsi|\,\dy\bigg\}\lii
c(|\xi| + |\eta| + \epsi).\end{eqnarray}
Moreover, using the fact
that $0< g(\cdot)\lii1$,
the estimates 
$|P_s\eta - P_{\bar s}\eta|\lii
c_\otimes|s-\bar s||\eta|$
and \eqref{psifficompact},
and the identity $P_{\bar
s}\circ(P_{\bar s}\eta+
\grad\psi_\epsi)=P_{\bar
s}\eta+ \grad\psi_\epsi$,
in this order, we obtain
{\setlength\arraycolsep{0.5pt}
\begin{eqnarray*}
&&\calQ \tilde f(r,s,\xi,P_{\bar
s}\eta) - \calQ\tilde
f(\bar r,\bar s,\xi,P_{\bar
s}\eta)\\
&&\quad\lii \int_Q \tilde
f(r,s,\xi + \grad \ffi_\epsi,P_{\bar
s}\eta+ \grad\psi_\epsi)\,\dy
- \int_Q \tilde f(\bar
r,\bar s,\xi + \grad
\ffi_\epsi,P_{\bar s}\eta+
\grad\psi_\epsi)\,\dy
+ \epsi\\
&&\quad=  \int_Q  f(r,s,\xi
+ \grad \ffi_\epsi,P_s\circ(P_{\bar
s}\eta+ \grad\psi_\epsi))
- f(\bar r,\bar s,\xi
+ \grad \ffi_\epsi, P_{\bar
s}\circ(P_{\bar s}\eta+
\grad\psi_\epsi))\,\dy
+ \epsi\\
&&\quad \lii \int_Q c_\otimes|s-\bar
s||P_{\bar s}\eta+ \grad\psi_\epsi|
+ \big|rP_s\circ(P_{\bar
s}\eta+ \grad\psi_\epsi))
 -\bar r P_{\bar s}\circ(P_{\bar
s}\eta+ \grad\psi_\epsi))
 + (s-\bar
s)\otimes(\xi+\grad\ffi_\epsi)\big|\,\dy
+\epsi\\
&&\quad \lii  c(c_\otimes
+ 1)|s-\bar s|(|\xi|
+ |\eta|+\epsi) + \int_Q
 \big|rP_s\circ(P_{\bar
s}\eta+ \grad\psi_\epsi))
\mp rP_{\bar s}\circ(P_{\bar
s}\eta+ \grad\psi_\epsi))
\\
&&\hskip115mm -\bar r P_{\bar
s}\circ(P_{\bar s}\eta+
\grad\psi_\epsi))\big|\,\dy
+\epsi\\
&&\quad \lii c|s-\bar s|(|\xi|
+ |\eta|+\epsi) +c|s-\bar s|\int_Q
|P_{\bar s}\eta+ \grad\psi_\epsi|\,\dy
+ |r-\bar
r|\int_Q|P_{\bar s}\eta+
\grad\psi_\epsi|\,\dy
+ \epsi\\
&&\quad \lii c(|r-\bar
r| +|s-\bar s | )(|\xi|+|\eta|
+ \epsi) +\epsi.
\end{eqnarray*}}%
Letting $\epsi\to0^+$, we conclude that
\begin{eqnarray}\label{QQbarfrbarrsbars1}
\begin{aligned}
\calQ \tilde f(r,s,\xi,P_{\bar
s}\eta) \lii \calQ \tilde f(\bar r,\bar s,\xi,P_{\bar
s}\eta) + c(|r-\bar r| +|s-\bar s | )(|\xi|+|\eta|).
\end{aligned}
\end{eqnarray}
Finally,  interchanging
the roles of $(r,s)$
and $(\bar r,\bar s)$, we conclude \eqref{QQbarfrbarrsbars}.

Property \eqref{contrecQtildef}
follows from \eqref{QQbarfrbarrsbars}
and the definition
of $(\calQ\tilde f)^\infty$.

{\sl Step 3.} We show
that \eqref{contQtildef}
holds true.

Arguing as in the previous
steps, we have
\begin{eqnarray*}
\begin{aligned}
& |Q\tilde f(\bar r,
\bar s,\xi, \eta) -  Q\tilde f(\bar r,
\bar s,\bar \xi, \eta)|
\lii  c |\xi - \bar
\xi| (1+|\xi| +|\bar\xi| + |\eta|)
\end{aligned}
\end{eqnarray*}
and
\begin{eqnarray*}
\begin{aligned}
& |Q\tilde f(\bar r,
\bar s, \bar \xi, \eta) - 
Q\tilde f(\bar r,
\bar s,\bar \xi, \bar \eta)|
\lii 
c|\eta
- \bar \eta|.
\end{aligned}
\end{eqnarray*}
Using these two estimates
together with \eqref{QQbarfrbarrsbars},
we obtain  \eqref{contQtildef}.
\end{proof}

Next, we show that for each $\epsi>0$, the function $H_\epsi$ defined by
\begin{eqnarray}\label{defofhepsi}
\begin{aligned}
&  H_\epsi(r,s,\xi,\eta):= \calQ \tilde f(r,s,\xi,\eta) + \epsi(|\xi| + |\eta|),\quad (r,s,\xi,\eta)\in \RR\times \RR^3\times \RR^2\times \RR^{3\times 2},
\end{aligned}%
\end{eqnarray}
satisfies hypotheses (H1)--(H4) of \cite{FMXCIII}. These integrands will play an important role in the proof of the lower bound for $\calF$.

\begin{proposition}\label{hepsiasFM} The function $H_\epsi:\RR\times \RR^3\times \RR^2\times \RR^{3\times 2}\to[0,+\infty)$ defined in \eqref{defofhepsi} satisfies the following conditions:
\begin{itemize}
\item [(i)] $H_\epsi$ is continuous;
\item [(ii)] $H_\epsi(r,s,\cdot,\cdot)$ is quasiconvex for all $(r,s)\in\RR\times\RR^3$;
\item [(iii)] there exists a positive constant, $C$,  depending only on $\beta$, such that for all $0<\epsi\lii 1$ and  $(r,s,\xi,\eta)\in
\RR\times \RR^3\times
\RR^2\times \RR^{3\times
2}$,
\begin{eqnarray*}
\begin{aligned}
\epsi(|\xi| + |\eta|)\lii H_\epsi(r,s,\xi,\eta)\lii C(|\xi| + |\eta|);
\end{aligned}
\end{eqnarray*}
\item[(iv)] for every compact set $\mathfrak{V}\subset\RR\times \RR^3$, there exists a positive constant, $C_\mathfrak{V}$,  depending only on $\mathfrak{V}$, such that for all $(r,s,\xi,\eta),\,
(\bar r,\bar s,\xi,\eta)\in
\mathfrak{V}\times \RR^2\times \RR^{3\times
2}$,
\begin{eqnarray*}
\begin{aligned}
|H_\epsi(r,s,\xi,\eta) - H_\epsi(\bar r, \bar s,\xi,\eta)|\lii C_\mathfrak{V}(|r-\bar r| + |s-\bar s|)(1+|\xi| + |\eta|).
\end{aligned}
\end{eqnarray*}
\end{itemize}

\end{proposition}

\begin{proof}
Conditions {\it (ii)} and  {\it (iii)} follow from the definition of $H_\epsi$ and from \eqref{boundsQQbarf}. To deduce {\it (i)} and  {\it (iv)} it suffices to observe that on the one hand, 
\begin{eqnarray*}
\begin{aligned}
\calQ \tilde f(r,s,\xi,\eta) =0
\end{aligned}
\end{eqnarray*}
for all $(r,s,\xi,\eta)\in \RR\times \RR^3\times \RR^2\times \RR^{3\times 2}$ such that $|s|\lii \frac{3}{4}$ in view of the definition of $\tilde f$ (see \eqref{deftildef} and \eqref{defcutofftf}). On the other hand, if $(r,s,\xi,\eta)\in \RR\times \RR^3\times \RR^2\times \RR^{3\times 2}$ is such that $|s|> \frac{1}{4}$, then
\begin{eqnarray*}
\begin{aligned}
\calQ \tilde f(r,s,\xi,\eta) = \phi(|s|)\calQ \tilde f(\tilde r, \tilde s, \xi,\eta) = \phi(|s|)\calQ \tilde f(\tilde r, \tilde s, \xi,P_{\tilde s}\eta).
\end{aligned}
\end{eqnarray*}
Hence, to conclude, it suffices to use the local
Lipschitz
continuity in $\RR\times
\big(\RR^3\backslash\{0\}\big)$
of $(\tilde r,\tilde
s)$ defined in \eqref{deftilders}
as function of $(r,s)$,
 \eqref{contQtildef},
 the
estimates $|P_{\tilde
s}\eta - P_{\tilde {\bar
s}}\eta|\lii
c_\otimes|\tilde s-\tilde
{\bar
s}||\eta|$ and $|P_{\tilde
s}\eta|\lii \sqrt2|\eta|$,
  the Lipschitz continuity of $\phi$, and  
\eqref{boundsQQbarf}.
\end{proof}

We now turn our attention to the {\sl jump integrand} $K$ defined by \eqref{defK}--\eqref{defP}.
We  prove that an
analogous result to \cite[Lemma~{2.15}]{FMXCIII} (see also \cite[Lemma~{4.1}]{ACELVII})
holds even though our
functions $f$ and $f^\infty$
do not  satisfy some of
the hypotheses assumed
in  \cite{ACELVII,FMXCIII},
such as quasiconvexity.  

\begin{lemma}\label{pptyrecfct}
Let $K:([\alpha,\beta]\times
S^2) \times ([\alpha,\beta]\times
S^2) \times S^1 \to [0,\infty)$
be the function defined
 by \eqref{defK}--\,\eqref{defP}.
Then,
\begin{itemize}
\item [(a)] There exists
a positive constant, $C,$
such that for all $a$,
$a'$, $b$, $b'\in [\alpha,\beta]\times
S^2$, and $\nu\in S^1$, it
holds
\begin{eqnarray*}
\begin{aligned}
& |K(a,b,\nu) - K(a',b',\nu)|
\lii C(|a-a'| + |b-b'|).
\end{aligned}
\end{eqnarray*}

\item [(b)] For all  $a$, $b \in [\alpha,\beta]\times
S^2$, the map $\nu\in S^1
\mapsto K(a,b,\nu)$ is
upper semicontinuous.

\item [(c)] $K$ is upper semicontinuous
in $([\alpha,\beta]\times
S^2) \times ([\alpha,\beta]\times
S^2) \times S^1$.

\item [(d)] There is
a positive constant, $C,$
such that for all $a$,
 $b\in [\alpha,\beta]\times
S^2$ and $\nu\in S^1$,
we have
\begin{eqnarray*}
\begin{aligned}
& K(a,b,\nu) \lii C|a-b|.
\end{aligned}
\end{eqnarray*}
\end{itemize}
\end{lemma}

\begin{proof} {\it (a)}
We start by proving that there is a positive constant,
\(C\), such that for all \(a=(r_1,s_1),
\, b =(r_2,s_2)\in [\alpha,\beta]\times
S^2\), one has
\begin{eqnarray}\label{dgbddbyeuc}
\begin{aligned}
& d_{[\alpha,
\beta] \times S^2}( a, b) \lii C|
a -  b|,
\end{aligned}
\end{eqnarray}
where
{\setlength\arraycolsep{0.5pt}%
\begin{eqnarray*}
&&d_{[\alpha,
\beta] \times S^2}( a, b):= \inf\bigg\{ \int_0^1 |\gamma'(t)|
\,\dt\!: \, \gamma \in W^{1,1}((0,1); [\alpha,
\beta] \times S^2), \, \gamma(0)= a, \, \gamma(1)
=  b \bigg\}
\end{eqnarray*}}%
is the geodesic
distance between $ a$
and $ b$ on  $[\alpha,\beta]\times
S^2$. 

We claim that to prove \eqref{dgbddbyeuc}
it suffices to prove that there is a positive constant,
\(C\), independent of \(s_1\) and \(s_2\), such that

\begin{equation}
\label{equivmetS}
\begin{aligned}
 d_{ S^2}( s_1, s_2) \lii C|
s_1 -  s_2|, 
\end{aligned}
\end{equation}
where  \(d_{S^2} ( s_1, s_2):= \inf\{ \int_0^1
|\gamma'(t)|
\,\dt\!: \, \gamma \in W^{1,1}((0,1);  S^2), \, \gamma(0)=
s_1, \,
\gamma(1)
=  s_2 \}\)
is the geodesic
distance between $ s_1$
and $ s_2$ on $S^2$. Indeed, let \(\gamma \in W^{1,1}((0,1);
 S^2)\) be such that \(\gamma(0)=
s_1\) and \(
\gamma(1)
=  s_2 \). Then, \(\bar \gamma:[0,1] \to [\alpha,\beta]\times
S^2\)  defined  by \(\bar\gamma(t)
:= ((1-t) r_1 + t r_2, \gamma(t))\), \(t\in [0,1]\),
belongs to \(W^{1,1}((0,1); [\alpha,
\beta] \times S^2)\) and satisfies \(\bar\gamma(0)=
a\) and \(
\bar\gamma(1)
=  b \). Moreover,
\begin{equation*}
\begin{aligned}
d_{[\alpha,
\beta] \times S^2}( a, b) \lii \int_0^1
|\bar \gamma'(t)|
\,\dt\lii|r_1 - r_2| +  \int_0^1
|\gamma'(t)|
\,\dt. 
\end{aligned}
\end{equation*}
Thus, taking the infimum over all  \(\gamma \in W^{1,1}((0,1);
 S^2)\) with \(\gamma(0)=
s_1\) and \(
\gamma(1)
=  s_2 \) in this estimate, \eqref{dgbddbyeuc} follows
from \eqref{equivmetS}.

To prove \eqref{equivmetS}, we show first that if \(|s_1
- s_2| \lii \frac{1}{2}\), then \(d_{S^2}(s_1,s_2)\lii
4|s_1 -s_2|\). To prove this implication, assume that
 \(|s_1 - s_2| \lii \frac{1}{2}\), and let 
 \[\gamma(t):=
\frac{(1-t)s_1 +
ts_2}{|(1-t)s_1 +
ts_2|}\]
 for \(t\in [0,1]\). Note that \(|(1-t)s_1 +
ts_2| = |s_1 - t(s_1-s_2)| \gii 1 - |s_1-s_2| \gii \frac{1}{2}\).
Moreover, \(\gamma\) is an admissible parameterization
for \(d_{S^2} ( s_1, s_2)\). Hence,
\[d_{S^2} ( s_1, s_2) \lii   \int_0^1
|\gamma'(t)|
\,\dt \lii   \int_0^1
\frac{2|s_1 - s_2|}{|(1-t)s_1 +
ts_2|} \,\dt \lii 4|s_1 -s_2|.\]
Therefore, if the claim \eqref{equivmetS} would fail,
then for all \(n\in\NN\), there would exist \(s_1^n,\,
s_2^n \in S^2\), \(s_1^n \not= s_2^n\), such that \(d_{S^2}
( s_1^n, s_2^n) >
n |s_1^n - s_2^n|\). Then, because
\(d_{S^2}
( s_1^n, s_2^n)\lii \pi \), we would have \( |s_1^n -
s_2^n| \lii \frac{1}{2}\) for all \(n\in\NN\) sufficiently
large. In turn, by the implication proved above, for
all
such \(n\in\NN\), we would
also have to have \(d_{S^2}
( s_1^n, s_2^n) \lii
4 |s_1^n - s_2^n|\). We are thus led to a contradiction.
Hence, \eqref{equivmetS} holds, and so does \eqref{dgbddbyeuc}.

Note that the infimum defining  \(d_{[\alpha,
\beta] \times S^2}( a, b)\) does not
change if instead of the interval \([0,1]\) we
consider any interval \([t_1, t_2] \subset \RR\)
with \(t_1< t_2\) as the domain of the parameterizations
$\gamma$.
Fix $\epsi>0$, and let
$\gamma_1$, $\gamma_2\in
W^{1,1} ((\frac{1}{4},
\frac{1}{2}); [\alpha,\beta]\times
S^2)$ be such that
\begin{eqnarray}\label{pathonman}
\begin{aligned}
& \gamma_1\Big(\frac{1}{4}\Big)
= b,\quad \gamma_1\Big(\frac{1}{2}\Big)
= b',\quad \int_{\frac{1}{4}}^{\frac{1}{2}}
|\gamma_1'(t)|\,\dt -\epsi
\lii d_{[\alpha,
\beta] \times S^2}(b,b')\lii C|b-b'|,\\
& \gamma_2\Big(\frac{1}{4}\Big)
= a,\quad \gamma_2\Big(\frac{1}{2}\Big)
= a',\quad \int_{\frac{1}{4}}^{\frac{1}{2}}
|\gamma_2'(t)|\,\dt -\epsi
\lii d_{[\alpha,
\beta] \times S^2}(a,a')\lii C|a-a'|.
\end{aligned}
\end{eqnarray}
Let $\vartheta = (\ffi,\psi)\in \calP(a,b,\nu)$,
and define $\vartheta^*
= (\ffi^*,\psi^*)\in \calP(a',b',\nu)$
by setting,
for $y\in Q_\nu$,
\begin{eqnarray*}
\begin{aligned}
& \vartheta^*(y):= 
       \begin{cases}
       \gamma_1(y\cdot
       \nu) & \hbox{if
       } \frac{1}{4}
       < y\cdot \nu <
       \frac{1}{2},\\
       \vartheta (2y)
       & \hbox{if } |y\cdot
       \nu| < \frac{1}{4},\\
       \gamma_2(-y\cdot
       \nu) & \hbox{if
       } -\frac{1}{2}
       < y\cdot \nu <
       -\frac{1}{4}.
       \end{cases}
\end{aligned}
\end{eqnarray*}
Denoting by $\nu_1\in S^1$ a fixed vector such that
 $\{\nu_1,\nu\}$ is an orthonormal basis of $\RR^2$,
 we have that
{\setlength\arraycolsep{0.5pt}
\begin{eqnarray*}
K(a',b',\nu) && \lii \int_{Q_\nu}
f^\infty (\vartheta^*(y),
\grad \vartheta^*(y))\,\dy\\
&& = \int_{|y\cdot \nu_1|
< \frac{1}{2}} \int_{|y\cdot
\nu|
< \frac{1}{4}} f^\infty
(\vartheta(2y),
2\grad \vartheta(2y))\,\dy\\
&&\qquad + \int_{|y\cdot
\nu_1|
< \frac{1}{2}} \int_{\frac{1}{4}
       < y\cdot \nu <
       \frac{1}{2}} f^\infty
(\gamma_1(y\cdot\nu),
\gamma_1'(y\cdot\nu)\otimes\nu)\,\dy\\
&&\qquad + \int_{|y\cdot
\nu_1|
< \frac{1}{2}} \int_{-\frac{1}{2}
       < y\cdot \nu <
       -\frac{1}{4}}
f^\infty
(\gamma_2(-y\cdot\nu),
-\gamma_2'(-y\cdot\nu)\otimes\nu)
\,\dy. 
\end{eqnarray*}}%
Hence, using \eqref{upborecfct}, 
\eqref{pathonman}, the 1-homogeneity of
$f^\infty(r,s,\cdot,\cdot)$,
and the 1-periodicity of
$\vartheta$ in the $\nu_1$-direction, we have
{\setlength\arraycolsep{0.5pt}
\begin{eqnarray*}
K(a',b',\nu) && \lii 
\frac{1}{2} \int_{|y\cdot
\nu_1|
< 1} \int_{|y\cdot
\nu|
< \frac{1}{2}} f^\infty
(\vartheta(y),
\grad \vartheta(y))\,\dy+
   (3+\beta)\big(
C|b-b'|+ C|a-a'| +2\epsi
\big)\\
&& = \int_{Q_\nu}  f^\infty
(\vartheta(y),
\grad \vartheta(y))\,\dy+
  (3+\beta)\big(
C|b-b'|+ C|a-a'| +2\epsi
\big).
\end{eqnarray*}}%
Letting $\epsi\to0^+$
and taking the infimum
over $\vartheta = (\ffi,\psi)\in
\calP(a,b,\nu),$ we deduce
that
\begin{eqnarray*}
\begin{aligned}
& K(a,b,\nu) \lii K(a',b',\nu)
+ C   (3+\beta)\big(
|b-b'|+|a-a'| \big).
\end{aligned}
\end{eqnarray*}
Interchanging the roles
between $(a,b)$ and $(a',b'),$
 assertion {\it (a)}
 follows.

{\it (b)} Let $\nu_n$,
$\nu\in S^1$, $n\in\NN$,
be such that $\lim_{n\to\infty}
|\nu_n - \nu| =0$.
Fix $\epsi>0,$
and let $\vartheta
= (\ffi,\psi)
\in \calP(a,b,\nu)$ be
such that 
\begin{eqnarray*}
\begin{aligned}
& \int_{Q_\nu}
f^\infty(\vartheta(y),\grad\vartheta
(y))\,\dy \lii  K(a,b,\nu) + \epsi.
\end{aligned}
\end{eqnarray*}
Let $R$ be a rotation  such
that $R e_2 = \nu$, and
choose rotations $R_n$
with $\lim_{n\to\infty}
|R_n - R| =0$ and $R_n
e_2 = \nu_n$. Define $\vartheta_n\in
\calP(a,b,\nu_n)$ by
setting 
\begin{eqnarray*}
\begin{aligned}
& \vartheta_n(y):= \vartheta(R
R_n^T y) \hbox{ for } y\in Q_{\nu_n}.
\end{aligned}
\end{eqnarray*}
Then,
{\setlength\arraycolsep{0.5pt}
\begin{eqnarray*}
 K(a,b,\nu_n) && \lii
 \int_{Q_{\nu_n}} f^\infty
 (\vartheta_n(y),\grad\vartheta_n(y))
 \,\dy
 = \int _{Q_{\nu_n}}
f^\infty
 (\vartheta(R
R_n^T y),\grad\vartheta(R
R_n^T y)R
R_n^T)
 \,\dy\\
&&= \int _{Q_{\nu}}
f^\infty
 (\vartheta( z),\grad\vartheta(
z)R
R_n^T)
 \,\d z.
\end{eqnarray*}}%
Since $f^\infty$ is upper
semicontinuous, in view
of \eqref{upborecfct}
and Fatou's Lemma, we
obtain
\begin{eqnarray*}
\begin{aligned}
 \limsup_{n\to\infty}
K(a,b,\nu_n) &\lii \int _{Q_{\nu}}
 \limsup_{n\to\infty}
 f^\infty
 (\vartheta( z),\grad\vartheta(
z)R
R_n^T)
 \,\d z \lii \int_{Q_\nu}
f^\infty(\vartheta(z),\grad\vartheta
(z))\,\d z \lii  K(a,b,\nu)
+ \epsi.
\end{aligned}
\end{eqnarray*}
To conclude, let $\epsi\to0^+$.

{\it (c)} It follows
from  {\it (a)}
 and {\it (b)}.

{\it (d)} Fix $\epsi>0$,
and let $\gamma \in
W^{1,1}((-\frac{1}{2},
\frac{1}{2}); [\alpha,\beta]\times
S^2)$ be such that
\begin{eqnarray}\label{almostd}
\begin{aligned}
& \gamma\Big(-\frac{1}{2}\Big)
= a,\quad \gamma\Big(\frac{1}{2}\Big)
= b,\quad \int_{-\frac{1}{2}}^{\frac{1}{2}}
|\gamma'(t)|\,\dt -\epsi
\lii d_{[\alpha,
\beta] \times S^2}(a,b)\lii C|a-b|,
\end{aligned}
\end{eqnarray}
where $C$ is the constant
in \eqref{dgbddbyeuc},
and define  $\vartheta(y):=
\gamma(y\cdot\nu)$ for $y\in
Q_\nu$. Then,
$\vartheta \in \calP(a,b,\nu)$
and, arguing as in {\it (a)}, 
\begin{eqnarray*}
\begin{aligned}
 K(a,b,\nu) & \lii \int_{Q_\nu}
 f^\infty(\vartheta( y),\grad\vartheta(
y))\, \dy =  \int_{Q_\nu}
f^\infty
(\gamma(y\cdot\nu),
\gamma'(y\cdot\nu)\otimes\nu)\,\dy\lii   (3+\beta)\int_{-\frac{1}{2}}^{\frac{1}{2}}
|\gamma'(t)|\,\dt. 
\end{aligned}
\end{eqnarray*}
This estimate, together with
\eqref{almostd}, yields
the conclusion.
\end{proof}

\subsection{Auxiliary Lemmas}\label{auxlem}

As in \cite{ACELVII}, given $y\in B(0,\frac{1}{2}) \subset\RR^3,$ we define the projection function $\pi_y\!: \overline{B(0,1)}\backslash\{y\}\to S^2$  by setting
{\setlength\arraycolsep{0.5pt}
\begin{eqnarray*}
&&\pi_y(s):= y + {-y\cdot
(s-y) + \sqrt{(y\cdot
(s-y))^2 + |s-y|^2(1-|y|^2)}\over
|s-y|^2}(s-y),
\end{eqnarray*}}%
which projects each $s\in \overline{B(0,1)}\backslash\{y\}$ onto $S^2$ along the direction $s-y$. We have that
\begin{eqnarray}\label{projonSisId}
{\pi_y}_{|S^2} = {\rm Id}_{S^2}, \quad \grad \pi_y
(s) = \II_{3 \times 3} + (s-y) \otimes \frac{1}{|s- y |^2}
\Big( \frac{|y|^2 - 1}{s \cdot y -1} - 2 \Big)
s \hbox{ if  } s\in S^2. \end{eqnarray}
Note that by \eqref{projonSisId}, if  \(s\in S^2\)
and \(w \in T_s(S^2) \), then 
{\setlength\arraycolsep{0.5pt}%
\begin{eqnarray}\label{gradprojonS}
&&\grad \pi_y (s) w  = w.
\end{eqnarray}}%
Furthermore,
there exists a positive constant, \(\bar C,\) independent
of $y\in B(0,\frac{1}{2})$, such that for all $s\in \overline{B(0,1)}\backslash\{y\}$,
we have
\begin{eqnarray}\label{estGradHess}
|\grad \pi_y(s)|\lii {\bar C\over |s-y|}, \qquad |\grad^2 \pi_y(s)|\lii {\bar C\over |s-y|^2}.
\end{eqnarray}
Consequently, there exists a positive constant, \(\overline
C,\)
independent
of $y\in B(0,\frac{1}{2})$, such that for all $s_1,\,s_2\in
\big\{s\in \overline{B(0,1)}\!:\,
\dist(s,S^2)\lii {1\over 4}\big\}$,
we have
\begin{eqnarray}\label{uniflipballtf}
|\pi_y(s_1)- \pi_y(s_2)|\lii \overline C|s_1-s_2|,\qquad |\grad\pi_y(s_1)- \grad\pi_y(s_2)|\lii \overline C|s_1-s_2|.
\end{eqnarray}
The following result holds (see also 
\cite[Lem.~5.2 and Lem.~6.1]{ACELVII}).

\begin{lemma}\label{seqonman}
Let $A\in\calA(\Omega)$, let  $v\in W^{1,1}(A;\overline{B(0,1)})\cap C^\infty(A;\RR^3)$, and  let $A'$ be an open subset of $A$. 
Then, there exists $y\in B(0,\frac{1}{2})$, depending on $v$ and $A'$, such that $\pi_y\circ v\in W^{1,1}(A';S^2) \cap C^\infty(A;S^2) $ and
\begin{eqnarray}\label{approxonMan}
\begin{aligned}
&\int_{A'} |\grad (\pi_y\circ v)|\,\dx\lii  C_\star\int_{A'}|\grad v|\,\dx,
\end{aligned}
\end{eqnarray}
where $C_\star$ is a positive constant independent of $A$, $A'$, $v$, and $y$.
\end{lemma}

\begin{proof}
Since $v:A\subset\RR^2\to\RR^3$ is smooth,  {\sl Mini-Sard}'s Theorem (see \cite{GuPo74})  yields $\calL^3(v(A))=0$. In particular, setting $G:=\{y\in B(0,\frac{1}{2})\!: \,\hbox{there exists } x\in A \hbox{ such that } v(x)=y\}$, then  $\calL^3(G)=0$.  Moreover, for all $y\in B(0,\frac{1}{2})\backslash G$, the function $\pi_y\circ v$ belongs to $C^\infty (A;S^2)$ and, by Fubini's Theorem and the first estimate in \eqref{estGradHess},
{\setlength\arraycolsep{0.5pt}
\begin{eqnarray}\label{gradproj1}
\int_{B(0,\frac{1}{2})}\int_{A'}
|\grad (\pi_y(v(x)))|\,\dx\,\dy
&&= \int_{B(0,\frac{1}{2})}\int_{A'}
|(\grad \pi_y)(v(x))\grad
v(x)|\,\dx\,\dy \nonumber \\
&&\lii \bar C \int_{A'}
\bigg(|\grad v(x)|\int_{B(0,\frac{1}{2})}
\frac{1}{| v(x) - y|}\,\dy\bigg)\,\dx.
\end{eqnarray}}%
For fixed $x\in {A'}$, use the change of variables $z= y - v(x)$ to get
\begin{eqnarray}\label{gradproj2}
\begin{aligned}
\int_{B(0,\frac{1}{2})} \frac{1}{| v(x) - y|}\,\dy&= \int_{B(-v(x),\frac{1}{2})} \frac{1}{| z|}\,\d z \lii\int_{B(0,\frac{3}{2})} \frac{1}{|z|}\,\d z =:c_1\in\RR,
\end{aligned}
\end{eqnarray}
where we used the fact that $\Vert v\Vert_{L^\infty(A)}\lii 1$. From \eqref{gradproj1} and \eqref{gradproj2}, we conclude that
\begin{eqnarray*}
\begin{aligned}
\int_{B(0,\frac{1}{2})}\int_{A'} |\grad (\pi_y(v(x)))|\,\dx\,\dy\lii c_1\bar C \int_{A'} |\grad v(x)|\,\dx.
\end{aligned}
\end{eqnarray*}
Consequently, we can find $y\in B(0,\frac{1}{2})\backslash G$ such that \eqref{approxonMan} holds with $C_\star:=c_1\bar C/\calL^3(B(0,\frac{1}{2})) $. Finally, we observe that for such $y$, we have  $\pi_y\circ v\in W^{1,1}(A';S^2)  \cap C^\infty(A;S^2)$.
\end{proof}

\begin{lemma}\label{man+tr}
Let $A\in\calA_\infty(\Omega)$ and  $w = (u,v)
\in BV(A;[\alpha,\beta]\times S^2)$.
Then, there exists a sequence $\{\bar w_n\}_{n\in\NN}\subset W^{1,1} (A; [\alpha,\beta] \times S^2)\cap C^\infty (A;\RR \times \RR^3)$ such that $\bar w_n = w$ on $\partial A$ for all $n\in\NN$, $\lim_{n\to\infty}\Vert \bar w_n - w\Vert_{L^1(A;\RR \times \RR^3)}=0$, and $\limsup_{n\to\infty}\int_A|\grad\bar w_n(x)|\,\dx\lii \widetilde C |Dw|(A)$, where $\widetilde C$ is a positive constant only depending on $\bar C$,
\(\overline C\), and $C_\star$. 
\end{lemma}

\begin{proof}
Because \(w = (u,v)\)  takes values on $[\alpha,\beta] \times S^2$,   its mollification (see \eqref{defmollfct}) takes values on $[\alpha,\beta] \times \overline{B(0,1)}$.
Thus, by \cite[Thm.~{2.17}, Rmk.~{1.18}]{GiLXXXIV}  (see also \cite[Rmk.~{2.12}]{GiLXXXIV}), there exists a sequence $\{w_n\}_{n\in\NN}\subset W^{1,1} (A; [\alpha,\beta]
\times \overline{B(0,1)})\cap C^\infty (A;\RR \times \RR^3)$ such that
\begin{eqnarray}\label{man+tr1}
\begin{aligned}
& w_n = w \hbox{ on $\partial A$ for all $n\in\NN$}, \quad w_n\weaklystar w \hbox{ weakly-$\star$ in $BV(A;\RR \times \RR^3)$ as  $n\to\infty$,}\\
&\lim_{n\to\infty} \int_A|\grad w_n(x)|\,\dx = |Dw|(A).
\end{aligned}
\end{eqnarray}
We write $w_n(\cdot)=(u_n(\cdot),v_n(\cdot))\in [\alpha,\beta]\times   \overline{B(0,1)}$ $\calL^2$-\aev\ in $A$.
Fix $\delta_0\in (0,\frac{1}{4})$, and set $A_n:=\{x\in A\!:\, \dist(v_n(x), S^2)>\delta_0\}$. By Lemma~\ref{seqonman} applied to $A$, $v_n$, and $A_n$, we can find $y_n\in B(0,\frac{1}{2})$ such that $\pi_{y_n}\circ v_n\in   C^\infty(A;S^2) $ and
\begin{eqnarray*}
\begin{aligned}
&\int_{A_n} |\grad (\pi_{y_n}\circ v_n)|\,\dx\lii  C_\star\int_{A_n}|\grad v_n|\,\dx.
\end{aligned}
\end{eqnarray*}
Using the first estimate in \eqref{estGradHess}, we obtain 
\begin{eqnarray*}
\begin{aligned}
&\int_{A\backslash A_n} |\grad (\pi_{y_n}\circ v_n)|\,\dx\lii  \int_{A\backslash A_n}\frac{\bar C}{|v_n - y_n|}|\grad v_n|\,\dx \lii 4\bar C \int_{A\backslash A_n}|\grad v_n|\,\dx,
\end{aligned}
\end{eqnarray*}
and so
\begin{eqnarray}\label{man+tr2}
\begin{aligned}
&\int_{A} |\grad (\pi_{y_n}\circ v_n)|\,\dx\lii \max\{C_\star, 4\bar C \}\int_{A}|\grad v_n|\,\dx.
\end{aligned}
\end{eqnarray}
Setting $\bar w_n:=(u_n, \pi_{y_n}\circ v_n)$, we have  that $\{\bar w_n\}_{n\in\NN}\subset W^{1,1} (A; [\alpha,\beta] \times S^2)\cap C^\infty (A;\RR \times \RR^3)$ is a bounded sequence in  $W^{1,1} (A; [\alpha,\beta] \times S^2)$. Moreover, using \eqref{projonSisId}, the first estimate in  \eqref{uniflipballtf}, and the fact that $v(\cdot)\in S^2$ (so that \(\pi_{y_n}\circ
 v= v\)) for $\calL^2$-\aev\ in $A$,
and the estimate \(|\pi_{y_n}\circ  v_n -
v| \lii 2 \lii 2/\delta_0 | v_n - v| \) in \(A_n\),
we obtain
{\setlength\arraycolsep{0.5pt}
\begin{eqnarray}\label{man+tr3}
\int_A|\pi_{y_n}\circ
v_n - v|\,\dx && = \int_{A_n}
|\pi_{y_n}\circ v_n -
v|\,\dx  + \int_{A\backslash
A_n} |\pi_{y_n}\circ
v_n - \pi_{y_n}\circ
v|\,\dx \nonumber\\
&&\lii 2\calL^2(A_n)
+ \overline C\int_{A\backslash
A_n} | v_n - v|\,\dx
\lii \frac{2}{\delta_0}
\int_{A_n} | v_n - v|\,\dx
+ \overline C\int_{A\backslash
A_n} | v_n - v|\,\dx
\nonumber\\
&& \lii \max\Big\{\frac{2}{\delta_0},
\overline C\Big\} \int_{A}
| v_n - v|\,\dx.
\end{eqnarray}}%
In view of   \eqref{man+tr1}--\eqref{man+tr3}, we conclude that $\{\bar w_n\}_{n\in\NN}$ satisfies the requirements stated in Lemma~\ref{man+tr}.
\end{proof}

\begin{remark}\label{man+trhalflip}
If $A\in\calA(\Omega)$ is of the form $A=A_1\sminus \overline A_0$, where $A_1\in\calA(\Omega)$,  $A_0\in\calA_\infty(\Omega)$, and $A_0\subset\subset A_1$, then a simple adaptation of the proof above yields the existence of a sequence as in Lemma~\ref{man+tr} with the trace condition only holding on $\partial A_0$; that is, the trace condition becomes ``$\bar w_n = w$ on $\partial A_0$'.
\end{remark}

The next lemma is a simplified version of a result proved in \cite{BZLXXXVIII} (see also \cite[Thm.~2.2]{ACELVII}), which will be useful in the subsequent slicing result.

\begin{lemma}\label{denssmoothS}
Let $\Omega$ be an open subset of $\RR^2$. The space $W^{1,1}(\Omega;S^2) \cap C^\infty(\Omega;S^2)$ is dense in $W^{1,1}(\Omega;S^2)$ with respect to the $W^{1,1}(\Omega;\RR^3)$-norm.
\end{lemma}

\begin{lemma}\label{slicingMan}
Let $\Omega\subset\RR^2$ be an open, bounded set and  $\mathfrak{h}:[\alpha,\beta]\times S^2\times\RR^2\times\RR^{3\times 2}\to[0,+\infty)$ be an upper semicontinuous function satisfying, for some  $C>0$ and
for all \((r,s,\xi,\eta)\in [\alpha,\beta]\times
S^2\times\RR^2\times\RR^{3\times 2}\),
\begin{eqnarray}\label{suffbounds}
\begin{aligned}
0\lii \mathfrak{h}(r,s,\xi,\eta)\lii C(1+ |\xi|+|\eta|).
\end{aligned}
\end{eqnarray}
Let $A\in \calA_\infty(\Omega)$, $w=(u,v)\in BV(A;[\alpha,\beta]\times S^2)$, and $w_n=(u_n,v_n)\in W^{1,1}(A;[\alpha,\beta]\times S^2)$, $n\in\NN$, be such that $\lim_{n\to\infty}\Vert w_n- w\Vert_{L^1(A;\RR
\times \RR^3)} =0$. Then, for all  $n\in\NN$, there exists $\tilde w_n=(\tilde u_n,\tilde v_n)\in W^{1,1}(A;[\alpha,\beta]\times S^2)$ satisfying 
{\setlength\arraycolsep{0.5pt}
\begin{eqnarray}
&&\lim_{n\to\infty}\Vert\tilde w_n- w\Vert_{L^1(A;\RR
\times \RR^3)} =0, \quad \tilde w_n= w \hbox{ on } \partial A, \label{L1andTr}\\
&& \limsup_{n\to\infty}\int_A \mathfrak{h}(\tilde u_n, \tilde v_n, \grad\tilde u_n, \grad\tilde v_n)\,\dx \lii \liminf_{n\to\infty}\int_A \mathfrak{h}( u_n,  v_n, \grad u_n, \grad v_n)\,\dx. \label{notincenergy}
\end{eqnarray}}%

\begin{proof}
In  view of the hypotheses on $\mathfrak{h}$, by Fatou's Lemma, Lemma~\ref{denssmoothS}, and using a diagonalization argument, we may assume  that the component $v_n$ of $w_n$ belongs to $W^{1,1}(A;S^2) \cap C^\infty(A;S^2)$.

Extracting a subsequence, if needed, we may assume without
loss of generality that the limit
inferior on the right-hand
side of \eqref{notincenergy} is a limit. By Lemma~\ref{man+tr}, there exists a sequence $\{\bar w_n\}_{n\in\NN}=\{(\bar u_n,\bar v_n)\}_{n\in\NN}
\subset W^{1,1} (A; [\alpha,\beta] \times S^2)\cap C^\infty (A;\RR \times \RR^3)$ such that  
\begin{eqnarray}\label{slicing}
\begin{aligned}
&\bar w_n = w \hbox{ on $\partial A$ for all $n\in\NN$}, \quad \bar w_n\weaklystar w \hbox{ weakly-$\star$ in $BV(A;\RR \times \RR^3)$ as  $n\to\infty$}.
\end{aligned}
\end{eqnarray}
For each $n\in\NN$, let $\displaystyle \kappa_n:=\frac{a_n}{b_n}$, where %
{\setlength\arraycolsep{0.5pt}
\begin{eqnarray*}
&& a_n:=\sqrt{\Vert u_n-\bar u_n\Vert_{L^1(A)} + \Vert v_n-\bar v_n\Vert_{L^1(A;\RR^3)}}\,,\\
&& b_n:=n\bigg[\!\Big| 1 +\Vert \grad u_n\Vert_{L^1(A;\RR^2)} +\Vert \grad \bar u_n\Vert_{L^1(A;\RR^2)} + \Vert \grad v_n\Vert_{L^1(A;\RR^{3\times 2})} + \Vert \grad\bar v_n\Vert_{L^1(A;\RR^{3\times 2})}\Big|\!\bigg], 
\end{eqnarray*}}%
with $\big[\!|b|\!\big]$ denoting the integer part of $b$. Clearly, \(\kappa_n\to0^+\) as \(n\to\infty\). For $i\in\{1,\cdots, b_n\}$, define
{\setlength\arraycolsep{0.5pt}
\begin{eqnarray*}
&& A_{n,0}:=\{ x\in A\!:\, \dist(x,\partial A)>a_n\},\quad A_{n,i}:=\{ x\in A\!:\, \dist(x,\partial A)>a_n - i\kappa_n\}. 
\end{eqnarray*}}%
We have that $A_{n,0}\subset A_{n,1}\subset\cdots\subset A_{n,b_n}$ and, for all $n$ large enough, $A_{n,0}\not=
\emptyset$
since \(a_n \to 0\) as \(n\to\infty\). Fix any such $n$, and let $\ffi_i\in C^\infty_c(\RR^2;[0,1])$ be a cut-off function such that $\ffi_i =  1$ in $A_{n,i-1}$, $\ffi_i=0$ in  $\RR^2\backslash A_{n,i}$, and $\Vert \grad \ffi_i\Vert_\infty\lii  \frac{c}{\kappa_n}$, being $c$ a positive constant independent of $i$ and $n$, and set
\begin{eqnarray*}
\begin{aligned}
w_n^{i}=(u_n^{i}, v_n^{i}):= \ffi_i\, (u_n,v_n) + (1-\ffi_i)\, (\bar u_n,\bar v_n) = \ffi_i\, w_n + (1-\ffi_i)\, \bar w_n.
\end{aligned}
\end{eqnarray*}
We have that $w_n^{i}\in W^{1,1}(A;[\alpha,\beta] \times \overline{B(0,1)})$, $v_n^{i}\in W^{1,1}(A;\overline{B(0,1)})\cap C^\infty(A;\RR^3)$, 
{\setlength\arraycolsep{0.5pt}
\begin{eqnarray}
&& \hbox{$w_n^{i} = w_n$ in $A_{n,i-1}$,\quad $w_n^{i} =\bar w_n$ on $A\backslash A_{n,i}$} , \quad\lim_{n\to\infty}\Vert w_n^{i}- w\Vert_{L^1(A;\RR \times \RR^3)} =0, \label{onlayers} 
\end{eqnarray}}%
and, since $\grad w_n^{i} = \ffi_i\grad w_n +(1-\ffi_i)\grad\bar w_n + (w_n-\bar w_n)\otimes\grad \ffi_i$,
\begin{eqnarray}\label{slicing1}
\begin{aligned}
& \int_{A_{n,i}\backslash \overline A_{n,i-1}}|\grad w_n^{i}|\,\dx \lii \int_{A_{n,i}\backslash \overline A_{n,i-1}} \Big(|\grad w_n| + |\grad\bar w_n| + \frac{c}{\kappa_n}|w_n -\bar w_n|\Big)\,\dx.
\end{aligned}
\end{eqnarray}
We now apply Lemma~\ref{seqonman} to $A$, $v_n^i$, and $A_{n,i}\backslash\overline A_{n,i-1}$, to
find a point $y_n^i\in B(0,\frac{1}{2})$ such that  $\pi_{y_n^i}\circ v_n^i\in W^{1,1}(A_{n,i}\backslash\overline A_{n,i-1};S^2) \cap C^\infty(A;S^2)$ and
\begin{eqnarray}\label{slicing2}
\begin{aligned}
&\int_{A_{n,i}\backslash\overline A_{n,i-1}} |\grad (\pi_{y_n^i}\circ v_n^i)|\,\dx\lii  C_\star\int_{A_{n,i}\backslash\overline A_{n,i-1}}|\grad v_n^i|\,\dx.
\end{aligned}
\end{eqnarray}
In view of \eqref{onlayers} and \eqref{projonSisId}, since $v_n$ and $\bar v_n$ take values in $S^2$ $\calL^2$-\aev\ in $A$, we get
\begin{eqnarray}\label{slicing3}
\begin{aligned}
&\int_{A_{n,i-1}}|\grad (\pi_{y_n^i}\circ v_n^i)|\,\dx = \int_{A_{n,i-1}}|\grad (\pi_{y_n^i}\circ v_n)|\,\dx = \int_{A_{n,i-1}}|\grad v_n|\,\dx,\\
&\int_{A\backslash A_{n,i}}|\grad (\pi_{y_n^i}\circ v_n^i)|\,\dx = \int_{A\backslash A_{n,i}}|\grad (\pi_{y_n^i}\circ\bar  v_n)|\,\dx = \int_{A\backslash A_{n,i}}|\grad\bar v_n|\,\dx.
\end{aligned}
\end{eqnarray}
Thus, \eqref{slicing2} and \eqref{slicing3} yield $\tilde w_n^i=(\tilde u_n^i,\tilde v_n^i):= (u_n^i,\pi_{y_n^i}\circ v_n^i)\in W^{1,1}(A;[\alpha,\beta]\times S^2)$. Moreover, the first condition in \eqref{slicing}, \eqref{projonSisId}, and the second condition in \eqref{onlayers} ensure that
\begin{eqnarray}\label{trtildew}
\begin{aligned}
&\tilde w_n^i = w \hbox{ on } \partial A.
\end{aligned}
\end{eqnarray}
Next, we  prove that for fixed $i$,
\begin{eqnarray}\label{L1tildew}
\begin{aligned}
&\lim_{n\to\infty}\Vert \tilde w_n^i - w\Vert_{L^1(A;\RR \times \RR^3)}=0.
\end{aligned}
\end{eqnarray}
We have 
\begin{eqnarray*}
\begin{aligned}
& \int_A |u_n^i - u|\,\dx = \int_A |\ffi_i\, u_n + (1-\ffi_i)\, \bar u_n - \ffi_i\,u - (1-\ffi_i)\, u|\,\dx \lii \int_A \big(| u_n -u| + | \bar u_n - u|\big)\,\dx
\end{aligned}
\end{eqnarray*}
and, arguing as in \eqref{man+tr3} with $A_n$ replaced by the open set $\{x\in A\!: \, \dist(v_n^i,S^2)>\delta_0\}$, where $\delta_0\in(0,\frac{1}{4})$ is fixed, 
\begin{eqnarray*}
\begin{aligned}
&\int_A |\pi_{y_n^i}\circ  v_n^i - v|\,\dx \lii \max\Big\{\frac{2}{\delta_0}, \overline C\Big\}\int_A |v_n^i - v|\,\dx\lii \max\Big\{\frac{2}{\delta_0}, \overline C\Big\}\int_A \big(| v_n -v| + | \bar v_n - v|\big)\,\dx.
\end{aligned}
\end{eqnarray*}
This yields
 \eqref{L1tildew} because $\{u_n\}_{n\in\NN}$ and $\{\bar u_n\}_{n\in\NN}$ are sequences converging to $u$ in $L^1(A)$, while $\{v_n\}_{n\in\NN}$ and $\{\bar v_n\}_{n\in\NN}$ are sequences converging to $v$ in $L^1(A;\RR^3)$.

We now estimate the functional evaluated at $\tilde w^i_n$. Using the bounds in \eqref{suffbounds}, \eqref{slicing2}, and \eqref{slicing1}, in this order, we deduce that
{\setlength\arraycolsep{0.5pt}
\begin{eqnarray}\label{inttildew}
&& \int_A \mathfrak{h}(\tilde u_n^i,
\tilde v_n^i, \grad\tilde
u_n^i, \grad\tilde v_n^i)\,\dx
\nonumber\\
&&\quad\lii \int_{\overline
A_{n,i-1}} \mathfrak{h}( u_n,  v_n,
\grad u_n, \grad v_n)\,\dx
+ \int_{A_{n,i}\backslash\overline
A_{n,i-1}} C(1+|\grad
\tilde u_n^i| + |\grad\tilde
v_n^i|)\,\dx\nonumber\\
&&\hskip20mm + \int_{A\backslash
A_{n,i}} C(1 + |\grad
\bar u_n| + |\grad \bar v_n|)\,\dx
\nonumber\\
&&\quad \lii \int_{A}
\mathfrak{h}( u_n,  v_n, \grad u_n,
\grad v_n)\,\dx + \tilde
C\int_{A_{n,i}\backslash\overline
A_{n,i-1}} \Big(1+|\grad
w_n| + |\grad\bar w_n|
+ \frac{c}{\kappa_n}|w_n
-\bar w_n|\Big)\,\dx
\nonumber\\
&&\hskip20mm +\, 2C\int_{A\backslash
A_{n,0}} (1 + |\grad
\bar w_n| )\,\dx,
\end{eqnarray}}%
where $\tilde C$ is a positive constant only depending on $C$ and $C_\star$.
Furthermore, using the definition of $\kappa_n$,
{\setlength\arraycolsep{0.5pt}
\begin{eqnarray*}
&& \frac{1}{b_n}\sum_{i=1}^{b_n}
\int_{A_{n,i}\backslash\overline
A_{n,i-1}} \Big(1+|\grad
w_n| + |\grad\bar w_n|
+ \frac{c}{\kappa_n}|w_n
-\bar w_n|\Big)\,\dx\\
&&\quad \lii  \frac{1}{b_n}
\int_A \Big(1+|\grad
w_n| + |\grad\bar w_n|
+ \frac{c}{\kappa_n}|w_n
-\bar w_n|\Big)\,\dx
\lii \frac{\calL^2(A)}{b_n}
+ \frac{1}{n} + c\Vert
w_n - \bar w_n 
\Vert_{L^1(A;\RR \times \RR^3)}^{\frac{1}{2}}.
\end{eqnarray*}}%
Thus, there exists $i_n\in\{1,\cdots,b_n\}$ such that
{\setlength\arraycolsep{0.5pt}
\begin{eqnarray}\label{inttildew1}
&& \int_{A_{n,i_n}\backslash\overline
A_{n,i_n-1}} \Big(1+|\grad
w_n| + |\grad\bar w_n|
+ \frac{c}{\kappa_n}|w_n
-\bar w_n|\Big)\,\dx
\nonumber\\
&&\quad\lii \frac{\calL^2(A)}{b_n}
+ \frac{1}{n} + c\Vert
w_n - \bar w_n\Vert_{L^1(A;\RR \times \RR^3)}^{\frac{1}{2}}
= o(1) \enspace \hbox{as $n\to\infty$.}
\end{eqnarray}}%
Fixing  $j\in\NN$ and defining $\tilde A_j:=\{ x\in A\!:\, \dist(x,\partial A)>1/j\}$, we have  $A\backslash \tilde A_j\supset A\backslash A_{n,0}$ for all $n$ large enough because $a_n\to0$ as $n\to\infty$. Hence, using the fact that $A\backslash \tilde A_j$ is a closed subset of $A$ and $D\bar w_n\weaklystar Dw$ weakly-$\star$ in ${\cal M}(A;\RR^2 \times \RR^{3\times 2})$, we get, for fixed $j$,
\begin{eqnarray}\label{inttildew2}
\begin{aligned}
& \limsup_{n\to\infty}\int_{A\backslash A_{n,0}} (1 + |\grad \bar w_n| )\,\dx \lii \limsup_{n\to\infty}\int_{A\backslash \tilde A_j} (1 + |\grad \bar w_n| )\,\dx \lii  \calL^2(A\backslash \tilde A_j) + |Dw|(A\backslash \tilde A_j).
\end{aligned}
\end{eqnarray}
Observing that $\{A\backslash \tilde A_j\}_{j\in\NN}$ is a decreasing sequence of $(\calL^2 + |Dw|)$-finite measure sets whose intersection is the empty set, letting $j\to\infty$ in \eqref{inttildew2}, we conclude that
\begin{eqnarray}\label{inttildew3}
\begin{aligned}
& \lim_{n\to\infty}\int_{A\backslash A_{n,0}} (1 + |\grad \bar w_n| )\,\dx =0.
\end{aligned}
\end{eqnarray}
Finally, setting $\tilde w_n:=\tilde w_n^{i_n}$, in view of \eqref{trtildew}--\eqref{inttildew1} and \eqref{inttildew3}, the sequence $\{\tilde w_n\}_{n\in\NN}\subset W^{1,1}(A;[\alpha,\beta]\times S^2)$ satisfies \eqref{L1andTr} and \eqref{notincenergy}.
\end{proof}

\begin{remark}\label{slicinghalflip}
If $A\in\calA(\Omega)$ is of the form $A=A_1\sminus \overline A_0$, where $A_1\in\calA(\Omega)$,  $A_0\in\calA_\infty(\Omega)$, and $A_0\subset\subset A_1$, then in view of Remark~\ref{man+trhalflip}, Lemma~\ref{slicingMan} holds for all such open sets $A$ as long as we replace the trace condition in \eqref{L1andTr} by  ``$\bar w_n = w$ on $\partial A_0$''.
\end{remark}

\begin{lemma}\label{CalFmeasure}
For every $(u,v)\in BV(\Omega;[\alpha,\beta])\times BV(\Omega;S^2)$, the set function 
{\setlength\arraycolsep{0.5pt}
\begin{eqnarray}\label{locrelaxfct}
A\in \calA(\Omega)\mapsto
\calF(u,v;A):=\inf\Big\{
&& \liminf_{n\to+\infty}
\int_A f(u_n(x), v_n(x),\grad
u_n(x),\grad v_n(x))\,\dx\!:
\nonumber\\
&& \hskip25mm n\in\NN, (u_n,v_n)\in W^{1,1}(A;[\alpha,\beta])\times
W^{1,1}(A;S^2),
\nonumber\\
&& \hskip25mm u_n\to u \hbox{
in } L^1(A), v_n\to v
\hbox{ in } L^1(A;\RR^3)\Big\}
\end{eqnarray}}%
is the restriction of a Radon measure on $\Omega$ to $\calA(\Omega)$.
\end{lemma}

\begin{proof}
Fix $w=(u,v)\in BV(\Omega;[\alpha,\beta]\times S^2)$. Using the bounds \eqref{boundsf} and a
diagonalization argument, we can find a sequence $\{w_n\}_{n\in\NN}= \{(u_n,v_n)\}_{n\in\NN}\subset W^{1,1}(\Omega;[\alpha,\beta]\times S^2)$ converging to $w$ in $L^1(\Omega;\RR \times \RR^3)$ such that
{\setlength\arraycolsep{0.5pt}
\begin{eqnarray*}
&& \calF(w;\Omega)=\lim_{n\to\infty}\int_\Omega
f(u_n(x), v_n(x),\grad
u_n(x),\grad v_n(x))\,\dx,\\
&& \mu_n:= f(u_n,v_n,\grad
u_n,\grad v_n)\calL^2_{\lfloor\Omega}\weaklystar\mu
\hbox{ weakly-$\star$
in ${\cal M}(\Omega)$}
\end{eqnarray*}}%
for some nonnegative Radon measure $\mu\in\calM(\Omega)$.


We claim that for all $A\in\calA(\Omega)$, 
\begin{eqnarray}\label{calF=mu}
\begin{aligned}
& \calF(w;A) = \mu(A).
\end{aligned}
\end{eqnarray}
We will proceed in  three steps.

\underbar{Step1.} We prove that for all $A\in\calA(\Omega)$,
\begin{eqnarray}\label{ubcalF}
\begin{aligned}
& \calF(w;A)\lii (3+\beta)\widetilde C|Dw|(A),
\end{aligned}
\end{eqnarray}
where $\widetilde C$ is a positive constant only depending on $\bar C$ and $C_\star$. 

Arguing as in Lemma~\ref{man+tr}, we can find a sequence $\{\bar w_n\}_{n\in\NN}= \{(\bar u_n,\bar v_n)\}_{n\in\NN}\subset W^{1,1} (A; [\alpha,\beta] \times S^2)\cap C^\infty (A;\RR \times \RR^3)$ such that $\lim_{n\to\infty}\Vert \bar w_n - w\Vert_{L^1(A;\RR \times \RR^3)}=0$ and $\limsup_{n\to\infty}\int_A|\grad\bar w_n(x)|\,\dx\lii \widetilde C |Dw|(A)$, where $\widetilde C$ is a positive constant only depending on $\bar C$ and $C_\star$. Then,
by \eqref{boundsf},
{\setlength\arraycolsep{0.5pt}
\begin{eqnarray*}
\calF(w;A)&&\lii \liminf_{n\to\infty}\int_A
f(\bar u_n(x), \bar v_n(x),
\grad\bar u_n(x), \grad
\bar v_n(x))\,\dx\\
&&\lii \liminf_{n\to\infty}\int_A
(3+\beta)|\grad \bar
w_n(x)|\,\dx\lii (3+\beta)\widetilde
C|Dw|(A).
\end{eqnarray*}}%

\underbar{Step 2.} We claim that for all $A_1, A_2, A_3\in \calA(\Omega)$ such that $A_1\subset\subset A_2\subset A_3$, the following inequality holds
\begin{eqnarray}\label{subcalF}
\begin{aligned}
&\calF(w;A_3)\lii \calF(w;A_2) + \calF(w;A_3\sminus \overline A_1).
\end{aligned}
\end{eqnarray}
Let $U\in \calA_\infty(\Omega)$ be such that $A_1\subset\subset U\subset\subset A_2$. Let $\{w_n^{1}\}_{n\in\NN}= \{(u_n^1,v_n^1)\}_{n\in\NN}\subset W^{1,1}(A_3\sminus \overline A_1;[\alpha,\beta]\times S^2)$ and $\{w_n^{2}\}_{n\in\NN}= \{(u_n^2,v_n^2)\}_{n\in\NN}\subset W^{1,1}(A_2;[\alpha,\beta]\times S^2)$ be sequences converging to $w$ in $L^1(A_3\sminus \overline A_1;\RR \times \RR^3)$ and $L^1(A_2;\RR \times \RR^3)$, respectively, and such that
\begin{eqnarray}\label{subcalF1}
\begin{aligned}
&\calF(w;A_3\sminus \overline A_1)=\lim_{n\to\infty}\int_{A_3\sminus \overline A_1} f(u_n^1(x), v_n^1(x),\grad u_n^1(x),\grad v_n^1(x))\,\dx,\\
& \calF(w;A_2)=\lim_{n\to\infty}\int_{A_2} f(u_n^2(x), v_n^2(x),\grad u_n^2(x),\grad v_n^2(x))\,\dx.
\end{aligned}
\end{eqnarray}  
In view of Lemma~\ref{slicingMan} and 
Remark~\ref{slicinghalflip}, we can find sequences
$\{\tilde w_n^{1}\}_{n\in\NN}= \{(\tilde u_n^1,\tilde v_n^1)\}_{n\in\NN}\subset W^{1,1}(A_3\sminus \overline U;[\alpha,\beta]\times S^2)$ and $\{\tilde w_n^{2}\}_{n\in\NN}= \{(\tilde u_n^2,\tilde v_n^2)\}_{n\in\NN}\subset W^{1,1}(U;[\alpha,\beta]\times S^2)$ \ converging to $w$ in $L^1(A_3\sminus \overline U;\RR \times \RR^3)$ and $L^1(U;\RR \times \RR^3)$, respectively, and such that
\begin{eqnarray}\label{subcalF2}
\begin{aligned}
& \tilde w^1_n=w \hbox{ on } \partial U,\quad  \tilde w^2_n=w \hbox{ on } \partial U,\\
&\limsup_{n\to\infty}\int_{A_3\sminus \overline U} f(u_n^1, v_n^1,\grad u_n^1,\grad v_n^1)\,\dx \gii \limsup_{n\to\infty}\int_{A_3\sminus \overline U} f(\tilde u_n^1,\tilde  v_n^1,\grad\tilde  u_n^1,\grad\tilde  v_n^1)\,\dx,\\
& \limsup_{n\to\infty}\int_{U} f(u_n^2, v_n^2,\grad u_n^2,\grad v_n^2)\,\dx\gii \limsup_{n\to\infty}\int_{U} f(\tilde u_n^2,\tilde  v_n^2,\grad\tilde  u_n^2,\grad\tilde  v_n^2)\,\dx.
\end{aligned}
\end{eqnarray}  

Define for $n\in\NN$,
\begin{eqnarray*}
\begin{aligned}
&\tilde w_n:=\begin{cases}
\tilde w_n^{1} & \hbox{in } A_3\sminus \overline U,\\
\tilde w_n^2  & \hbox{in } U.
\end{cases}
\end{aligned}
\end{eqnarray*}
Then, $\tilde w_n=(\tilde u_n,\tilde v_n)\in W^{1,1}(A_3;[\alpha,\beta]\times S^2)$ and $\{\tilde w_n\}_{n\in\NN}$ is a sequence converging to $w$ in $L^1(A_3;\RR \times \RR^3)$. Moreover, using \eqref{subcalF1} and \eqref{subcalF2}, together with the set inclusions $A_3\sminus\overline U \subset A_3\sminus\overline A_1$ and $U\subset A_2$ and the non-negativeness of $f$,
{\setlength\arraycolsep{0.5pt}
\begin{eqnarray*}
\calF(w;A_3)&&\lii \liminf_{n\to\infty}
\int_{A_3} f(\tilde u_n,\tilde
 v_n,\grad \tilde u_n,\grad
\tilde v_n)\,\dx\\
&& = \liminf_{n\to\infty}\bigg(
\int_{A_3\sminus \overline
U} f(\tilde u_n^1,\tilde
 v_n^1,\grad\tilde  u_n^1,\grad\tilde
 v_n^1)\,\dx + \int_{U}
f(\tilde u_n^2,\tilde
 v_n^2,\grad\tilde  u_n^2,\grad\tilde
 v_n^2)\,\dx\bigg)\\
&&\lii \limsup_{n\to\infty}\int_{A_3\sminus
\overline U} f(\tilde
u_n^1,\tilde  v_n^1,\grad\tilde
 u_n^1,\grad\tilde  v_n^1)\,\dx
+ \limsup_{n\to\infty}\int_{U}
f(\tilde u_n^2,\tilde
 v_n^2,\grad\tilde  u_n^2,\grad\tilde
 v_n^2)\,\dx\\
&&\lii \calF(w;A_3\sminus \overline
A_1)  + \calF(w;A_2),
\end{eqnarray*}}%
which concludes Step~2.

\underbar{Step 3.} We establish \eqref{calF=mu}.
Fix $A\in \calA(\Omega)$.

{\sl Substep 3.1.} We prove that $\calF(w;A)\lii \mu(A)$.

Using the upper  semicontinuity of the weak-$\star$ convergence in $\calM(\Omega)$ with respect to compact sets and the fact that $\{w_n\}_{n\in\NN}$ is an admissible sequence for $\calF(w;A)$, we conclude that
\begin{eqnarray}\label{caF=mu1}
\begin{aligned}
& \calF(w;A)\lii \liminf_{n\to\infty}\int_A f(u_n, v_n,\grad u_n,\grad v_n)\,\dx = \liminf_{n\to\infty}\mu_n(A) \lii \limsup_{n\to\infty} \mu_n(\overline A)\lii \mu(\overline A).
\end{aligned}
\end{eqnarray}
Fix $\epsi>0$ and let $A_\epsi'$, $A_\epsi''\in\calA(\Omega)$ be such that $A_\epsi'\subset\subset A_\epsi''\subset\subset A$ and $|Dw|(A\sminus\overline A_\epsi')<\epsi/\tilde c$, where $\tilde c:= (3+\beta)\widetilde C$. Using \eqref{subcalF}, \eqref{ubcalF}, and \eqref{caF=mu1}, in this order, we obtain
\begin{eqnarray*}
\begin{aligned}
\calF(w;A)\lii \calF(w; A_\epsi'') + \calF(w; A\sminus\overline A_\epsi') \lii \calF(w; A_\epsi'') + \epsi \lii \mu(\overline A_\epsi'') + \epsi \lii \mu(A) + \epsi,
\end{aligned}
\end{eqnarray*}
from which Substep~{3.1} follows by letting $\epsi\to0^+$.

{\sl Substep 3.2.} We prove that $\calF(w;A)\gii \mu(A)$.

Fix $\epsi>0$, and let $A_\epsi\in\calA(\Omega)$ be such that $A_\epsi\subset\subset  A$ and $\mu(A\sminus\overline A_\epsi)<\epsi$. Using the equality $\calF(w;\Omega)=\mu(\Omega)$, from Substep~{3.1} (applied to $\Omega\sminus\overline A_\epsi$) and Step~2 (applied to $A_\epsi\subset\subset A\subset \Omega$), it follows that
{\setlength\arraycolsep{0.5pt}
\begin{eqnarray*}
 \mu(A) &&= \mu(A\sminus
\overline A_\epsi) +
\mu(\overline A_\epsi)
<\epsi + \mu(\overline
A_\epsi)  = \epsi + \mu(\Omega)
- \mu(\Omega\sminus\overline
A_\epsi)\\
&&\lii \epsi + \calF(w;\Omega)
- \calF(w;\Omega\sminus\overline
A_\epsi) \lii 
\epsi + \calF(w;A).
\end{eqnarray*}}%
Letting $\epsi\to0^+$, we conclude the proof of Substep~{3.2} as well as of Lemma~\ref{CalFmeasure}.
\end{proof}
\end{lemma}

\begin{lemma}\label{onManae}
Let $w\in BV(\Omega;[\alpha,\beta]\times S^2)$. Then\footnote{We refer to Subsection~\ref{onBV} for the notation concerning $BV$ functions.}, 
\begin{itemize}
\item [(a)] $\tilde w(x)\in [\alpha,\beta]\times S^2$ for all $x\in A_w=\Omega\sminus S_w$;

\item [(b)] $w^\pm(x)\in [\alpha,\beta]\times
S^2$ for all $x\in J_w$;

\item[(c)] $\grad w(x)\in \big[T_{w(x)}([\alpha,\beta]\times S^2) \big]^2$ for $\calL^2$-\aev\ $x\in\Omega$;

\item[(d)] $\displaystyle W^c(x):={\d D^cw\over \d|D^cw|}(x)\in \big[T_{\tilde w(x)}([\alpha,\beta]\times S^2)\big]^2$ for $|D^c w|$-\aev\ $x\in\Omega$.
\end{itemize}

\end{lemma}

\begin{proof}
We start by proving {\it (a)} and {\it (b)}. Let $x_0\in A_w$. Because $w(\cdot)\in [\alpha,\beta]\times S^2$  $\calL^2$-\aev\ in $\Omega$, we have $|w(\cdot) - \tilde w(x_0)|\gii\dist\big(\tilde w(x_0), [\alpha,\beta]\times S^2\big)$ $\calL^2$-\aev\ in $\Omega$, and so
{\setlength\arraycolsep{0.5pt}
\begin{eqnarray*}
&& 0=\lim_{\epsilon\to0^+}\dashint_{B(x_0,\epsilon)}|w(y)
- \tilde w(x_0)|\,\dy
\gii\dist\big(\tilde
w(x_0), [\alpha,\beta]\times
S^2\big).
\end{eqnarray*}}%
 This implies that $\tilde w(x_0)\in [\alpha,\beta]\times S^2$. Similarly, if $x_0\in J_w$, then
{\setlength\arraycolsep{0.5pt}
\begin{eqnarray*}
&&0=\lim_{\epsilon\to0^+}\dashint_{B^\pm_{\nu_w(x_0)}(x_0,\epsilon)}|w(y)
-  w^\pm(x_0)|\,\dy \gii\dist\big(
w^\pm(x_0), [\alpha,\beta]\times
S^2\big),
\end{eqnarray*}}%
from which we conclude that $w^\pm(x_0)\in [\alpha,\beta]\times S^2$.

In order to prove {\it (c)} and {\it (d)}, we fix an open, bounded subset $U$ of $\RR\times\RR^3$ such that $U\supset[\alpha,\beta]\times S^2$, and we consider a cut-off function $\theta\in C^\infty_c(\RR\times\RR^3;[0,1])$ satisfying $\supp\theta\subset U$ and $\theta(r,s)=1$ for all $(r,s)\in[\alpha,\beta]\times S^2$. Finally, we define $\phi:\RR\times\RR^3\to\RR$ by setting $\phi(r,s):=\theta(r,s)(|s|^2-1)$.
Then, $\phi$ belongs to $C^1_c(\RR\times\RR^3)$ and
\begin{eqnarray*}
\begin{aligned}
& \frac{\partial\phi}{\partial r}(r,s)= 
\frac{\partial\theta}{\partial r}
(r,s)(|s|^2-1),\quad 
\frac{\partial\phi}{\partial s}(r,s)= 
\frac{\partial\theta}{\partial s}(r,s)(|s|^2-1) + 2\theta(r,s)s.
\end{aligned}
\end{eqnarray*}
Hence, if $(r,s)\in[\alpha,\beta]\times S^2$ and $h=(h_1,h')\in\RR\times\RR^3$, then
\begin{eqnarray}\label{kernelgradphi}
\begin{aligned}
& \grad\phi(r,s)\cdot h=0\, \Leftrightarrow\, h_1\in\RR\enspace \land\enspace h'\cdot s =0 \, \Leftrightarrow\, h\in T_{(r,s)}([\alpha,\beta]\times S^2).
\end{aligned}
\end{eqnarray}
Moreover, by Theorem~\ref{chainruleBV}, we have that $\phi\circ w\in BV(\Omega)$ and (see \eqref{chruBV})
{\setlength\arraycolsep{0.5pt}
\begin{eqnarray*}
D (\phi\circ w) &&= \grad
\phi(w)\grad w\calL^N
+ \big(\phi(w^+) - \ffi(w^-)\big)\otimes
\nu_w{\calH^{N-1}}_{\lfloor
J_w}+  \grad \phi(\tilde
w) D^cw\\
&&=\grad \phi(w)\grad
w\calL^N +  \grad \phi(\tilde
w) W^c |D^cw|,
\end{eqnarray*}}%
where we also used {\it (b)} together with the fact that $\phi(r,s)=0$ for $s\in S^2$.
Similarly, since  $w(\cdot)\in [\alpha,\beta]\times S^2$ for $\calL^2$-\aev\ in $\Omega$, it follows that $\phi\circ w=0$  for $\calL^2$-\aev\ in $\Omega$. Thus, $D(\phi\circ w)\equiv0$ and, because $\calL^2$ and $|D^cw|$ are mutually singular measures, we conclude that
{\setlength\arraycolsep{0.5pt}%
\begin{eqnarray*}
&& \grad \phi(w(x))\grad w(x) = 0 \hbox{ for } \calL^2\hbox{-\aev\ } x\in\Omega,\qquad \grad \phi(\tilde w(x)) W^c(x) = 0 \hbox{ for } |D^cw|\hbox{-\aev\ } x\in\Omega;
\end{eqnarray*}}%
that is,
{\setlength\arraycolsep{0.5pt}%
\begin{eqnarray*}
&& \grad \phi(w(x))\cdot(\grad w(x),0) = 0 \hbox{ for }
\calL^2\hbox{-\aev\ } x\in\Omega,\qquad \grad
\phi(\tilde w(x)) \cdot(0, W^c(x)) = 0 \hbox{ for } |D^cw|\hbox{-\aev\
} x\in\Omega,
\end{eqnarray*}}%

which, together with \eqref{kernelgradphi}, yields {\it (c)} and {\it (d)}.\end{proof}

\subsection{On the Lower Bound for ${\calF}$}\label{lowerbound}

Let $G$ denote the function on the right-hand side of \eqref{lscF1}. We claim that 
\begin{eqnarray*}
\begin{aligned}
\calF(u,v)\gii G(u,v)
\end{aligned}
\end{eqnarray*} 
for all $(u,v)\in L^1(\Omega)\times L^1(\Omega;\RR^3)$ or, equivalently,
\begin{eqnarray}\label{liminfineq}
\begin{aligned}
\liminf_{n\to+\infty} F(u_n,v_n)\gii G(u,v)
\end{aligned}
\end{eqnarray}
whenever $\{(u_n,v_n)\}_{n\in\NN}\subset L^1(\Omega)\times L^1(\Omega;\RR^3)$ is a  sequence converging to $(u,v)$ in $L^1(\Omega)\times L^1(\Omega;\RR^3)$. To prove \eqref{liminfineq}, we may assume without loss of generality that
\begin{eqnarray}\label{liminfineq1}
\begin{aligned}
\liminf_{n\to+\infty} F(u_n,v_n) = \lim_{n\to+\infty} F(u_n,v_n)\in\RR^+_0, 
\end{aligned}
\end{eqnarray}%
and for all $n\in\NN$,
\begin{eqnarray*}
\begin{aligned}
 (u_n,v_n)\in W^{1,1}(\Omega;[\alpha,\beta])\times
W^{1,1}(\Omega;S^2).
\end{aligned}
\end{eqnarray*}
In particular, \begin{eqnarray}\label{liminfineq2}
\begin{aligned}
F(u_n,v_n)&= \int_\Omega \big(|\grad u_n| + g(|\grad u_n|)|\grad v_n| + |\grad(u_n v_n)|
\big) \,\dx=\int_\Omega f(u_n,v_n,\grad u_n,\grad v_n)\,\dx\lii C,
\end{aligned}
\end{eqnarray}
for some positive constant $C$ independent of $n$. Hence,
\begin{eqnarray}\label{lbcompact1}
\begin{aligned}
&\int_\Omega \big( |\grad u_n|+ |\grad(u_n v_n)|\big)\,\dx \lii C,
\end{aligned}
\end{eqnarray}
and, in turn, 
\begin{eqnarray}\label{lbcompact2}
\begin{aligned}
&\alpha\int_\Omega |\grad v_n|\,\dx \lii  \int_\Omega |u_n\grad v_n + v_n\otimes\grad u_n|\,\dx  + \int_\Omega |v_n\otimes\grad u_n|\,\dx \lii C.
\end{aligned}
\end{eqnarray}
Thus,
 up to the extraction of a subsequence (not relabeled),
we have\begin{eqnarray*}
\begin{aligned}
& u_n\weaklystar u \hbox{ weakly-$\star$ in } BV(\Omega) \hbox{ and } v_n\weaklystar v \hbox{ weakly-$\star$ in } BV(\Omega;\RR^3) \hbox{ as $n\to+\infty$;}\\
&u(x)\in [\alpha,\beta] \hbox{ and } v(x)\in S^2 \hbox{ for $\calL^2$-\aev\ } x\in\Omega; \\
& \mu_n:=  f(u_n,v_n,\grad u_n,\grad v_n) 
\calL^2_{\lfloor\Omega} \weaklystar \mu \hbox{
weakly-$\star$ in } \calM(\Omega)
\end{aligned}
\end{eqnarray*}
for some nonnegative finite Radon measure 
$\mu\in\calM(\Omega)$. In view of the Radon-Nikodym Theorem, we can decompose $\mu$ into a sum of four mutually singular, nonnegative finite Radon measures as follows:
\begin{eqnarray*}
\begin{aligned}
\mu = \mu_a \calL^2_{\lfloor\Omega} + \mu_c|D^c(u,v)|
+\mu_j|(u,v)^+ - (u,v)^-|\calH^1_{\lfloor J_{(u,v)}} + \mu_s.
\end{aligned}
\end{eqnarray*}
We claim that
{\setlength\arraycolsep{0.5pt}
\begin{eqnarray}
&& \mu_a(x_0)\gii \calQ_T f(u(x_0), v(x_0),\grad u(x_0),\grad v(x_0))\enspace \hbox{ for $\calL^2$-\aev\ $x_0\in\Omega$;}\label{leblb} \\
&& \mu_c(x_0)\gii (\calQ_T f)^\infty\big(\tilde u(x_0) ,\tilde v(x_0), W^c_u(x_0), W^c_v(x_0)\big) \label{cantlb} \enspace \hbox{ for $|D^c(u,v)|$-\aev\ $x_0\in\Omega$;}\quad\qquad\\
&&\mu_j(x_0) \gii \frac{1}{|(u,v)^+(x_0) - (u,v)^-(x_0)|} K\big((u,v)^+(x_0), (u,v)^-(x_0), \nu_{(u,v)}(x_0)\big) \nonumber\\
&&\hskip60.85mm \hbox{ for $|(u,v)^+ - (u,v)^-|\calH^1_{\lfloor J_{(u,v)}}$-\aev\ $x_0\in\Omega$.} \label{jumplb} 
\end{eqnarray}}%

Assume that \eqref{leblb}, \eqref{cantlb}, and \eqref{jumplb} hold, and let $\{\phi_k\}_{k\in \NN}\subset C^\infty_0(\Omega;[0,1])$ be  an increasing
sequence    of smooth cut-off functions such that $\sup_{k\in\NN}\phi_k(x)=1$ for all $x\in\Omega$.
Then,   using the convergence \(\mu_n \weaklystar \mu\)
weakly-$\star$ in \( \calM(\Omega)\), 
{\setlength\arraycolsep{0.5pt}
\begin{eqnarray}\label{almostlb}
\lim_{n\to+\infty}\int_\Omega
f(u_n,v_n,\grad u_n,\grad
v_n)\,\dx &&\gii \liminf_{n\to+\infty}\int_\Omega
\phi_k(x) f(u_n,v_n,\grad
u_n,\grad v_n)\,\dx 
=\int_\Omega \phi_k(x)\,\d\mu
\nonumber \\
&&\gii\intO \phi_k(x)
\calQ_Tf(u(x), v(x),\grad
u(x),\grad v(x))\,\dx
\nonumber\\
&&\quad + \int_{S_{(u,v)}}
\phi_k(x) K\big((u,v)^+(x),(u,v)^-(x),\nu_{(u,v)}(x)\big)\,\d\calH^1(x)
\nonumber\\
&&\quad +\int_{\Omega}\phi_k(x)(\calQ_T
f)^\infty\big(\tilde
u(x) ,\tilde v(x), W^c_u(x),
W^c_v(x)\big)\,\d|D^c(u,v)|(x).
\end{eqnarray}}%
In view of  Lebesgue's Monotone Convergence  Theorem, \eqref{liminfineq1}, and \eqref{liminfineq2}, letting $k\to+\infty$ in \eqref{almostlb}, we obtain \eqref{liminfineq}. 

We start by proving \eqref{leblb} and \eqref{cantlb}.
Let $H_\epsi$ be the function defined in \eqref{defofhepsi}, and let $A\in\calA(\Omega)$. Because  $H_\epsi$ satisfies conditions (H1)--(H4) of \cite{FMXCIII}
by Proposition~\ref{hepsiasFM},
 we have
{\setlength\arraycolsep{0.5pt}
\begin{eqnarray}\label{byFMXCIII1}
&&\liminf_{n\to+\infty}\int_A
H_\epsi(u_n,v_n,\grad
u_n,\grad v_n)\,\dx \nonumber\\ 
&& \quad  \gii \int_A H_\epsi(u,v,\grad
u,\grad v)\,\dx + \int_A
(H_\epsi)^\infty(\tilde
u,\tilde v,W^c_u,W^c_v)\,
\d|D^c(u,v)|(x)\nonumber\\
&&\quad =\int_A \calQ \tilde f(u,v,\grad
u,\grad v)\,\dx + \int_A
(\calQ\tilde f)^\infty(\tilde
u,\tilde v,W^c_u,W^c_v)\,
\d|D^c(u,v)|(x) + O(\epsi)
\end{eqnarray}}%
as $\epsi\to0^+$, where in the last equality we also  used the identity
\begin{eqnarray*}
\begin{aligned}
(H_\epsi)^\infty(r,s,\xi,\eta) = (\calQ\tilde
f)^\infty(r,s,\xi,\eta) + \epsi(|\xi| + |\eta|).
\end{aligned}
\end{eqnarray*}
Recalling  that for $(r,s)\in[\alpha,\beta]\times S^2$, $T_{(r,s)}([\alpha,\beta]\times S^2) = \RR\times T_s(S^2)$, from Lemma~\ref{onManae}, \eqref{fandtildef},
\eqref{lbcompact1}, and \eqref{lbcompact2}, we conclude that
{\setlength\arraycolsep{0.5pt}
\begin{eqnarray*}
&&\liminf_{n\to+\infty}\int_A
f(u_n,v_n,\grad u_n,\grad
v_n)\,\dx =  \liminf_{n\to+\infty}\int_A
\tilde f(u_n,v_n,\grad
u_n,\grad v_n)\,\dx \nonumber \\
&&\quad\gii \liminf_{n\to+\infty}\int_A
\calQ \tilde f(u_n,v_n,\grad
u_n,\grad v_n)\,\dx \gii
\liminf_{n\to+\infty}\int_A
H_\epsi(u_n,v_n,\grad
u_n,\grad v_n)\,\dx -
\epsi C.
\end{eqnarray*}}%
These estimates and \eqref{byFMXCIII1} entail
{\setlength\arraycolsep{0.5pt}
\begin{eqnarray}\label{byFMXCIII2}
&&\liminf_{n\to+\infty}\int_A
f(u_n,v_n,\grad u_n,\grad
v_n)\,\dx \gii \int_A \calQ \tilde f(u,v,\grad
u,\grad v)\,\dx + \int_A
(\calQ\tilde f)^\infty(\tilde
u,\tilde v,W^c_u,W^c_v)\,
\d|D^c(u,v)|(x). \nonumber \\
\end{eqnarray}}%
 Since $\calL^N_{\lfloor\Omega}$ and $|D^c(u,v)|_{\lfloor\Omega}$ are mutually singular, \eqref{leblb} and \eqref{cantlb} are a consequence of \eqref{byFMXCIII2}.

We now establish \eqref{jumplb}. We start by recalling that if $\nu\in S^1,$ then $Q_\nu$ denotes the unit cube in $\RR^2$ centered at the origin and with two faces orthogonal to $\nu$. We set
\begin{eqnarray*}
\begin{aligned}
Q_\nu^+:= \big\{x\in Q_\nu\!:\,\, x\cdot \nu>0\big\}, \qquad Q_\nu^-:= \big\{x\in Q_\nu\!:\,\, x\cdot \nu<0\big\},
\end{aligned}
\end{eqnarray*}
and, for $x_0\in\RR^2$ and $\epsilon>0$,
\begin{eqnarray*}
\begin{aligned}
Q_\nu(x_0,\epsilon):= x_0 + \epsi Q_\nu,\qquad Q_\nu^\pm(x_0,\epsilon):= x_0 + \epsi Q_\nu^\pm.
\end{aligned}
\end{eqnarray*}
To simplify the notation, we further set $w:=(u,v)$ and $w_n:=(u_n,v_n)$, $n\in\NN$. Let $x_0\in J_w$ be such that
{\setlength\arraycolsep{0.5pt}
\begin{eqnarray}
&&\lim_{\epsilon\to0^+} \frac{1}{\epsilon^2}\int_{Q^\pm_{\nu_w(x_0)}(x_0,\epsilon)}\big|w(x) - w^\pm(x_0)\big|\,\dx = 0, \label{appjplb} \\
&& \lim_{\epsilon\to0^+} \frac{1}{\epsilon} \int_{S_w\cap Q_{\nu_w(x_0)}(x_0,\epsilon)} \big|w^+(x) - w^-(x)\big|\,\d\calH^1(x) =\lim_{\epsilon\to0^+} \frac{1}{\epsilon}|w^+
- w^-|\calH^1_{\lfloor(S_w\cap Q_{\nu_w(x_0)}(x_0,\epsilon))}
\nonumber \\
&&\hskip75.75mm=  \big|w^+(x_0) - w^-(x_0)\big|, \label{lebpjumplb}\\
&&\mu_j(x_0) = \lim_{\epsilon\to0^+} \frac{\mu(Q_{\nu_w(x_0)}(x_0,\epsilon))}{|w^+ - w^-|\calH^1_{\lfloor(S_w\cap Q_{\nu_w(x_0)}(x_0,\epsilon))}}\in\RR.\label{RDjumplb}
\end{eqnarray}}%
In view of Proposition~\ref{appppties}, Theorem~\ref{conseqLBDT}, and Besicovitch Differentiation Theorem, \eqref{appjplb}--\eqref{RDjumplb} hold for $|w^+ - w^-|\calH^1_{\lfloor J_w}$-\aev\ $x_0\in\Omega$.

Let $\{\epsilon_i\}_{i\in\NN}$ be a sequence of positive numbers converging to zero such that the boundary of each $Q_{\nu_w(x_0)}(x_0,\epsilon_i)$ has zero $\mu$--measure. Using \eqref{lebpjumplb}, \eqref{RDjumplb}, and the weak-$\star$ convergence $\mu_n\weaklystar\mu$ in $\calM(\Omega)$, we obtain
{\setlength\arraycolsep{0.5pt}
\begin{eqnarray}\label{jumplb1}
&& \big|w^+(x_0) - w^-(x_0)\big|
\mu_j(x_0) = \lim_{i\to+\infty}
\frac{1}{\epsilon_i}
\int_{Q_{\nu_w(x_0)}(x_0,\epsilon_i)}
\d\mu \nonumber\\
&&\quad= \lim_{i\to+\infty}
\lim_{n\to+\infty}
\frac{1}{\epsilon_i}
\int_{Q_{\nu_w(x_0)}(x_0,\epsilon_i)}
f(u_n(x),v_n(x),\grad
u_n(x),\grad v_n(x))\,\dx
\nonumber\\
&&\quad = \lim_{i\to+\infty}
\lim_{n\to+\infty}
\int_{Q_{\nu_w(x_0)}}\epsilon_i\,
f\Big(u_{n,\epsilon_i}(y),
v_{n,\epsilon_i}(y),
\frac{1}{\epsilon_i}\grad
u_{n,\epsilon_i}(y),
\frac{1}{\epsilon_i}\grad
v_{n,\epsilon_i}(y)\Big)\,\dy
\nonumber\\
&&\quad= \lim_{i\to+\infty}
\lim_{n\to+\infty}
\int_{Q_{\nu_w(x_0)}}\Big[|\grad
u_{n,\epsilon_i}(y)|
+  |\grad (u_{n,\epsilon_i}
v_{n,\epsilon_i})(y)|
+ g\Big(\frac{1}{\epsilon_i}|\grad
u_{n,\epsilon_i}(y)|\Big)|\grad
v_{n,\epsilon_i}(y)|\Big]\,\dy
,
\end{eqnarray}}%
where 
\begin{eqnarray*}
\begin{aligned}
u_{n,\epsilon_i}(y):= u_n(x_0 + \epsilon_i y),\quad v_{n,\epsilon_i}(y):= v_n(x_0 + \epsilon_i y),\quad y\in Q_{\nu_w(x_0)}.
\end{aligned}
\end{eqnarray*}

Setting $w_{n,\epsilon_i}:= (u_{n,\epsilon_i},
v_{n,\epsilon_i})$, we have that $w_{n,\epsilon_i}\in W^{1,1}(Q_{\nu_w(x_0)};[\alpha,\beta]\times S^2$,
and, in view of \eqref{appjplb},
\begin{eqnarray}\label{jumplb2}
\begin{aligned}
&  \lim_{i\to+\infty}\lim_{n\to+\infty}\int_{Q_{\nu_w(x_0)}}|w_{n,\epsilon_i}(y) - w_0(y)|\,\dy = 0,
\end{aligned}%
\end{eqnarray}
where
\begin{eqnarray*}
\begin{aligned}
w_0(y):=
\begin{cases}
 w^+(x_0) & \hbox{ if } y\cdot \nu_w(x_0) \gii 0,\\
 w^-(x_0) & \hbox{ if } y\cdot \nu_w(x_0) < 0.
\end{cases}
\end{aligned}
\end{eqnarray*}
By a standard diagonalization argument, from \eqref{jumplb1}, \eqref{jumplb2}, and Lemma~\ref{slicingMan}, we can construct a sequence $\{\bar w_k\}_{k\in\NN} = \{(\bar u_k,\bar v_k)\}_{k\in\NN}\subset W^{1,1}(Q_{\nu_w(x_0)};[\alpha,\beta]\times S^2)$ such that $\bar w_k = w_0$ on $\partial Q_{\nu_w}(x_0)$ for all $k\in\NN$, 
\begin{eqnarray*}
\begin{aligned}
\lim_{k\to+\infty}\Vert \bar w_k - w_0\Vert_{L^1(Q_{\nu_w(x_0)};
\RR \times \RR^3)}=0,
\end{aligned}
\end{eqnarray*}
and
{\setlength\arraycolsep{0.5pt}
\begin{eqnarray}\label{jumplb3}
&&\big|w^+(x_0) - w^-(x_0)\big|
\mu_j(x_0) \gii \limsup_{k\to+\infty}
\int_{Q_{\nu_w(x_0)}}\Big[|\grad
\bar u_k(y)| +  |\grad
(\bar u_k \bar v_k)(y)|+
g\Big(\frac{1}{\epsilon_{i_k}}|\grad\bar
u_k(y)|\Big)|\grad\bar
v_k(y)|\Big]\,\dy. \nonumber\\
\end{eqnarray}}%
Because
{\setlength\arraycolsep{0.5pt}
\begin{eqnarray*}
\int_{Q_{\nu_w(x_0)}}
g\Big(\frac{1}{\epsilon_{i_k}}
|\grad\bar
u_k(y)|\Big)|\grad\bar
v_k(y)|\,\dy&&\gii \int_{\{y\in
Q_{\nu_w(x_0)}: \grad
\bar u_k(y)=0\}} |\grad\bar
v_k(y)|\,\dy\\
&& = \int_{Q_{\nu_w(x_0)}}
\chi_{_{\{0\}}}(|\grad
\bar u_k(y)|)|\grad\bar
v_k(y)|\,\dy,
\end{eqnarray*}}%
from \eqref{jumplb3} and \eqref{defK}, and since $(\bar u_k,\bar v_k)\in \calP((u,v)^+(x_0), (u,v)^-(x_0),\nu_{(u,v)}(x_0))$ 
for all $k\in\NN$, we obtain  \eqref{jumplb}.

\subsection{On the Upper Bound for ${\calF}$}\label{upperbound}

We  identify the Radon measure on $\Omega$ given by Lemma~\ref{CalFmeasure} with its restriction $\calF(u,v;\cdot)$ to
$\calA(\Omega)$ introduced
in \eqref{locrelaxfct}.
In view of \eqref{ubcalF}
and \cite[Step~1 of Prop.~{4.4}]{AMTXCI},
we have that $\calF(u,v;\cdot)$
is
 local   in $\calB(\Omega)
$ in the following sense:
\begin{eqnarray*}
\begin{aligned}
&\calF(u,v;B)=\calF(u',v';B)
\end{aligned}
\end{eqnarray*}
for all $B\in\calB(\Omega)$
and $w:=(u,v)$, $
w':=(u',
 v')\in BV(\Omega;[\alpha,
\beta])\times BV(\Omega;S^2)$ such that
\begin{eqnarray*}
\begin{aligned}
& B\subset S_w \cap S_{w'},\qquad (w^+(x), w^-(x),\nu_w(x))
\sim ({w'}^+(x), {w'}^-(x),\nu_{w'}(x)) \,\, \hbox{ for all }
x\in B,
\end{aligned}
\end{eqnarray*}
where
\begin{eqnarray}\label{equivJump}
\begin{aligned}
& (a,b,\nu) \sim (a',b',\nu')
\Leftrightarrow (a=a'
\land b=b' \land \nu
= \nu') \lor (b=a'
\land a=b' \land \nu
= -\nu') .
\end{aligned}
\end{eqnarray}

\begin{lemma}\label{abscontpartCalF}
Let $(u,v)\in BV(\Omega;[\alpha,
\beta])\times BV(\Omega;S^2)$. Then, for $\calL^2$-\aev\ $x_0\in\Omega,$ we have
\begin{eqnarray}\label{ubabscontpart}
{\d\calF(u,v;\cdot)\over\d\calL^2}(x_0)\lii \calQ_T f(u(x_0),v(x_0),\grad u(x_0),\grad v(x_0)).
\end{eqnarray}
\end{lemma}

\begin{proof} Let $x_0\in\Omega$ be such that
{\setlength\arraycolsep{0.5pt}
\allowdisplaybreaks
\begin{eqnarray}
&&\alpha\lii u(x_0)\lii\beta, \label{uinconvex} \\
&&|v(x_0)|=1,\enspace \grad v(x_0)\in [T_{v(x_0)}(S^2)]^2, \label{vinstwo}\\
&&{\d\calF(u,v;\cdot)\over\d\calL^2}(x_0)  {\hbox{ exists and is finite}}, \label{rnforcalF}\\
&&\lim_{\epsilon\to0^+}
\dashint_{B(x_0,\epsilon)}|u(x)-u(x_0)|\,\dx = 0,\hskip45mm \label{lebu}\\
&&\lim_{\epsilon\to0^+}
\dashint_{B(x_0,\epsilon)}|\grad u(x)-\grad u(x_0)|\,\dx = 0, \label{lebgradu}\\
&&\lim_{\epsilon\to0^+}
\dashint_{B(x_0,\epsilon)}|v(x)-v(x_0)|\,\dx = 0, \label{lebv}\\
&&\lim_{\epsilon\to0^+}
\dashint_{B(x_0,\epsilon)}|\grad v(x)-\grad v(x_0)|\,\dx = 0, \label{lebgradv}\\
&&\lim_{\epsi\to0^+}{|D^s u|(B(x_0,\epsilon))\over \calL^2(B(x_0,\epsilon))}=0, \label{lspiszerou}\\
&&\lim_{\epsi\to0^+}{|D^s v|(B(x_0,\epsilon))\over \calL^2(B(x_0,\epsilon))}=0. \label{lspiszerov}
\end{eqnarray}}%

We observe that \eqref{uinconvex}--\eqref{lspiszerov} hold for $\calL^2$-\aev\ $x_0\in\Omega$. 

Fix $\epsi>0$. 
Let $\ffi_\epsi\in W^{1,\infty}_0(Q)$ and $\psi_\epsi\in W^{1,\infty}_0(Q;T_{v(x_0)}(S^2))$, extended by periodicity to the whole $\RR^2$, be such that%
{\setlength\arraycolsep{0.5pt}
\begin{eqnarray}\label{definfQQT}
&&\calQ_T f(u(x_0),v(x_0),\grad
u(x_0),\grad v(x_0))
+ \epsi \gii\int_Q f(u(x_0),v(x_0),
\grad
u(x_0) + \grad\ffi_\epsi(y),\grad
v(x_0)+\grad\psi_\epsi(y))\,\dy.\qquad\quad
\end{eqnarray}}%
For each $n\in\NN $ and \(\epsi>0\), consider the function  $\Phi_{n,\epsi}:\RR\to\RR$ defined by
{\setlength\arraycolsep{0.5pt}
\begin{eqnarray*}
&& \Phi_{n,\epsi}(r):={n(\beta-\alpha)
r + (\beta + \alpha)\Vert\ffi_\epsi\Vert_\infty\over
n(\beta-\alpha)  +
 2\Vert\ffi_\epsi\Vert_\infty}\cdot
\end{eqnarray*}}%
Then, $\Phi_{n,\epsi}\in C^\infty(\RR)$, $0<\Phi_{n,\epsi}'(r)\lii 1$ for all $r\in\RR$, and  $\{\Phi_{n,\epsi}'\}_{n\in\NN}$ converges uniformly to 1 in $\RR$. Observe  that
{\setlength\arraycolsep{0.5pt}
\begin{eqnarray}\label{Phinrest}
&& {\Phi_{n,\epsi}}|_{[\alpha-\Vert\ffi_\epsi\Vert_\infty/
n, \beta + \Vert\ffi_\epsi\Vert_\infty/n]}:
\Bigg[\alpha-{\Vert\ffi_\epsi\Vert_\infty\over
n}, \beta + {\Vert\ffi_\epsi\Vert_\infty\over
n}\Bigg]\to[\alpha,\beta]
\end{eqnarray}}%
defines a projection of $\big[\alpha-\Vert\ffi_\epsi\Vert_\infty/ n, \beta + \Vert\ffi_\epsi\Vert_\infty/n\big]$ onto $[\alpha,\beta]$.

Fix $\delta_0\in(0,1/2)$, and consider the nearest point  projection $\Pi:y\in B(v(x_0),\delta_0)\mapsto \frac{y}{|y|}\in S^2$ of $B(v(x_0),\delta_0)$ onto $S^2$, which defines a $C^\infty$ mapping.
Let
{\setlength\arraycolsep{0.5pt}
\begin{eqnarray*}
&&a_\epsi:=\max\Big\{2 + 2|\grad u(x_0)| + \Vert\grad\ffi_\epsi\Vert_\infty, (\Vert\grad\Pi\Vert_{\infty} + 1)(2 + 2|\grad v(x_0)| + \Vert\grad\psi_\epsi\Vert_\infty)\Big\}, \\
&&b_\epsi:=1 + |\grad v(x_0)| + \Vert\grad\psi_\epsi\Vert_\infty.
\end{eqnarray*}}%
In view of the continuity properties of $f$ and the regularity of $\Pi$, we can find $\ell_\epsi\in (0,1)$  such that
{\setlength\arraycolsep{0.5pt}
\begin{eqnarray}\label{bycontofff}
&& |\xi_1|, |\xi_2|,|\eta_1|,|\eta_2|\lii
a_\epsi,\, |\xi_1-\xi_2|,
|\eta_1-\eta_2|\lii \ell_\epsi
 \,\Rightarrow\, |f(u(x_0),v(x_0),
 \xi_1,\eta_1)-
f(u(x_0),v(x_0),\xi_2,\eta_2)|
\lii\epsi, \qquad\quad
\end{eqnarray}}%
and there exists $\delta_\epsi\in(0,\delta_0)$ such that
\begin{eqnarray}\label{byregofPi}
s_1, s_2\in B(v(x_0),\delta_\epsi)\Rightarrow |\grad\Pi(s_1)-\grad\Pi(s_2)|\lii{\ell_\epsi\over2b_\epsi}\cdot
\end{eqnarray}
Let $\{\varsigma_k\}_{k\in\NN}$ be a decreasing sequence of positive real numbers such that $B(x_0,2\varsigma_k)\subset\Omega$ and $|Du|(\partial B(x_0,\varsigma_k))=|Dv|(\partial B(x_0,\varsigma_k))=0$ for all $k\in\NN$. Let $\{\rho_n\}_{n\in\NN}$ be the sequence of standard mollifiers defined in \eqref{smoothmoll} for $\delta=1/n$. Choose $n_0=n_0(x_0)\in\NN$ such that for all $n\gii n_0,$ we have $B(x_0,2\varsigma_1)\subset \{x\in\Omega\!: \dist(x,\partial\Omega)>1/n\}$. For $n\gii n_0,$ we define (see \eqref{defmollfct})
{\setlength\arraycolsep{0.5pt}
\begin{eqnarray*}
&& u_n(x):= u*\rho_n,
\quad v_n:= v * \rho_n.
\end{eqnarray*}}%
Then (see Lemma~\ref{lemmaAMT}-{\sl
i)}), for all $k\in\NN$, $u_n\in W^{1,1}(B(x_0,\varsigma_k);[\alpha,\beta])\cap C^\infty(\overline{B(x_0,\varsigma_k)})$ converges
strictly to \(u\) in 
 $BV(B(x_0,\varsigma_k))$, and  $v_n\in W^{1,1}(B(x_0,\varsigma_k);\overline{B(0,1)})\cap C^\infty(\overline{B(x_0,\varsigma_k)};\RR^3)
$ converges
strictly to \(v\) in 
 $BV(B(x_0,\varsigma_k);\RR^3)$ as \(n\to \infty\);
i.e.,     for all $k\in\NN$,  $u_n\weaklystar u$ weakly-$\star$ in 
 $BV(B(x_0,\varsigma_k))$,  $v_n\weaklystar v$ weakly-$\star$ in $BV(B(x_0,\varsigma_k);\RR^3)$, as $n\to\infty$, and 
\begin{eqnarray}\label{approumanif}
\lim_{n\to\infty}\int_{B(x_0,\varsigma_k)} |\grad u_n|\,\dx = |Du|(B(x_0,\varsigma_k)), \qquad \lim_{n\to\infty}\int_{B(x_0,\varsigma_k)} |\grad v_n|\,\dx = |Dv|(B(x_0,\varsigma_k)).
\end{eqnarray}
Without loss of generality, we  assume that $n_0=1$. Also, in what follows, $C_\epsi$ represents a positive constant depending on $\epsi$ but independent of $n$ and $k$ and  whose value may change from one instance to another.

\underbar{Step 1.} We construct admissible sequences for $\calF(u,v;B(x_0,\varsigma_k))$.

Let $v_{n,k}:= \pi_{y_{n,k}}\circ v_n\in W^{1,1}(B(x_0,\varsigma_k);S^2)\cap
C^\infty(\overline{B(x_0,\varsigma_k)};\RR^3)$,
where $y_{n,k}\in B(0,1/2)$ is given by Lemma~\ref{seqonman}
applied to $B(x_0;\varsigma_k)$, $v_n$, and \(A^\epsi_{n,k}:=\{x\in B(x_0,\varsigma_k)\!:\,
\dist(v_n(x),S^2)> \delta_\epsi/2\}\), so that
{\setlength\arraycolsep{0.5pt}%
\begin{eqnarray}\label{inAenk}
&&\int_{A^\epsi_{n,k}} |\grad v_{n,k}(x)|\,\dx
\lii C_\star \int_{A^\epsi_{n,k}} |\grad v_n (x)
| \,\dx.
\end{eqnarray}}%
Because \(\frac{\delta_\epsi}{2} \lii \frac{1}{4}\),
 \(|v_n(x) - y_{n,k}| \gii
\frac{1}{4}\) whenever   \(x\in B(x_0, \varsigma_k) \backslash
A^\epsi_{n,k}\). Thus, by \eqref{estGradHess}, for   \(x\in B(x_0, \varsigma_k) \backslash
A^\epsi_{n,k}\), we have
{\setlength\arraycolsep{0.5pt}%
\begin{eqnarray}\label{outsideAenk1}
&&|\grad v_{n,k}(x)| = |\grad \pi_{y_{n,k}}(v_n(x))
\grad v_n(x)|  \lii   4 \bar C |\grad v_n(x)|.
\end{eqnarray}}%
Moreover, by \eqref{gradprojonS} and \eqref{vinstwo},
\(\grad \pi_{y_{n,k}}(v(x_0))
\grad v(x_0) = \grad v(x_0)\).
Hence, \eqref{estGradHess},  \eqref{uniflipballtf},
and the inequality \(|v(x_0) - y_{n,k}| \gii
\frac{1}{2}\)  yield,   for   \(x\in B(x_0, \varsigma_k) \backslash
A^\epsi_{n,k}\) and \(\calC:=\max \{ 2 \overline C + 2\bar
C, \overline C |\grad v(x_0)| 
\}, \)
{\setlength\arraycolsep{0.5pt}%
\begin{eqnarray}\label{outsideAenk2}
&& |\grad v_{n,k}(x) - \grad v(x_0)| \nonumber\\
&&\quad \lii |\grad
\pi_{y_{n,k}}(v_n(x))
\grad v_n(x) - \grad \pi_{y_{n,k}}(v(x_0))
\grad v_n(x)| + | \grad \pi_{y_{n,k}}(v(x_0))
\grad v_n(x)  - \grad \pi_{y_{n,k}}(v(x_0))
\grad v(x_0)| \nonumber\\
&&\quad \lii \overline C |v_n(x) -v(x_0)| |\grad
v_n(x)| + 2 \bar C|\grad v_{n}(x) - \grad v(x_0)| \nonumber\\
&&\quad \lii \calC \big( |\grad v_{n}(x) - \grad v(x_0)| + |v_n(x) -v(x_0)| \big).
\end{eqnarray}}%
We claim that
\begin{eqnarray}\label{approvmanif}
\lim_{n\to\infty}\int_{B(x_0,\varsigma_k)}|v_{n,k}(x)
- v(x)|\,\dx=0.
\end{eqnarray}
In fact, using the
condition $v(x)\in S^2$ for $\calL^2$-\aev\ $x\in\Omega$
and the convergence $v_n\to v$
in $L^1(B(x_0,\varsigma_k);\RR^3)$ as $n\to\infty$, we get
{\setlength\arraycolsep{0.5pt}
\begin{eqnarray*}
&& \limsup_{n\to\infty} \calL^2(A^\epsi_{n,k})
\lii \limsup_{n\to\infty} {2\over\delta_\epsi}\int_{A^\epsi_{n,k}}
\dist (v_n(x), S^2)\,\dx
\lii \limsup_{n\to\infty}  {2\over\delta_\epsi}\int_{B(x_0,\varsigma_k)}
|v_n(x)- v(x)|\,\dx = 0,
\end{eqnarray*}}%
and so
\begin{eqnarray}\label{measAnkzero}
\lim_{n\to\infty} \calL^2(A^\epsi_{n,k}) =0.
\end{eqnarray}
Observing that 
{\setlength\arraycolsep{0.5pt}
\begin{eqnarray*}
\int_{B(x_0,\varsigma_k)}|v_{n,k}(x)
- v(x)|\,\dx&&= \int_{A^\epsi_{n,k}}|v_{n,k}(x)
- v(x)|\,\dx + \int_{B(x_0,\varsigma_k)\backslash
A^\epsi_{n,k}}|v_{n,k}(x)
- v(x)|\,\dx\cr
&&\lii  2 \calL^2(A^\epsi_{n,k})
+\overline C \int_{B(x_0,\varsigma_k)}|v_n(x)
- v(x)|\,\dx,
\end{eqnarray*}}%
where we used \eqref{projonSisId}
and \eqref{uniflipballtf},
 invoking again   the convergence $v_n\to v$ in $L^1
(B(x_0,\varsigma_k);\RR^3)$  as $n\to\infty$ and
\eqref{measAnkzero}, we obtain
\eqref{approvmanif}.

Let $\zeta_1\in C_c^\infty(\RR;[0,1])$ and $\zeta_2\in C_c^\infty(\RR^3;[0,1])$ be  cut-off functions such that $\Vert\zeta_1'\Vert_\infty\lii 2/\delta_\epsi$, $\Vert\grad\zeta_2\Vert_\infty\lii 2/\delta_\epsi$, and
{\setlength\arraycolsep{0.5pt}
\begin{eqnarray*}
&&\zeta_1(r)=1 {\hbox{ if }} r\in \Big(-{\delta_\epsi\over 4},{\delta_\epsi\over 4}\Big),\quad  \zeta_1(r)=0 {\hbox{ if }} r\not\in \Big(-{\delta_\epsi\over 2},{\delta_\epsi\over 2}\Big), \\
&&\zeta_2(s)=1 {\hbox{ if }} s\in B\Big(0,{\delta_\epsi\over 4}\Big),\quad  \zeta_2(s)=0 {\hbox{ if }} s\not\in B\Big(0,{\delta_\epsi\over 2}\Big).
\end{eqnarray*}}%
Set 
{\setlength\arraycolsep{0.5pt}
\begin{eqnarray*}
&&u_{n,k}^\epsi(x):= u_n(x) + {1\over n}\,\zeta_1(u_n(x) - u(x_0))\,\ffi_\epsi(nx),\quad x\in B(x_0,\varsigma_k), \\
&&v_{n,k}^\epsi(x):= v_{n,k}(x) + {1\over n}\,\zeta_2(v_{n,k}(x) - v(x_0))\,\psi_\epsi(nx),\quad x\in B(x_0,\varsigma_k).
\end{eqnarray*}}%
Finally, for $n\in\NN$ such that $\displaystyle n>{2\over\delta_\epsi}\max\big\{\Vert\ffi_\epsi\Vert_\infty,\Vert\psi_\epsi\Vert_\infty\big\}$ and for $x\in B(x_0,\varsigma_k)$, define
{\setlength\arraycolsep{0.5pt}
\begin{eqnarray*}
&& \bar u_{n,k}^\epsi(x)
:=\Phi_{n,\epsi}(u_{n,k}^\epsi(x))
\end{eqnarray*}}%
and
\begin{eqnarray*}
\bar v_{n,k}^\epsi(x):=
\begin{cases} \displaystyle
v_{n,k}(x) & \hbox{ if } \displaystyle |v_{n,k}(x)-v(x_0)|\gii {\delta_\epsi\over 2},\\ \\
\displaystyle\Pi(v_{n,k}^\epsi(x)) & \hbox{ if } \displaystyle |v_{n,k}(x)-v(x_0)|< {\delta_\epsi\over 2}\cdot
\end{cases}
\end{eqnarray*}

By \eqref{Phinrest}, we have
  $\bar u_{n,k}^\epsi\in W^{1,1}(B(x_0,\varsigma_k);
 [\alpha, \beta])$  and $\bar v_{n,k}^\epsi\in W^{1,1}(B(x_0,\varsigma_k);S^2)$. 
We claim that
\begin{equation}
\label{convweakstarbars}
\begin{aligned}
& \hbox{$\{\bar u_{n,k}^\epsi\}_{n\in\NN}$ weakly-$\star$
converges to $u$ in 
$BV(B(x_0,\varsigma_k))$},\\
& \hbox{$\{\bar v_{n,k}^\epsi
\}_{n\in\NN}$ weakly-$\star$ converges to $v$
in $BV(B(x_0,\varsigma_k);\RR^3)$}.
\end{aligned}
\end{equation}
In fact, we have 
{\setlength\arraycolsep{0.5pt}
\begin{eqnarray*}
&& \int_{B(x_0,\varsigma_k)}
|\bar u_{n,k}^\epsi(x)
- u(x)|\,\dx
=  \int_{B(x_0,\varsigma_k)}\Big|
{n(\beta-\alpha)
u_{n,k}^\epsi(x) + (\beta
+ \alpha)\Vert\ffi_\epsi\Vert_\infty
\over
n(\beta-\alpha)  + 2\Vert\ffi_\epsi
\Vert_\infty}-
u(x)\Big|\,\dx \\
&&\quad\lii {n(\beta-\alpha)
\over n(\beta-\alpha)
 + 2\Vert\ffi_\epsi\Vert_\infty}
 \int_{B(x_0,\varsigma_k)}
|u_{n,k}^\epsi(x) - u(x)|\,\dx + \Big|{n(\beta-\alpha)
\over n(\beta-\alpha)
 + 2\Vert\ffi_\epsi\Vert_\infty}
- 1\Big|\int_{B(x_0,\varsigma_k)}
|u(x)|\,\dx\\
&&\qquad  +\, {(\beta +
\alpha)\Vert\ffi_\epsi\Vert_\infty
\over
n(\beta-\alpha)  + 2\Vert\ffi_\epsi
\Vert_\infty}\calL^2({B(x_0,
\varsigma_k)})\\
&&\quad\lii \int_{B(x_0,\varsigma_k)}
|u_n(x) - u(x)|\,\dx
+ {\Vert\ffi_\epsi\Vert_\infty\over
n}\calL^2({B(x_0,\varsigma_k)}) + {2\Vert\ffi_\epsi\Vert_\infty
\over n(\beta-\alpha)
 + 2\Vert\ffi_\epsi\Vert_\infty}
\int_{B(x_0,\varsigma_k)}
|u(x)|\,\dx\\
&&\qquad +\, {(\beta + \alpha)\Vert\ffi_\epsi
\Vert_\infty\over
n(\beta-\alpha)  + 2
\Vert\ffi_\epsi\Vert_\infty}
\calL^2({B(x_0,\varsigma_k)}).
\end{eqnarray*}}%
Letting $n\to\infty$, taking into account that $\{u_n\}_{n\in\NN}$ converges to $u$ in $L^1({B(x_0,\varsigma_k)})$, we conclude that $\{\bar u_{n,k}^\epsi\}_{n\in\NN}$ converges to $u$ in $L^1({B(x_0,\varsigma_k)})$. Furthermore, using the fact that $0<\Phi_{n,\epsi}'(r)\lii 1$ for all $r\in\RR$, we obtain for $\calL^2$-\aev\ $x\in{B(x_0,\varsigma_k)}$,
{\setlength\arraycolsep{0.5pt}
\begin{eqnarray}\label{estbaruf}
|\grad\bar u_{n,k}^\epsi(x)|&&=\big|\Phi'_{n,\epsi}(u_{n,k}^\epsi(x))\grad
u_{n,k}^\epsi(x)\big|
\nonumber\\
&&\lii \Big|\grad u_n(x)
+ \zeta_1(u_n(x) - u(x_0))\grad\ffi_\epsi(nx)
+ {1\over n}\ffi_\epsi(nx)\zeta_1'(u_n(x)
- u(x_0))\grad u_n(x)\Big|
\nonumber\\
&&\lii |\grad u_n(x) |
+ \Vert\grad\ffi_\epsi\Vert_\infty
+ {\Vert\ffi_\epsi\Vert_\infty\over
n}\Vert\zeta_1'\Vert_\infty
|\grad u_n(x)| \nonumber\\
&&\lii 2|\grad u_n(x)
| + \Vert\grad\ffi_\epsi\Vert_\infty,
\end{eqnarray}}%
provided that $n>2\Vert\ffi_\epsi\Vert_\infty/\delta_\epsi$,
where we used the fact that \(\Vert\zeta_1'\Vert_\infty\lii 2/\delta_\epsi\). 
This, together with \eqref{approumanif} and the convergence in $L^1({B(x_0,\varsigma_k)})$ proved above, allows us to conclude that $\{\bar u_{n,k}^\epsi\}_{n\in\NN}$ is a bounded sequence in $W^{1,1}({B(x_0,\varsigma_k)})$ weakly-$\star$ converging to $u$ in $BV({B(x_0,\varsigma_k)})$. We observe further that in view of \eqref{estbaruf} we have that
\begin{eqnarray}\label{estgradbaru}
|\grad\bar u_{n,k}^\epsi(x)|\lii C_\epsi(1+|\grad u_n(x) - \grad u(x_0)|)
\end{eqnarray}
for $\calL^2$-\aev\ $x\in{B(x_0,\varsigma_k)}$.

We now turn to the sequence $\{\bar v_{n,k}^\epsi\}_{n\in\NN}$. We start by proving  that $\bar v_{n,k}^\epsi\to v$ in $L^1(B(x_0,\varsigma_k);\RR^3)$  as $n\to\infty$.
Because $\Pi(v_{n,k}(\cdot))=v_{n,k}(\cdot)$  in $\{x\in{B(x_0,\varsigma_k)}\!:\, |v_{n,k}(x)-v(x_0)|
 <\delta_\epsi/2\}$, it follows that
{\setlength\arraycolsep{0.5pt}
\begin{eqnarray*}
 \int_{B(x_0,\varsigma_k)}
|\bar v_{n,k}^\epsi(x)
- v(x)|\,\dx &&= \int_{\{x\in{B(x_0,\varsigma_k)}:
|v_{n,k}(x)-v(x_0)|\gii{\delta_\epsi\over2}\}}
|v_{n,k}(x) - v(x)|\,\dx\\
&&\quad + \int_{\{x\in{B(x_0,\varsigma_k)}: 
|v_{n,k}(x)-v(x_0)|<{\delta_\epsi\over2}\}}
|\Pi(v_{n,k}^\epsi(x)) - \Pi(v_{n,k}(x)
+ \Pi(v_{n,k}(x)) -
v(x)|\,\dx\\
&&\lii\int_{B(x_0,\varsigma_k)}
|v_{n,k}(x) - v(x)|\,\dx
+ \Vert\Pi\Vert_{1,\infty}\int_{B(x_0,\varsigma_k)}
|v_{n,k}^\epsi(x) - v_{n,k}(x)|\,\dx
\\
&&\lii \int_{B(x_0,\varsigma_k)}
|v_{n,k}(x) - v(x)|\,\dx
+ \Vert\Pi\Vert_{1,\infty}{\Vert\psi_\epsi\Vert_\infty\over
n}\calL^2({B(x_0,\varsigma_k)}),
\end{eqnarray*}}%
which, together with \eqref{approvmanif}, implies the convergence in $L^1({B(x_0,\varsigma_k)};\RR^3)$ of $\{\bar v_{n,k}^\epsi\}_{n\in\NN}$ to $v$. To estimate the sequence $\{\grad\bar v_{n,k}^\epsi\}_{n\in\NN}$, we observe that if $|v_{n,k}(x) - v(x_0)|>\delta_\epsi/2,$ then 
\begin{eqnarray}\label{estgradbarveasy}
|\grad\bar v_{n,k}^\epsi(x)|= |\grad v_{n,k}(x)|.
\end{eqnarray}
If $|v_{n,k}(x) - v(x_0)|<\delta_\epsi/2,$ then, arguing as in \eqref{estbaruf},
 for  $n>2\Vert\psi_\epsi\Vert_\infty/\delta_\epsi,$
 we have
{\setlength\arraycolsep{0.5pt}
\begin{eqnarray}\label{estgradbarveasyI}
|\grad\bar v_{n,k}^\epsi(x)|&&
=\big|\grad\Pi(v_{n,k}^\epsi(x))\grad
v_{n,k}^\epsi(x)\big|
\nonumber\\
&&\lii \Vert\grad\Pi\Vert_\infty\Big|
\grad v_{n,k}(x) + \zeta_2(v_{n,k}(x)
- v(x_0))\grad\psi_\epsi(nx)
 + {1\over
n}\psi_\epsi(nx)\otimes\grad\zeta_2(v_{n,k}(x)
- v(x_0))\grad v_{n,k}(x)\Big|
\nonumber\\
&&\lii \Vert\grad\Pi\Vert_\infty\big(2|\grad
v_{n,k}(x) | + \Vert\grad\psi_\epsi\Vert_\infty\big).
\end{eqnarray}}%
Hence, in view of \eqref{inAenk}, \eqref{outsideAenk1}, \eqref{approumanif},  and the convergence in $L^1({B(x_0,\varsigma_k)};\RR^3)$ proved above, we infer that $\{\bar v_{n,k}^\epsi\}_{n\in\NN}$ is a bounded sequence in $W^{1,1}({B(x_0,\varsigma_k)};\RR^3)$ weakly-$\star$ converging to $v$ in $BV({B(x_0,\varsigma_k)};\RR^3)$. 
We observe further that from \eqref{estgradbarveasy} and \eqref{estgradbarveasyI}, we get
\begin{eqnarray}\label{estgradbarv}
|\grad\bar v_{n,k}^\epsi(x)|\lii C_\epsi(1+|\grad v_{n,k}(x) - \grad v(x_0)|)
\end{eqnarray}
for $\calL^2$-\aev\ $x\in{B(x_0,\varsigma_k)}$.

We have just proved that $\{\bar u_{n,k}^\epsi\}_{n\in\NN}$ and $\{\bar v_{n,k}^\epsi\}_{n\in\NN}$ are admissible sequences for $\calF(u,v;B(x_0;\varsigma_k))$, which concludes Step~1.

\underbar{Step 2.} We prove that the sequences $\{\bar u_{n,k}^\epsi\}_{n\in\NN}$ and $\{\bar v_{n,k}^\epsi\}_{n\in\NN}$ constructed in Step~1 satisfy\begin{eqnarray}\label{acp2}
{\d\calF(u,v;\cdot)\over\d\calL^2}(x_0)\lii \limsup_{k\to\infty}\limsup_{n\to\infty}\dashint_{B(x_0,\varsigma_k)} f(u(x_0), v(x_0), \grad \bar u_{n,k}^\epsi(x), \grad \bar v_{n,k}^\epsi(x))\,\dx.
\end{eqnarray}

By Step~1, \eqref{locrelaxfct}, and \eqref{rnforcalF}, we obtain
{\setlength\arraycolsep{0.5pt}
\begin{eqnarray*}
{\d\calF(u,v;\cdot)\over\d\calL^2}
(x_0) &&= \lim_{k\to\infty}
{\calF(u,v;B(x_0,\varsigma_k))
\over\calL^2(B(x_0,\varsigma_k))}
\nonumber\\
&&\lii\liminf_{k\to\infty}
\liminf_{n\to\infty}
\dashint_{B(x_0,\varsigma_k)}
f(\bar u_{n,k}^\epsi(x),
\bar v_{n,k}^\epsi(x),
\grad \bar u_{n,k}^\epsi(x),
\grad \bar v_{n,k}^\epsi(x))\,\dx.
\end{eqnarray*}}%
We claim that 
{\setlength\arraycolsep{0.5pt}
\begin{eqnarray}\label{justdepongrad}
&&\liminf_{k\to\infty}
\liminf_{n\to\infty}
\dashint_{B(x_0,\varsigma_k)}
f(\bar u_{n,k}^\epsi(x),
\bar v_{n,k}^\epsi(x),
\grad \bar u_{n,k}^\epsi(x),
\grad \bar v_{n,k}^\epsi(x))\,\dx
\nonumber\\
&&\quad\lii \limsup_{k\to\infty}\limsup_{n\to\infty}\dashint_{B(x_0,\varsigma_k)}
f(u(x_0), v(x_0), \grad
\bar u_{n,k}^\epsi(x),
\grad \bar v_{n,k}^\epsi(x))\,\dx,
\end{eqnarray}}%
from which \eqref{acp2} follows.
Using the definition of $f$ (see \eqref{defoff}), \eqref{estgradbaru},  and \eqref{estgradbarv}, we get
{\setlength\arraycolsep{0.5pt}
\begin{eqnarray}\label{justdepongradst}
&& \int_{B(x_0,\varsigma_k)}
\big|f(\bar u_{n,k}^\epsi(x),
\bar v_{n,k}^\epsi(x),
\grad \bar u_{n,k}^\epsi(x),
\grad \bar v_{n,k}^\epsi(x))
- f(u(x_0), v(x_0), \grad
\bar u_{n,k}^\epsi(x),
\grad \bar v_{n,k}^\epsi(x))\big|
\,\dx\nonumber\\
&&\quad\lii \int_{B(x_0,\varsigma_k)}
\big(|\bar u_{n,k}^\epsi(x)
- u(x_0)||\grad \bar
v_{n,k}^\epsi(x)| + |\bar
v_{n,k}^\epsi(x) - v(x_0)||\grad
\bar u_{n,k}^\epsi(x)|\big)\ \,\dx
\nonumber\\
&&\quad \lii  C_\epsi\bigg[
\int_{B(x_0,\varsigma_k)}
\big( |\bar u_{n,k}^\epsi(x)
- u(x_0)| + |\bar v_{n,k}^\epsi(x)
- v(x_0)|\big)\ \,\dx \nonumber\\
&&\qquad + \int_{B(x_0,
\varsigma_k)}
|\bar u_{n,k}^\epsi(x)
- u(x_0)||\grad v_{n,k}(x)|\,\dx+
\int_{B(x_0,\varsigma_k)}
|\bar v_{n,k}^\epsi(x)
- v(x_0)||\grad  u_n(x)|\,\dx\bigg].
\end{eqnarray}}%
By \eqref{convweakstarbars}, \eqref{lebu},
and \eqref{lebv},  we
deduce that
{\setlength\arraycolsep{0.5pt}
\begin{eqnarray}\label{justdepongradnd}
&& \lim_{k\to\infty}\lim_{n\to\infty} 
\dashint_{B(x_0,\varsigma_k)}
\big(|\bar u_{n,k}^\epsi(x)
- u(x_0)| + |\bar v_{n,k}^\epsi(x)
- v(x_0)|\big)\,\dx \nonumber\\
&&\quad = \lim_{k\to\infty}\dashint_{B(x_0,\varsigma_k)}
\big(|u(x) - u(x_0)| + |v(x)
- v(x_0)|\big)\,\dx=0.
\end{eqnarray}}%
We now estimate the  last two integrals in 
\eqref{justdepongradst}.
Since \(|\bar u^\epsi_{n,k}(\cdot) - u(x_0)|
\lii 2\beta\), \eqref{inAenk}, and \eqref{outsideAenk1},
we obtain
{\setlength\arraycolsep{0.5pt}%
\begin{eqnarray}\label{justdepongradfith}
&&\int_{B(x_0,\varsigma_k)}
|\bar u_{n,k}^\epsi(x)
- u(x_0)||\grad  v_{n,k}(x)|\,\dx\nonumber\\
&&\quad \lii 2\beta
  \int_{A^\epsi_{n,k}}
| \grad v_{n,k}(x)|\,\dx + 4\bar C 
\int_{B(x_0,\varsigma_k)\backslash
A^\epsi_{n,k}}  |\bar u_{n,k}^\epsi(x) - u(x_0)||
\grad v_n(x)|\,\dx \nonumber\\
&&\quad \lii 2\beta C_\star
  \int_{A^\epsi_{n,k}}
\big(| \grad v_{n}(x) - \grad v(x_0)| + |\grad v(x_0)| \big) \,\dx \nonumber\\
&&\qquad  + 4\bar C 
\int_{B(x_0,\varsigma_k)\backslash
A^\epsi_{n,k}} \big( 2\beta |
\grad v_n(x) - \grad v(x_0)| + |\bar u_{n,k}^\epsi(x) - u(x_0)||
\grad v(x_0)|\big)\,\dx
\nonumber\\
&&\quad\lii  \tilde C\bigg(
\int_{B(x_0,\varsigma_k)}
|\grad v_n(x) - \grad
v(x_0)|\,\dx +   \calL^2(A^\epsi_{n,k})
+   \int_{B(x_0,\varsigma_k)}
 |\bar u_{n,k}^\epsi(x)
- u(x_0)|\,\dx\bigg),
\end{eqnarray}}%
where $\tilde C:=\max\{2\beta(C_\star + 4\bar
C), 2\beta C_\star |\grad v(x_0)|, 4\bar C|\grad
v(x_0)|\}$.
 Because $v_n=v*\rho_n$ and $|Dv|(\partial B(x_0,\varsigma_k))=0$, we have 
\begin{eqnarray}\label{justdepongradsith}
\lim_{n\to\infty}\int_{B(x_0,\varsigma_k)} |\grad v_n(x) - \grad v(x_0)|\,\dx = \int_{B(x_0,\varsigma_k)} |\grad v(x) - \grad v(x_0)|\,\dx + |D^s v|(B(x_0,\varsigma_k)).
\end{eqnarray}
Recalling that $\{\bar u_{n,k}^\epsi\}_{n\in\NN}$ converges to $u$ in $L^1({B(x_0,\varsigma_k)})$
(see \eqref{convweakstarbars}), from  \eqref{justdepongradfith}, \eqref{justdepongradsith}, \eqref{measAnkzero}, \eqref{lebgradv}, \eqref{lspiszerov}, and \eqref{lebu}, we deduce that
\begin{eqnarray}\label{justdepongradteth}
\lim_{k\to\infty}\lim_{n\to\infty}\dashint_{B(x_0,\varsigma_k)} |\bar u_{n,k}^\epsi(x) - u(x_0)||\grad v_{n,k}(x)|\,\dx  =0.
\end{eqnarray}

Finally, we estimate the last integral in \eqref{justdepongradst}. We have that
{\setlength\arraycolsep{0.5pt}
\begin{eqnarray}\label{justdepongradseth}
&& \int_{B(x_0,\varsigma_k)}
|\bar v_{n,k}^\epsi(x)
- v(x_0)||\grad  u_n(x)|\,\dx
\nonumber\\
&&\quad\lii 2 \int_{B(x_0,\varsigma_k)}
|\grad  u_n(x) - \grad
 u(x_0)|\,\dx + \int_{B(x_0,\varsigma_k)}
|\bar v_{n,k}^\epsi(x)
- v(x_0)||\grad  u(x_0)|\,\dx.
\end{eqnarray}}%
Arguing as above, an equality for $u$ similar to that in \eqref{justdepongradsith} holds; that is,
\begin{eqnarray}\label{justdepongradeith}
\lim_{n\to\infty}\int_{B(x_0,\varsigma_k)} |\grad u_n(x) - \grad u(x_0)|\,\dx = \int_{B(x_0,\varsigma_k)} |\grad u(x) - \grad u(x_0)|\,\dx + |D^s u|(B(x_0,\varsigma_k)).
\end{eqnarray}
Recalling that $\{\bar v_{n,k}^\epsi\}_{n\in\NN}$ converges to $v$ in $L^1({B(x_0,\varsigma_k)};\RR^3)
$ (see \eqref{convweakstarbars}), from  \eqref{justdepongradseth}, \eqref{justdepongradeith}, \eqref{lebgradu}, \eqref{lspiszerou}, and \eqref{lebv}, we deduce that
\begin{eqnarray}\label{justdepongradnith}
\lim_{k\to\infty}\lim_{n\to\infty}\dashint_{B(x_0,\varsigma_k)} |\bar v_{n,k}^\epsi(x) - v(x_0)||\grad  u_n(x)|\,\dx  =0.
\end{eqnarray}
Hence, \eqref{justdepongrad} follows from \eqref{justdepongradst}, \eqref{justdepongradnd}, \eqref{justdepongradteth}, and \eqref{justdepongradnith}.

\underbar{Step 3.} We conclude the proof of Lemma~\ref{abscontpartCalF}.

Set
{\setlength\arraycolsep{0.5pt}
\begin{eqnarray*}
&& z_n^\epsi (x):= \grad
u(x_0)\, x + {1\over
n}\ffi_\epsi(nx),\quad
w_n^\epsi (x):= \grad
v(x_0)\, x + {1\over
n}\psi_\epsi(nx),
\end{eqnarray*}}%
and observe that
{\setlength\arraycolsep{0.5pt}
\begin{eqnarray*}
&& |\grad z_n^\epsi (x)|\lii
|\grad u(x_0)| + \Vert\grad\ffi_\epsi\Vert_\infty\lii
a_\epsi, \quad |\grad
w_n^\epsi (x)|\lii |\grad
v(x_0)| + \Vert\grad\psi_\epsi\Vert_\infty\lii
a_\epsi.
\end{eqnarray*}}%
Additionally, let
{\setlength\arraycolsep{0.5pt}
\begin{eqnarray*}
&& \gamma_\epsi:= 
{\ell_\epsi\over
2(\Vert\grad\Pi\Vert_\infty+1)},
\end{eqnarray*}}%
and define
{\setlength\arraycolsep{0.5pt}
\begin{eqnarray*}
&&B_{u}:=\Big\{x\in B(x_0,\varsigma_k)\!:\, |u_n(x) - u(x_0)|<{\delta_\epsi\over 4}\Big\}, \qquad
 B_{\grad u}:=\Big\{x\in B(x_0,\varsigma_k)\!:\, |\grad u_n(x) - \grad u(x_0)|<\gamma_\epsi\Big\},\\
&& B_{v}:=\Big\{x\in B(x_0,\varsigma_k)\!:\, |v_{n,k}(x) - v(x_0)|<{\delta_\epsi\over 4}\Big\}, \qquad B_{\grad v}:=\Big\{x\in B(x_0,\varsigma_k)\!:\, |\grad v_{n,k}(x) - \grad v(x_0)|<\gamma_\epsi\Big\}.
\end{eqnarray*}}%
If $x\in B_u\cap B_{\grad u}\cap B_{v}\cap B_{\grad v}$, then $\zeta_1(u_n(x) - u(x_0))=1$, $\zeta_2(v_{n,k}(x) - v(x_0))=1$, and, by \eqref{estbaruf}
and \eqref{estgradbarveasyI},
for   $\displaystyle n>2/\delta_\epsi\max\big 
\{\Vert\ffi_\epsi\Vert_\infty,\Vert\psi_\epsi
\Vert_\infty\big\}$, we have
{\setlength\arraycolsep{0.5pt}
\begin{eqnarray*}
&&|\grad\bar u_{n,k}^\epsi(x)|\lii 2\gamma_\epsi + 2|\grad u(x_0)| + \Vert\grad\ffi_\epsi\Vert_\infty\lii a_\epsi, \\
&&|\grad\bar v_{n,k}^\epsi(x)|\lii \Vert\grad\Pi\Vert_\infty \big(2\gamma_\epsi + 2|\grad v(x_0)| + \Vert\grad\psi_\epsi\Vert_\infty\big)\lii a_\epsi.
\end{eqnarray*}}%
Moreover, using the fact that $\Phi'_{n,\epsi}(r)= n(\beta-\alpha)
/(n(\beta-\alpha)  + 2\Vert\ffi_\epsi\Vert_\infty)\in(0,1]$ for all $r\in\RR$, we get
{\setlength\arraycolsep{0.5pt}
\begin{eqnarray*}
|\grad\bar u_{n,k}^\epsi(x) - \grad z_n^\epsi (x)| &&=\Big|\Phi_{n,\epsi}'\Big(u_n(x)+{1\over n}\ffi_\epsi(nx)\Big)\big(\grad u_n(x)\mp \grad u(x_0) + \grad\ffi_\epsi(nx)\big)- \big(\grad u(x_0) + \grad
\ffi_\epsi(nx)\big)\Big|\\
&&\lii |\grad u_n(x) - \grad u(x_0)| + \Big|\Phi_{n,\epsi}'\Big(u_n(x)+{1\over n}\ffi_\epsi(nx)\Big) - 1\Big||\grad u(x_0) + \grad\ffi_\epsi(nx)|\\
&&\lii \gamma_\epsi + {2\Vert\ffi_\epsi\Vert_\infty \over n(\beta-\alpha)  + 2\Vert\ffi_\epsi\Vert_\infty} \big(|\grad u(x_0)| + \Vert\grad\ffi_\epsi
\Vert_\infty\big)\\
&&\lii {\ell_\epsi\over 2} + {\ell_\epsi\over 2} = \ell_\epsi,
\end{eqnarray*}}%
provided that
{\setlength\arraycolsep{0.5pt}
\begin{eqnarray*}
&& n> {2\Vert\ffi_\epsi\Vert_\infty\big(2|\grad
u(x_0)| + 2\Vert\grad\ffi_\epsi\Vert_\infty
 -\ell_\epsi\big)\over
(\beta-\alpha)\ell_\epsi}\cdot
\end{eqnarray*}}%
Next, using the second condition in \eqref{vinstwo}
and the equality $\grad \Pi(v(x_0))\grad\psi_\epsi(\cdot)= \grad\psi_\epsi(\cdot)$, which holds for $\calL^2$-\aev\ in $\RR^2$ since $\grad\psi_\epsi(\cdot)\in [T_{v(x_0)}(S^2)]^2$ for $\calL^2$-\aev\ in $\RR^2$,  we obtain 
{\setlength\arraycolsep{0.5pt}
\begin{eqnarray*}
&& |\grad\bar v_{n,k}^\epsi(x)
- \grad w_n^\epsi (x)|\\
&&\quad=\Big|\grad\Pi\Big(v_{n,k}(x)+{1\over
n}\psi_\epsi(nx)\Big)\big(\grad
v_{n,k}(x)-\grad v(x_0) + \grad v(x_0)
+ \grad\psi_\epsi(nx)\big)-
\big(\grad v(x_0) + \grad\psi_\epsi
(nx)\big)\Big|\\
&&\quad\lii \Vert\grad\Pi\Vert_\infty|
\grad
v_{n,k}(x) - \grad v(x_0)|+
\Big|\grad\Pi\Big(v_{n,k}(x)+{1\over
n}\psi_\epsi(nx)\Big)
- \grad\Pi(v(x_0))\Big||\grad
v(x_0) + \grad\psi_\epsi(nx)|\\
&&\quad\lii \Vert\grad\Pi
\Vert_\infty\gamma_\epsi
+ {\ell_\epsi\over2b_\epsi}
\big(|\grad
v(x_0)| + \Vert\grad\psi_\epsi
\Vert_\infty\big)\lii {\ell_\epsi\over
2} + {\ell_\epsi\over
2} = \ell_\epsi,
\end{eqnarray*}}%
provided that $n>2\Vert\psi_\epsi\Vert_\infty/\delta_\epsi$, because for all such $n\in\NN,$ we have $|v_{n,k}(x)+{1\over n}\psi_\epsi(nx) - v(x_0)|\lii \delta_\epsi/4 + \delta_\epsi/2<\delta_\epsi$,
and so \eqref{byregofPi} applies.

Thus, using \eqref{bycontofff}, Riemann--Lebesgue's Lemma, and \eqref{definfQQT}, in this order, we conclude that
{\setlength\arraycolsep{0.5pt}
\begin{eqnarray}\label{ubonabsst}
&& \limsup_{k\to\infty}
\limsup_{n\to\infty}{1\over
\calL^2(B(x_0,\varsigma_k))}
\int_{B_u\cap
B_{\grad u}\cap B_{v}\cap
B_{\grad v}} f\big(u(x_0),
v(x_0), \grad \bar u_{n,k}^\epsi(x),
\grad \bar v_{n,k}^\epsi(x)\big)\,
\dx \nonumber\\
&&\quad\lii \limsup_{k\to\infty}\limsup_{n\to\infty}\dashint_{B(x_0,\varsigma_k)}
f\big(u(x_0), v(x_0),
\grad u(x_0) + \grad\ffi_\epsi(nx),
\grad v(x_0) + \grad\psi_\epsi(nx)
\big)\,\dx
+ \epsi \nonumber\\
&&\quad = \int_Q f\big(u(x_0),v(x_0),
\grad
u(x_0) + \grad\ffi_\epsi(y),\grad
v(x_0)+\grad\psi_\epsi(y)\big)\,\dy
+ \epsi \nonumber\\
&&\quad \lii \calQ_T
f\big(u(x_0),v(x_0),\grad
u(x_0),\grad v(x_0)\big)
+ 2\epsi.
\end{eqnarray}}%
Next, we observe that
from \eqref{inAenk} and \eqref{outsideAenk2}, we
have, for   \( c:= \max\{C_\star + 1,\calC, (C_\star
+1) |\grad
v(x_0)| \}\), 
{\setlength\arraycolsep{0.5pt}
\begin{eqnarray}\label{gvnk-gv0}
\int_{B(x_0,\varsigma_k)}  |\grad
 v_{n,k}(x) -\grad v(x_0)| \,\dx  &&\lii (C_\star
+1)  \int_{ A^\epsi_{n,k}} \big( |\grad
 v_{n}(x) - \grad v(x_0)| + |\grad v(x_0)|
\big)
\,\dx \nonumber\\
&&\qquad + \calC\bigg( \int_{B(x_0,\varsigma_k)
 \backslash  A^\epsi_{n,k}} \big(|\grad
 v_{n}(x) - \grad v(x_0)| + |v_n(x) - v(x_0)|
\,\big)\dx \bigg) \nonumber \\
&& \quad \lii c\bigg(
\int_{B(x_0,\varsigma_k)
 } \big(|\grad
 v_{n}(x) - \grad v(x_0)| + |v_n(x) - v(x_0)|
\, \big)\dx +  \calL^2(A^\epsi_{n,k})
 \bigg).\nonumber\\
\end{eqnarray}}%
Note also that 
\begin{eqnarray}\label{ubonabsrd}
{1\over \calL^2(B(x_0,\varsigma_k))}\int_{B(x_0,\varsigma_k)\backslash
B_u} 1\,\dx \lii {4\over \delta_\epsi } \dashint_{B(x_0,\varsigma_k)}
|u_n(x) - u(x_0)|\,\dx,
\end{eqnarray}
and,  by \eqref{boundsf}, \eqref{estgradbaru}, and \eqref{estgradbarv}, 
{\setlength\arraycolsep{0.5pt}
\begin{eqnarray}\label{ubonabsnd}
&& \limsup_{k\to\infty}\limsup_{n\to\infty}{1\over
\calL^2(B(x_0,\varsigma_k))}\int_{B(x_0,\varsigma_k)\backslash
B_u} f(u(x_0), v(x_0),
\grad \bar u_{n,k}^\epsi(x),
\grad \bar v_{n,k}^\epsi(x))\,\dx
\nonumber\\
&&\quad\lii (3+\beta)
\limsup_{k\to\infty}\limsup_{n\to\infty}{1\over
\calL^2(B(x_0,\varsigma_k))}\int_{B(x_0,\varsigma_k)\backslash
B_u}  \big(|\grad \bar u_{n,k}^\epsi(x)|
+ |\grad \bar v_{n,k}^\epsi(x)|\big)\,\dx
\nonumber\\
&&\quad\lii C_\epsi \limsup_{k\to\infty}\limsup_{n\to\infty}{1\over
\calL^2(B(x_0,\varsigma_k))}\int_{B(x_0,\varsigma_k)\backslash
B_u} \big(1+ |\grad u_n(x)
- \grad u(x_0)|  + |\grad
 v_{n,k}(x) -\grad v(x_0)|\big)\,\dx.\qquad\quad
\end{eqnarray}}%
 Plugging in \eqref{gvnk-gv0} and \eqref{ubonabsrd} in \eqref{ubonabsnd}, from the convergences $u_n\to u$ in $L^1(B(x_0,\varsigma_k))$ and \(v_n\to
v\) in \(L^1(B(x_0,\varsigma_k);\RR^3)\), as \(n\to\infty\),
and from \eqref{measAnkzero},   \eqref{justdepongradsith},  \eqref{justdepongradeith},  \eqref{lebu}, \eqref{lebgradu},
\eqref{lebgradv}, \eqref{lspiszerou},  and \eqref{lspiszerov}, it follows that
\begin{eqnarray}\label{ubonabsfoth}
\lim_{k\to\infty}\lim_{n\to\infty}{1\over \calL^2(B(x_0,\varsigma_k))}\int_{B(x_0,\varsigma_k)\backslash B_u} f(u(x_0), v(x_0), \grad \bar u_{n,k}^\epsi(x), \grad \bar v_{n,k}^\epsi(x))\,\dx = 0.
\end{eqnarray}
Similarly, using in addition \eqref{approvmanif} and \eqref{lebv},
\begin{eqnarray}\label{ubonabsfith}
\lim_{k\to\infty}\lim_{n\to\infty}{1\over \calL^2(B(x_0,\varsigma_k))}\int_{B(x_0,\varsigma_k)\backslash B_v} f(u(x_0), v(x_0), \grad \bar u_{n,k}^\epsi(x), \grad \bar v_{n,k}^\epsi(x))\,\dx = 0.
\end{eqnarray}
Also, since
{\setlength\arraycolsep{0.5pt}%
\begin{eqnarray*}
&&\int_{B(x_0,\varsigma_k)\backslash B_{\grad
u}} 1\,\dx \lii {1\over \gamma_\epsi } 
\int_{B(x_0,\varsigma_k)}
|\grad u_n(x) - \grad u(x_0)|\,\dx, \\
&& 
\int_{B(x_0,\varsigma_k)\backslash B_{\grad
v}} 1\,\dx \lii {1\over \gamma_\epsi } 
\int_{B(x_0,\varsigma_k) }
|\grad v_{n,k}(x) - \grad v(x_0)|\,\dx,
\end{eqnarray*}}%
we have
{\setlength\arraycolsep{0.5pt}
\begin{eqnarray}\label{ubonabsfsith}
&& \limsup_{k\to\infty}\limsup_{n\to\infty}{1\over
\calL^2(B(x_0,\varsigma_k))}\int_{(B_u\cap
B_v)\backslash B_{\grad
u}} f(u(x_0), v(x_0),
\grad \bar u_{n,k}^\epsi(x),
\grad \bar v_{n,k}^\epsi(x))\,\dx  = 0
\end{eqnarray}}%
and
{\setlength\arraycolsep{0.5pt}
\begin{eqnarray}\label{ubonabsfseth}
&& \limsup_{k\to\infty}\limsup_{n\to\infty}{1\over
\calL^2(B(x_0,\varsigma_k))}\int_{(B_u\cap
B_v)\backslash B_{\grad
v}} f(u(x_0), v(x_0),
\grad \bar u_{n,k}^\epsi(x),
\grad \bar v_{n,k}^\epsi(x))\,\dx  = 0.
\end{eqnarray}}%
Finally, owing to \eqref{acp2},  \eqref{ubonabsst}, and  \eqref{ubonabsfoth}--\eqref{ubonabsfseth}, we conclude that
{\setlength\arraycolsep{0.5pt}
\begin{eqnarray*}
&&{\d\calF(u,v;\cdot)\over\d\calL^2}(x_0)
\lii \calQ_T f(u(x_0),v(x_0),\grad
u(x_0),\grad v(x_0))
+ 2\epsi,
\end{eqnarray*}}%
and \eqref{ubabscontpart} follows by letting $\epsi\to0^+$.
\end{proof}

\begin{lemma}\label{inff=infQQf} 
The infimum in \eqref{locrelaxfct} does not change if we replace $f$ by $\calQ_T f$.
\end{lemma}

\begin{proof} For $(u,v)\in BV(\Omega;[\alpha,\beta])\times BV(\Omega;S^2)$ and $A\in \calA(\Omega)$, set
{\setlength\arraycolsep{0.5pt}
\begin{eqnarray*}
\calQF(u,v;A)&&:=\inf\Big\{
 \liminf_{n\to+\infty}
\int_A \calQ_T f(u_n(x),
v_n(x),\grad u_n(x),\grad
v_n(x))\,\dx\!: \\
&&\quad  n\in\NN,\,
(u_n,v_n)\in W^{1,1}(A;[\alpha,\beta])\times
W^{1,1}(A;S^2),\,  u_n\to
u \hbox{ in } L^1(A),
v_n\to v \hbox{ in }
L^1(A;\RR^3)\Big\}.
\end{eqnarray*}}%
The inequality $\calQF(u,v;A)\lii \calF(u,v;A)$ follows from the fact that $\calQ_T f\lii f$. To prove the converse inequality, let \((\bar u,\bar v) \in W^{1,1}(A;[\alpha,\beta])\times
W^{1,1}(A;S^2) \). Using the growth conditions
\eqref{boundsf} satisfied by $f$, we conclude that
{\setlength\arraycolsep{0.5pt}
\begin{eqnarray*}
&&\calF(\bar u,\bar v;A)\lii \int_A f
(\bar u(x), \bar v(x),\grad \bar u(x),\grad
\bar v(x))\,\dx \lii 2\int_A |\grad \bar u(x)|\,\dx
+ (1+\beta) \int_A |\grad
\bar v(x)|\,\dx,
\end{eqnarray*}}%
which proves that $A\in \calA(\Omega)\mapsto \calF(\bar u,\bar v;A)$ is absolutely continuous with respect to the Lebesgue measure. Hence, by Lemma~\ref{abscontpartCalF}, we conclude that
\begin{eqnarray}\label{caseWoneone}
\calF(\bar u,\bar v;A)\lii \int_A \calQ_T f(\bar u(x), \bar v(x),\grad \bar u(x),\grad \bar v(x))\,\dx.
\end{eqnarray}

Fix $(u,v)\in BV(\Omega;[\alpha,\beta])\times BV(\Omega;S^2). $ Let $\{(u_n,v_n)\}_{n\in\NN}\subset W^{1,1}(A;[\alpha,\beta])\times
W^{1,1}(A;S^2)$ be such that $u_n\to u$ in $L^1(A)$ and $v_n\to v$  in $L^1(A;\RR^3)$. The sequential lower semicontinuity of $\calF(\cdot,\cdot;A)$ with respect to the strong
convergence in $L^1(A)\times L^1(A;\RR^3)$, together with \eqref{caseWoneone}, yields
{\setlength\arraycolsep{0.5pt}
\begin{eqnarray*}
&& \calF(u,v;A)\lii \liminf_{n\to\infty}
\calF(u_n,v_n;A)\lii
\liminf_{n\to\infty}
\int_A \calQ_T f(u_n(x),
v_n(x),\grad u_n(x),\grad
v_n(x))\,\dx.
\end{eqnarray*}}%
Taking the infimum over all admissible sequences, we conclude that $\calF(u,v;A)\lii \calQF(u,v;A)$.
\end{proof}

\begin{lemma}\label{cantorpartCalF}
Fix $(u,v)\in BV(\Omega;[\alpha,\beta])\times
BV(\Omega;S^2)$. Then,
\begin{eqnarray}\label{ubcantorpart}
{\d\calF(u,v;\cdot)\over\d |D^c(u,v)|}(x_0)\lii  (\calQ_T
f)^\infty\big(\tilde
u(x_0) ,\tilde v(x_0),
W^c_u(x_0), W^c_v(x_0)\big)
\end{eqnarray}
for $|D^c(u,v)|$-\aev\ $x_0\in\Omega$.
\end{lemma}

\begin{proof} Set $w:=(u,v),$ and define $\nu:= |Dw| - |D^cw|$. Let $x_0\in\Omega$ be such that
{\setlength\arraycolsep{0.5pt}
\begin{eqnarray}
&&\tilde w(x_0) = (\tilde u(x_0),\tilde v(x_0))\in [\alpha,\beta]\times S^2,\quad W^c_v(x_0)\in [T_{\tilde
v(x_0)}(S^2)]^2, \label{winmanifold} \\
&&{\d\calF(w;\cdot)\over\d |D^cw|}(x_0)  {\hbox{ exists and is finite}}, \label{rnforcalFCant} \\
&& \lim_{\epsilon\to0^+}\frac{\nu(B(x_0,\epsilon))}{|D^cw|(B(x_0,\epsilon))}=0, \label{nuCwarems}\\
&&\lim_{\epsilon\to0^+}
\dashint_{B(x_0,\epsilon)}|\tilde w(x) -\tilde w(x_0)|\,\d|D^cw|(x)=0, \label{LebPointwrtCw}\\
&&\lim_{\epsilon\to0^+}
\dashint_{B(x_0,\epsilon)}|{W^c}(x) - {W^c}(x_0)|\,\d|D^cw|(x)=0, \label{LebPointWcwrtCw}\\
&& \lim_{\epsilon\to0^+}\dashint_{B(x_0,\varsigma_k)}
\big|(\calQ \tilde
f)^\infty(\tilde
u(x),\tilde v(x),W_u^c(x),W_v^c(x))
\nonumber\\
&&\hskip35mm-(\calQ \tilde
f)^\infty(\tilde
u(x_0),\tilde v(x_0),W_u^c(x_0),W_v^c(x_0))
\,\big|\d|D^cw|(x)=0.\qquad\qquad\label{LebPointrecQtildef}
\end{eqnarray}}%
We observe that \eqref{winmanifold}--\eqref{LebPointrecQtildef} hold for $|D^cw|$-\aev\ $x_0\in\Omega$.

As in the proof of Lemma~\ref{abscontpartCalF}, let $\{\varsigma_k\}_{k\in\NN}$ be a decreasing sequence of positive real numbers such that $B(x_0,2\varsigma_k)\subset\Omega$ and $|Du|(\partial B(x_0,\varsigma_k))=|Dv|(\partial B(x_0,\varsigma_k))=0
=|Dw|(\partial B(x_0,\varsigma_k))
$ for all $k\in\NN$. Let $\{\rho_n\}_{n\in\NN}$ be the sequence of standard mollifiers defined in \eqref{smoothmoll} for $\delta=1/n$. Without
loss of generality, we
may assume that for all
$n\in\NN,$ we have $B(x_0,2\varsigma_1)\subset \{x\in\Omega\!: \dist(x,
\partial\Omega)>1/n\}$,
and  thus define (see \eqref{defmollfct})
{\setlength\arraycolsep{0.5pt}
\begin{eqnarray*}
&& u_n(x):= u*\rho_n,
\quad v_n:=  v* \rho_n, \quad  w_n:= w*\rho_n=(u_n,v_n).
\end{eqnarray*}}%
Then, $u_n\in W^{1,1}(B(x_0,\varsigma_k);[\alpha,\beta])\cap C^\infty(\overline{B(x_0,\varsigma_k)})$, $v_n\in W^{1,1}(B(x_0,\varsigma_k);\overline{B(0,1)})\cap C^\infty(\overline{B(x_0,\varsigma_k)};\RR^3)$,  $u_n\weaklystar u$ weakly-$\star$ in $BV(B(x_0,\varsigma_k))$, and $v_n\weaklystar v$ weakly-$\star$ in $BV(B(x_0,\varsigma_k);\RR^3)$, as $n\to\infty$, for all $k\in\NN$.  

Fix $\delta\in(0,1/4)$. For $n,k\in\NN$, consider the function   $v_{n,k}:= \pi_{y_{n,k}}\circ v_n\in W^{1,1}(B(x_0,\varsigma_k);S^2)\cap C^\infty(\overline{B(x_0,\varsigma_k)};\RR^3)$ with $y_{n,k}\in B(0,1/2)$ given by Lemma~\ref{seqonman} applied to $B(x_0;\varsigma_k)$,  $v_n$, and  $A_{n,k}:=\{x\in B(x_0,\varsigma_k)\!:\, \dist(v_n(x),S^2)> \delta\}$.
Then, arguing as in Step~1 of the proof of
Lemma~\ref{abscontpartCalF}, for all $
n,k\in\NN$ and $
x\in B(x_0,\varsigma_k)\backslash
A_{n,k} $,
  {\setlength\arraycolsep{0.5pt}
\begin{eqnarray}
&& \lim_{n\to\infty}\calL^2(A_{n,k})=0,\enspace \quad \int_{A_{n,k}}|\grad v_{n,k}(x)|\,\dx\lii C_\star\int_{A_{n,k}}|\grad v_n(x)|\,\dx,\enspace\label{onAnkCantor} \\
&&|\grad v_{n,k}(x)|\lii 4\bar C|\grad v_n(x)|,\quad |v_{n,k}(x) - \tilde v(x_0)|\lii \overline C |v_n(x) - \tilde v(x_0)|.\label{outsideAnkCantor}
\end{eqnarray}}%

\underbar{Step 1.} We prove that 
{\setlength\arraycolsep{0.5pt}
\begin{eqnarray}\label{ubcantor1}
&& {\d\calF(w;\cdot)\over\d
|D^cw|}(x_0) \nonumber\\
&&\quad\lii \limsup_{k\to\infty}\limsup_{n\to\infty}{1\over
|D^cw|(B(x_0,\varsigma_k))}\bigg(\int_{B(x_0,\varsigma_k)}
\calQ \tilde f(\tilde
u(x_0),\tilde v(x_0),
\grad  u_n(x), \grad
 v_{n,k}(x))\,\dx  +c\,I_{n,k}\bigg),
  \qquad\quad
\end{eqnarray}}%
where \(c\) is the constant in \eqref{QQbarfrbarrsbars}
and
\begin{eqnarray}\label{defIn}
I_{n,k}:= \int_{B(x_0,\varsigma_k)} \big(|u_n(x) - \tilde u(x_0)| + |v_{n,k}(x) - \tilde v(x_0)|\big)\big(|\grad u_n(x)| + |\grad v_{n,k}(x)|\big)\,\dx.
\end{eqnarray}

The sequence $\{(u_n,
v_{n,k})\}_{n\in\NN}$ is admissible  for $\calF(w;B(x_0;\varsigma_k))$, thus, in view of \eqref{rnforcalFCant} and Lemma~\ref{inff=infQQf}, we have
{\setlength\arraycolsep{0.5pt}
\begin{eqnarray*}
 {\d\calF(w;\cdot)\over\d|D^cw|}(x_0)
&&= \lim_{k\to\infty} {\calF(w;B(x_0,\varsigma_k))
\over|D^cw|(B(x_0,\varsigma_k))}\\
&&\lii\liminf_{k\to\infty}\liminf_{n\to\infty}{1\over
|D^cw|(B(x_0,\varsigma_k))}\int_{B(x_0,\varsigma_k)}
\calQ_Tf( u_n(x), 
v_{n,k}(x), \grad u_n(x),
\grad  v_{n,k}(x))\,\dx.
\end{eqnarray*}}%
Observing that $\grad v_{n,k}(\cdot)\in [T_{v_{n,k}(\cdot)}(S^2)]^2$ $\calL^2$-\aev\ in $B(x_0;\varsigma_k)$
and using \eqref{winmanifold},
  then \eqref{qcxfqcxbarf}
  and \eqref{QQbarfrbarrsbars}
applied to $r=u_{n,k}(x)$,
$\bar r= \tilde
u(x_0)$, $s= v_{n,k}(x)$,
$\bar s= \tilde
v(x_0)$, $\xi=\grad
u_n$, and $\eta=\grad
v_{n,k}$  entails
{\setlength\arraycolsep{0.5pt}
\begin{eqnarray*}
&& \calQ_Tf( u_n(x),
 v_{n,k}(x), \grad u_n(x),
\grad  v_{n,k}(x)) =
\calQ \tilde f( u_n(x),
 v_{n,k}(x), \grad u_n(x),
\grad  v_{n,k}(x)) \\
&&\quad\lii \calQ \tilde
f( \tilde u(x_0),  \tilde
v(x_0), \grad u_n(x),
\grad  v_{n,k}(x)) \\
&&\qquad +c\big(|u_n(x)
- \tilde u(x_0)| + |v_{n,k}(x)
- \tilde v(x_0)|\big)\big(|\grad
u_n(x)| + |\grad v_{n,k}(x)|\big),
\end{eqnarray*}}%
from which \eqref{ubcantor1} follows.

\underbar{Step 2.} We prove that
\begin{eqnarray}\label{Inktozero}
\lim_{k\to\infty}\lim_{n\to\infty}{1\over |D^cw|(B(x_0,\varsigma_k))}I_{n,k} = 0,
\end{eqnarray}
where $I_{n,k}$ is the integral defined in \eqref{defIn}.

In this step we will denote by $\calC_\delta$ any positive constant only depending on $\delta$, $\beta$,  $C_\star$, and $\bar C$ .
By \eqref{outsideAnkCantor},  recalling that $v_{n,k}(\cdot)$, $\tilde v(x_0)\in S^2$, $u_n(\cdot)$, $\tilde u(x_0)\in [\alpha,\beta]$, and using Lemma~\ref{lemmaAMT}-{\it ii)} applied to $u$ and $\mathfrak{h}(\cdot):=|u_n(\cdot) - \tilde u(x_0)| + |v_n(\cdot) - \tilde v(x_0)|$, we have
{\setlength\arraycolsep{0.5pt}
\begin{eqnarray*}
&& \int_{B(x_0,\varsigma_k)}
\big(|u_n(x) - \tilde
u(x_0)|+ |v_{n,k}(x)
- \tilde v(x_0)|\big)|\grad
u_n(x)|\,\dx  \\
&&\quad\lii (2\beta+2)\int_{A_{n,k}}
|\grad u_n(x)|\,\dx +  \int_{B(x_0,\varsigma_k)\backslash
A_{n,k}} \big(|u_n(x)
- \tilde u(x_0)|+ \overline C  |v_n(x) - \tilde v(x_0)|\big)|\grad
u_n(x)|\,\dx\\
&&\quad\lii \frac{2\beta+2}{\delta}\int_{A_{n,k}}
|v_n(x) - \tilde v(x_0)||\grad
u_n(x)|\,\dx \\
&&\hskip35mm+  \int_{B(x_0,\varsigma_k)\backslash
A_{n,k}} \big(|u_n(x)
- \tilde u(x_0)|+ \overline C  |v_n(x) - \tilde v(x_0)|\big)|\grad
u_n(x)|\,\dx\\
&&\quad\lii \calC_\delta\int_{B(x_0,\varsigma_k)}
\big(|u_n(x) - \tilde
u(x_0)|+   |v_n(x) -
\tilde v(x_0)|\big)|\grad
u_n(x)|\,\dx\\
&&\quad\lii \calC_\delta
\int_{B(x_0,\varsigma_k
+ \frac{1}{n})} \Big[
(|u_n - \tilde u(x_0)|*\rho_n)(x)+(|v_n
- \tilde v(x_0)|*\rho_n\big)(x)\Big] \,\d|Du|(x)\\
&&\quad\lii 2\calC_\delta
\int_{B(x_0,\varsigma_k
+ \frac{1}{n})}  (|w_n
- \tilde w(x_0)|*\rho_n)(x)\,\d|Du|(x).
\end{eqnarray*}}%
Similarly, 
{\setlength\arraycolsep{0.5pt}
\begin{eqnarray}
&& \int_{B(x_0,\varsigma_k)}
\big(|u_n(x) - \tilde
u(x_0)|+|v_{n,k}(x) -
\tilde v(x_0)|\big)|\grad
v_{n,k}(x)|\,\dx \nonumber  \\
&&\qquad\lii (2\beta+2)C_\star\int_{A_{n,k}}
|\grad v_n(x)|\,\dx +  4\bar C\int_{B(x_0,\varsigma_k)\backslash
A_{n,k}} \big(|u_n(x)
- \tilde u(x_0)|+\overline C|v_n(x) - \tilde v(x_0)|\big)|\grad
v_n(x)|\,\dx \nonumber\\
&&\qquad\lii \calC_\delta
\int_{B(x_0,\varsigma_k)}
\big(|u_n(x) - \tilde
u(x_0)| + |v_n(x) - \tilde
v(x_0)|\big)|\grad v_n(x)|\,\dx \nonumber\\
&&\qquad\lii 2\calC_\delta
\int_{B(x_0,\varsigma_k
+ \frac{1}{n})} (|w_n
- \tilde w(x_0)|*\rho_n)(x)\,\d|Dv|(x).\label{Inktozero1}
\end{eqnarray}}%

Hence, since $|Du| + |Dv|\lii 2|Dw|$ in ${\cal B}(\Omega)$, we deduce that
\begin{eqnarray*}
I_{n,k}\lii \calC_\delta\int_{B(x_0,\varsigma_k + \frac{1}{n})} (|w_n - \tilde w(x_0)|*\rho_n)(x)\,\d|Dw|(x).
\end{eqnarray*}

Using the estimate  $\Vert |w_n - \tilde w(x_0)|*\rho_n\Vert_{L^\infty(B(x_0,\varsigma_k + \frac{1}{n}))}\lii 2(\beta + 1)$, we obtain
{\setlength\arraycolsep{0.5pt}
\begin{eqnarray*}
I_{n,k}&&\lii  \calC_\delta
\int_{B(x_0,\varsigma_k
+ \frac{1}{n})\backslash
S_w} (|w_n - \tilde w(x_0)|*\rho_n)(x)\,\d|Dw|(x)
+ \calC_\delta|Dw|\Big(B\Big(x_0,\varsigma_k
+ \frac{1}{n}\Big)\cap
S_w\Big)\\
 &&\lii  \calC_\delta
\int_{B(x_0,\varsigma_k
+ \frac{1}{n})\backslash
S_w} \big((|w_n - w|*\rho_n)(x)
+ (| w- \tilde w(x_0)|*\rho_n)(x)\big)\,\d|Dw|(x)
\\
&&\hskip75mm + \calC_\delta|Dw|\Big(B\Big(x_0,\varsigma_k
+ \frac{1}{n}\Big)\cap
S_w\Big).
\end{eqnarray*}}%

Define $\bar w(\cdot):= |w(\cdot) - \tilde w(x_0)|$. By Proposition~\ref{appppties}~{\it (a)-ii)} and {\it (a)-iii)} applied to $\bar w$ and to $w$, respectively, we conclude that
\begin{eqnarray*}
 \lim_{n\to\infty}(| w- \tilde w(x_0)|*\rho_n)(x) =\tilde{\bar w}(x)=|\tilde w(x) - \tilde w(x_0)|\quad \hbox{ for all } x\in A_w=\Omega\backslash S_w,
\end{eqnarray*}
while in view of 
Lemma~\ref{lemmaAMT}-{\it iv)} applied to $w$,
\begin{eqnarray*}
\lim_{n\to\infty}(|w_n - w|*\rho_n)(x)=0 \quad \hbox{ for all } x\in A_w=\Omega\backslash
S_w.
\end{eqnarray*}
These last two limits, together with Lebesgue's Dominated and Monotone Convergence Theorems, yield
{\setlength\arraycolsep{0.5pt}
\begin{eqnarray*}
\limsup_{n\to\infty}
I_{n,k}&&\lii \calC_\delta
\int_{\overline{B(x_0,\varsigma_k
)}\backslash S_w} |\tilde
w(x)- \tilde w(x_0)|\,\d|Dw|(x)
+ \calC_\delta|Dw|\big(\overline{B(x_0,\varsigma_k
})\cap S_w\big)\\
&&\lii \calC_\delta \int_{{B(x_0,\varsigma_k
)}} |\tilde w(x)- \tilde
w(x_0)|\,\d|D^cw|(x)
+ \calC_\delta \nu({B(x_0,\varsigma_k
}))+\calC_\delta|Dw|({B(x_0,\varsigma_k
})\cap S_w),
\end{eqnarray*}}%
where we  used the fact that $|Dw|(\partial B(x_0,\varsigma_k))=0$ and the 
equality $|Dw|=|D^cw| +\nu$. We  observe that $|Dw|(B(x_0,\varsigma_k)\cap S_w)= \nu (B(x_0,\varsigma_k)\cap S_w) + |D^cw|(B(x_0,\varsigma_k)\cap S_w) = \nu (B(x_0,\varsigma_k)\cap S_w)\lii \nu (B(x_0,\varsigma_k))$; hence, by  \eqref{LebPointwrtCw} and \eqref{nuCwarems},
{\setlength\arraycolsep{0.5pt}
\begin{eqnarray*}
&& \limsup_{k\to\infty}\limsup_{n\to\infty}{1\over
|D^cw|(B(x_0,\varsigma_k))}I_{n,k}\\
&&\quad\lii\calC_\delta\limsup_{k\to\infty}
\bigg(
\dashint_{{B(x_0,\varsigma_k
)}} |\tilde w(x)- \tilde
w(x_0)|\,\d|D^cw|(x)
 +  \frac{\nu({B(x_0,\varsigma_k
}))}{|D^cw|(B(x_0,\varsigma_k))}
\bigg)=0,
\end{eqnarray*}}%
  and we conclude  \eqref{Inktozero}.

\underbar{Step 3.} We show that
{\setlength\arraycolsep{0.5pt}
\begin{eqnarray}\label{ubcantopart2}
&& \limsup_{k\to\infty}
\limsup_{n\to\infty}{1\over
|D^cw|(B(x_0,\varsigma_k))}
\int_{B(x_0,\varsigma_k)}
\calQ \tilde f(\tilde
u(x_0),\tilde v(x_0),
\grad  u_n(x), \grad
 v_{n,k}(x))\,\dx \nonumber\\
&&\quad\lii  (\calQ_T
f)^\infty\big(\tilde
u(x_0) ,\tilde v(x_0),
W^c_u(x_0), W^c_v(x_0)\big).
\end{eqnarray}}%
As in \cite[Prop.~{4.2}]{ADMXCII}, we define a function ${\mathfrak{\mathfrak{z}}}:\RR^{4\times2}\to[0,+\infty)$ by setting
\begin{eqnarray*}
\mathfrak{z}(\zeta):=\sup_{t>0}\frac{\calQ \tilde f(\tilde u(x_0),\tilde v(x_0), t\xi, t\eta)}{t},\quad \zeta\in \RR^{4\times2},
\end{eqnarray*}
where $\xi$ is the first row of $\zeta$ and $\eta$ is the $3\times2$ matrix obtained from $\zeta$ by erasing its first row. Observe that for all $r\in\RR$, $s\in\RR^3$, $\calQ \tilde f(r,s,0,0)=0$ 
since \(\calQ \tilde f \lii \tilde f\) and $\tilde f(r,s,0,0)=0$
by \eqref{deftildef}, \eqref{defPDetall}, and \eqref{defoff}. Note also that
\begin{eqnarray}\label{ggiiQQbarf}
\mathfrak{\mathfrak{z}}(\zeta)\gii  \calQ \tilde f(\tilde u(x_0),\tilde v(x_0), \xi, \eta)\quad\hbox{ for all } \zeta\in \RR^{4\times2}.
\end{eqnarray}
 Moreover (cf. \cite[Prop.~{4.2}]{ADMXCII}),  $\mathfrak{\mathfrak{z}}$ is a positively 1-homogeneous quasiconvex function satisfying \eqref{qcxisLip} and  the rank-one convexity of $\calQ \tilde f(\tilde u(x_0),\tilde v(x_0),\cdot,\cdot)$ implies that
\begin{eqnarray}\label{=onrank1}
\mathfrak{\mathfrak{z}}(\zeta) = (\calQ \tilde f)^\infty(\tilde u(x_0),\tilde v(x_0),\xi,\eta)\quad \hbox{ for all } \zeta\in\RR^{4\times 2},\,\, {\rm rank}(\zeta)\lii1.
\end{eqnarray}
In view of \eqref{ggiiQQbarf} and \eqref{qcxisLip},
we have
{\setlength\arraycolsep{0.5pt}
\begin{eqnarray}\label{ubcantopart3}
&& \int_{B(x_0,\varsigma_k)}
\calQ \tilde f(\tilde
u(x_0),\tilde v(x_0),
\grad  u_n(x), \grad
 v_{n,k}(x))\,\dx \nonumber\\
&&\quad\lii \int_{B(x_0,\varsigma_k)}
\mathfrak{\mathfrak{z}}(\grad w_n(x))\,\dx
+ L\int_{B(x_0,\varsigma_k)}
|\grad v_n(x) - \grad
v_{n,k}(x)|\,\dx.
\end{eqnarray}}%
We claim that
{\setlength\arraycolsep{0.5pt}
\begin{eqnarray}\label{ubcantopart4}
&& \limsup_{k\to\infty}
\limsup_{n\to\infty}{1\over
|D^cw|(B(x_0,\varsigma_k))}
\int_{B(x_0,\varsigma_k)}
\mathfrak{\mathfrak{z}}(\grad w_n(x))\,\dx
\lii  (\calQ_T
f)^\infty\big(\tilde
u(x_0) ,\tilde v(x_0),
W^c_u(x_0), W^c_v(x_0)\big)
\nonumber\\
\end{eqnarray}}%
and
that
\begin{eqnarray}\label{ubcantopart5}
\begin{aligned}
& \lim_{k\to\infty}\lim_{n\to\infty}{1\over |D^cw|(B(x_0,\varsigma_k))}\int_{B(x_0,\varsigma_k)} |\grad v_n(x) - \grad v_{n,k}(x)|\,\dx=0,
\end{aligned}
\end{eqnarray}
which, together with \eqref{ubcantopart3},  yield \eqref{ubcantopart2}.

We start by proving \eqref{ubcantopart4}. By Lemma~\ref{lemmaAMT}-{\it iii)}, we obtain
{\setlength\arraycolsep{0.5pt}
\begin{eqnarray*}
&& \lim_{n\to\infty}{1\over
|D^cw|(B(x_0,\varsigma_k))}\int_{B(x_0,\varsigma_k)}
\mathfrak{\mathfrak{z}}(\grad w_n(x))\,\dx = {1\over |D^cw|(B(x_0,
\varsigma_k))}\int_{B(x_0,
\varsigma_k)} \mathfrak{\mathfrak{z}} \Big(\frac{\d Dw}{\d |Dw|}(x)\Big)\,\d|Dw|(x)\\
&&\quad \lii {1\over
|D^cw|(B(x_0,\varsigma_k))}\int_{B(x_0,\varsigma_k)}
\mathfrak{\mathfrak{z}}(W^c(x))\,\d|D^cw|(x)
+ (2+\sqrt2(1+\beta)) {\nu(B(x_0,\varsigma_k))\over
|D^cw|(B(x_0,\varsigma_k))},
\end{eqnarray*}}%
where we also used \eqref{boundsQQbarf}. In view of 
Theorem~\ref{rank1Alberti},
\eqref{=onrank1}, \eqref{nuCwarems}, and \eqref{winmanifold}, in this order, we  have 
{\setlength\arraycolsep{0.5pt}
\begin{eqnarray}\label{ubcantopart10}
&& \limsup_{k\to\infty}
\limsup_{n\to\infty}{1\over
|D^cw|(B(x_0,\varsigma_k))}
\int_{B(x_0,\varsigma_k)}
\mathfrak{\mathfrak{z}}(\grad w_n(x))\,\dx
\nonumber\\
&&\quad \lii \limsup_{k\to\infty}
\dashint_{B(x_0,\varsigma_k)}
(\calQ \tilde f)^\infty(\tilde
u(x_0),\tilde v(x_0),W_u^c(x),
W_v^c(x))\,
\,\d|D^cw|(x)\nonumber\\
&&\quad\lii (\calQ \tilde
f)^\infty(\tilde
u(x_0),\tilde v(x_0),W_u^c(x_0),
W_v^c(x_0))
\nonumber\\
&&\qquad\enspace +\,\lim_{k\to\infty}
\dashint_{B(x_0,\varsigma_k)}
\big|(\calQ \tilde
f)^\infty(\tilde
u(x_0),\tilde v(x_0),W_u^c(x_0),
W_v^c(x_0))\nonumber\\
&&\hskip50mm-(\calQ \tilde
f)^\infty(\tilde
u(x),\tilde v(x),W_u^c(x),W_v^c(x))
\,\big|\,\d|D^cw|(x)
\nonumber\\
&&\qquad\enspace +\,
\limsup_{k\to\infty}
\dashint_{B(x_0,\varsigma_k)}
\big|(\calQ \tilde
f)^\infty(\tilde
u(x),\tilde v(x),W_u^c(x),W_v^c(x))
\nonumber\\
&&\hskip50mm-(\calQ \tilde
f)^\infty(\tilde
u(x_0),\tilde v(x_0),W_u^c(x),
W_v^c(x))
\,\big|\,\d|D^cw|(x)
\nonumber\\
&&\quad\lii (\calQ_T
f)^\infty(\tilde
u(x_0),\tilde v(x_0),W_u^c(x_0),
W_v^c(x_0))
\nonumber\\
&&\qquad\enspace +\,c
\limsup_{k\to\infty}
\dashint_{B(x_0,\varsigma_k)}
\big(\big|\tilde
u(x) - \tilde u(x_0)|
+|\tilde v(x)-\tilde
v(x_0)|\big)\big(|W_u^c(x)|
 + |W_v^c(x)|\big)\,\d|D^cw|(x),
\end{eqnarray}}%
where in the last inequality
we used \eqref{qcxfqcxbarf},
together with Lemma~\ref{onManae}  and  the definition
of the recession functions
of $\calQ_Tf$ and $\calQ
\tilde f$, \eqref{LebPointrecQtildef},
and \eqref{contrecQtildef}.
Furthermore,
{\setlength\arraycolsep{0.5pt}
\begin{eqnarray*}
&& \limsup_{k\to\infty}
\dashint_{B(x_0,\varsigma_k)}
\big(\big|\tilde
u(x) - \tilde u(x_0)|
+|\tilde v(x)-\tilde
v(x_0)|\big)\big(|W_u^c(x)|
 + |W_v^c(x)|\big)\,\d|D^cw|(x)\\
&&\quad\lii4\,\limsup_{k\to\infty}
\dashint_{B(x_0,\varsigma_k)}
\big[|W^c(x_0)||\tilde
w(x) - \tilde w(x_0)|
+(\beta+1)|W^c(x)-W^c(x_0)|\big]
\,\d|D^cw|(x),
\end{eqnarray*}}%
 which, together with
\eqref{ubcantopart10}, \eqref{LebPointwrtCw}, and
\eqref{LebPointWcwrtCw}, entails  \eqref{ubcantopart4}.

Finally, we establish \eqref{ubcantopart5}. Arguing as in Step~2, using \eqref{onAnkCantor}, \eqref{outsideAnkCantor}, and the second estimate in \eqref{uniflipballtf} applied to $y=y_{n,k}$, $s_1=v_n(x)$ (for $x\in B(x_0,\varsigma_k)\backslash  A_{n,k}$),  and $s_2=\tilde v(x_0)$, and recalling that $\grad v_{n,k}(\cdot) = \grad \pi_{y_{n,k}}(v_n(\cdot))\grad v_n(\cdot)$, we obtain
{\setlength\arraycolsep{0.5pt}
\begin{eqnarray}\label{ubcantopart6}
&&\int_{B(x_0,\varsigma_k)}
|\grad v_n(x) - \grad
v_{n,k}(x)|\,\dx \nonumber\\
&&\quad \lii \frac{1+
C_\star}{\delta}\int_{A_{n,k}}
|v_n(x) -\tilde v(x_0)|
 |\grad v_n(x)|\,\dx
+\overline C \int_{B(x_0,\varsigma_k)\backslash
 A_{n,k}} |v_n(x) -\tilde
v(x_0)|  |\grad v_n(x)|\,\dx
\nonumber \\
&&\qquad\enspace +\int_{B(x_0,
\varsigma_k)\backslash
 A_{n,k}} |\grad v_n(x)
- \grad \pi_{y_{n,k}}(\tilde
v(x_0))\grad v_n(x)|\,\dx.
\end{eqnarray}}%
Moreover (see \eqref{Inktozero1}),
\begin{eqnarray*}
\lim_{k\to\infty}\lim_{n\to\infty}{1\over |D^cw|(B(x_0,\varsigma_k))}\int_{B(x_0,\varsigma_k)}|v_n(x) -\tilde v(x_0)|  |\grad v_n(x)|\,\dx =0,
\end{eqnarray*}
and so, in view of \eqref{ubcantopart6}, to prove \eqref{ubcantopart5} it suffices to show that %
\begin{eqnarray}\label{ubcantopart7}
\lim_{k\to\infty}\lim_{n\to\infty}{1\over |D^cw|(B(x_0,\varsigma_k))}\int_{B(x_0,\varsigma_k)}|\grad v_n(x) - \grad \pi_{y_{n,k}}(\tilde v(x_0))\grad v_n(x)|\,\dx =0.
\end{eqnarray}

In the remaining part of the proof, $\II_{4\times 4}$ denotes the $4\times 4$ identity matrix and $B_{n,k}$ denotes the $4\times 4$  matrix whose last three rows and columns are those of $\grad\pi_{y_{n,k}}(\tilde v(x_0))$ and the first row and column are those of the identity matrix. We observe that $| \II_{4\times 4} - B_{n,k}| = | \II_{3\times 3} - \grad\pi_{y_{n,k}}(\tilde v(x_0))|\lii \sqrt{3}+ 2\bar C$. Moreover, 
\begin{eqnarray*}
\begin{aligned}
|\grad v_n(\cdot) - \grad \pi_{y_{n,k}}(\tilde v(x_0))\grad v_n(\cdot)| = |(\II_{4\times 4}- B_{n,k})\grad w_n(\cdot)| = |\grad(\vartheta_{n,k}*\rho_n)
(\cdot)|,
\end{aligned}
\end{eqnarray*}
where
\begin{eqnarray*}
\begin{aligned}
\vartheta_{n,k}(\cdot):=(\II_{4\times 4}- B_{n,k}) w.
\end{aligned}
\end{eqnarray*}
Since
\begin{eqnarray*}
\begin{aligned}
D\vartheta_{n,k} = (\II_{4\times 4}- B_{n,k})\grad w\calL^2_{\lfloor\Omega} + (\II_{4\times 4}- B_{n,k})(w^+ - w^-)\otimes \nu_w\calH^1_{\lfloor J_w} + (\II_{4\times 4}- B_{n,k}) W^c|D^c w|,
\end{aligned}
\end{eqnarray*}
using Lemma~\ref{lemmaAMT}-{\it ii)} with $\mathfrak{h}\equiv 1$ and observing that $1*\rho_n\equiv1$, we have
{\setlength\arraycolsep{0.5pt}
\begin{eqnarray}\label{ubcantopart8}
&& \int_{B(x_0,\varsigma_k)}|\grad
v_n(x) - \grad \pi_{y_{n,k}}(\tilde
v(x_0))\grad v_n(x)|\,\dx
\lii |D\vartheta_{n,k}|
\Big(B\Big(x_0,\varsigma_k
+ \frac{1}{n}\Big) \Big)
\nonumber\\
&&\qquad \lii(\sqrt{3}+ 2\bar C)\nu\Big(B\Big(x_0,\varsigma_k
+ \frac{1}{n}\Big) \Big)
+ \int_{B(x_0,\varsigma_k
+ \frac{1}{n})} |(\II_{4\times
4}- B_{n,k}) W^c(x)|\,\d|D^cw|(x).
\end{eqnarray}}%
By Lemma~\ref{onManae}, $v(x)\in S^2$ for all $x$ in $\Omega$ except possibly for $x$ belonging to the $\calH^1$-negligible set $S_w\backslash J_w$. Therefore, redefining $v$ on $S_w\backslash J_w$ so that $v(x)\in S^2$ for all $x\in\Omega$ if necessary, in view of \eqref{gradprojonS} and \eqref{winmanifold}, we conclude that
\begin{eqnarray*}
\begin{aligned}
\grad\pi_{y_{n,k}}(\tilde v(x))W_v^c(x) =  W_v^c(x)\quad \hbox{ for $|D^cw|$-\aev\ $x\in\Omega$}.
\end{aligned}
\end{eqnarray*}
Hence, using once more \eqref{uniflipballtf} and recalling that $|W^c(x)|= 1$ for $|D^cw|$-\aev\ $x\in\Omega$,
we deduce that
{\setlength\arraycolsep{0.5pt}
\begin{eqnarray}\label{ubcantopart9}
&& \int_{B(x_0,\varsigma_k
+ \frac{1}{n})} |(\II_{4\times
4}- B_{n,k}) W^c(x)|\,\d|D^cw|(x)
\nonumber\\
&&\quad = \int_{B(x_0,\varsigma_k
+ \frac{1}{n})} |\grad
\pi_{y_{n,k}}(\tilde
v(x))W_v^c(x)- \grad
\pi_{y_{n,k}}(\tilde
v(x_0)) W_v^c(x)|\,\d|D^cw|(x)
\nonumber\\
&&\quad\lii  \overline C
\int_{B(x_0,\varsigma_k
+ \frac{1}{n})} |\tilde
v(x) -  \tilde v(x_0)|\,\d|D^cw|(x)
\lii
 \overline C \int_{B(x_0,\varsigma_k
+ \frac{1}{n})} |\tilde
w(x) -  \tilde w(x_0)|\,\d|D^cw|(x).
\end{eqnarray}}%
Since $D^cw(\partial B(x_0,\varsigma_k))= \nu (\partial B(x_0,\varsigma_k))=0$, from \eqref{ubcantopart8}, \eqref{ubcantopart9}, \eqref{nuCwarems}, and  \eqref{LebPointwrtCw}, we infer \eqref{ubcantopart7}, which concludes Step~3.

Finally, \eqref{ubcantorpart} follows from \eqref{ubcantor1}, \eqref{Inktozero}, and \eqref{ubcantopart2}.
\end{proof}

\begin{lemma}\label{jumppartCalF}
If $(u,v)\in BV(\Omega;[\alpha,
\beta])\times BV(\Omega;S^2)$, then for all $A\in\calA(\Omega)$,
\begin{eqnarray}\label{ubjump}
\begin{aligned}
\calF(u,v;A \cap S_{(u,v)})\lii \int_{A\cap S_{(u,v)}}  K\big((u,v)^+(x),(u,v)^-(x),
 \nu_{(u,v)}(x)\big)\,\d\calH^1(x).
\end{aligned}
\end{eqnarray}
\end{lemma}

\begin{proof} Let $A\in\calA(\Omega),$
and set $w:=(u,v)$.
We will proceed in three steps, and we closely follow
 the argument in \cite[Step~3 in Sect.~5.2]{FMXCIII}
(see also \cite[Lem.~6.5]{ACELVII}).

\underbar{Step 1.} We prove that \eqref{ubjump} holds whenever $w$ is of the form
\begin{eqnarray*}
\begin{aligned}
w(x)= a\chi_E(x) +b\chi_{E^c}(x),
\end{aligned}
\end{eqnarray*}
where $a,\,b\in [\alpha,\beta]\times S^2$ and $E\subset\Omega$ is a set of finite perimeter in $\Omega$.

{\sl Substep 1.1.} We start by considering the case in which $A= \kappa + \lambda
Q_\nu$ and
\begin{eqnarray}\label{purejump}
\begin{aligned}
w(x)=
\begin{cases}
b & \hbox{if } x\cdot\nu>\sigma
\hbox{ and } x\in A,\\
a &\hbox{if } x\cdot\nu<\sigma
\hbox{ and } x\in A,
\end{cases}
\end{aligned}
\end{eqnarray}
for some $\kappa \in\RR^2$,  $\lambda \in\RR^+$,
 $\nu\in S^1
$, and $\sigma\in\RR$.

Without loss of generality,
we may assume that $
A\cap\{x\in\RR^2\!:\,
x\cdot\nu=\sigma\}\not=\emptyset$. 
Fix $\epsi>0,$ and let $\vartheta=(\ffi,\psi)\in \calP(a,b,\nu)$, depending on $\epsi$, be such that 
\begin{eqnarray}\label{ubjump0}
\begin{aligned}
 K(a,b,\nu) + \epsi>\int_{Q_\nu}
 f^\infty(\vartheta(y),
 \grad\vartheta(y))\,\dy.
\end{aligned}
\end{eqnarray}
{\sl Substep 1.1.1. Assume
that $\nu=e_2$.
}

Set $Q:=Q_{e_2}$, and
 for $n\in\NN$,
define $
w_n\in W^{1,1}_{\rm loc}(\RR^2;
[\alpha,\beta]\times
S^2)$ as 
\begin{eqnarray}\label{ubjump0'}
\begin{aligned}
w_n(x)=(u_n(x),v_n(x)):=
\begin{cases}
b & \hbox{if } x_2>\displaystyle \frac{\lambda}{2(2n+1)}+\sigma,\\
\displaystyle \vartheta\Big((2n+1)
\frac{x-(\kappa_1, \sigma)}{\lambda}\Big) & \hbox{if
} |x_2-\sigma|\displaystyle\lii
\frac{\lambda}{2(2n+1)},\\
a &\hbox{if } x_2<\displaystyle -\frac{\lambda}{2(2n+1)}+\sigma.
\end{cases}
\end{aligned}
\end{eqnarray}
For
all $n\in\NN$ large enough,
we have that 
 $A\cap\big\{x\in\RR^2\!:\,
x_2=\frac{\lambda}{2(2n+1)}+\sigma\big\}
\not=\emptyset$ and $A\cap\big\{x\in\RR^2\!:\,
x_2=-\frac{\lambda}{2(2n+1)}+\sigma\big\}
\not=\emptyset$. For
all such $n\in\NN$, a change of variables yields
{\setlength\arraycolsep{0.5pt}
\begin{eqnarray}\label{ubjump1}
\int_A |w_n(x) - w(x)|\,\dx
&&= 
\int_\sigma^{\frac{\lambda}{2(2n+1)}+
\sigma}
\int_{\kappa_1 - \frac{\lambda}{2}}
^{\kappa_1 + \frac{\lambda}{2}}
\Big|\vartheta\Big((2n+1)
\frac{x-(\kappa_1, \sigma)}{\lambda}
\Big)
- b\Big| \,\dx_1\,\dx_2
\nonumber\\
&&\qquad\qquad\quad+
\int^\sigma_{-\frac{\lambda}{2(2n+1)}
+\sigma}
\int_{\kappa_1 - \frac{\lambda}{2}}
^{\kappa_1 + \frac{\lambda}{2}}
\Big|\vartheta\Big((2n+1)
\frac{x-(\kappa_1, \sigma)}{\lambda}
\Big)
- a\Big| \,\dx_1\,\dx_2
\nonumber\\
&&=\frac{\lambda^2}{2n+1}
\bigg(
\int_0^{\frac{1}{2}}
\int_{-\frac{1}{2}}^{\frac{1}{2}}
|\vartheta((2n+1)y_1,
y_2) - b|\,\dy_1\,\dy_2
\nonumber\\
&&\qquad\qquad\quad+
\int^0_{-\frac{1}{2}}
\int_{-\frac{1}{2}}^{\frac{1}{2}}
|\vartheta((2n+1)y_1,
y_2) - a|\,\dy_1\,\dy_2
\bigg).
\end{eqnarray}}%
By the Riemann--Lebesgue
Lemma, by the 1-periodicity of $\vartheta$ in the $e_1$
direction, and by the Lebesgue Dominated Convergence Theorem, we  obtain
{\setlength\arraycolsep{0.5pt}
\begin{eqnarray*}
&& \lim_{n\to\infty} \bigg(
\int_0^{\frac{1}{2}}
\int_{-\frac{1}{2}}^{\frac{1}{2}}
|\vartheta((2n+1)y_1,
y_2) - b|\,\dy_1\,\dy_2+
\int^0_{\frac{1}{2}}
\int_{-\frac{1}{2}}^{\frac{1}{2}}
|\vartheta((2n+1)y_1,
y_2) - a|\,\dy_1\,\dy_2
\bigg)\\
&&\quad = \bigg(
\int_0^{\frac{1}{2}}
\int_{-\frac{1}{2}}^{\frac{1}{2}}
|\vartheta(z_1,
y_2) - b|\,\d z_1\,\dy_2+
\int^0_{\frac{1}{2}}
\int_{-\frac{1}{2}}^{\frac{1}{2}}
|\vartheta(z_1,
y_2) - a|\,\d z_1\,\dy_2
\bigg).
\end{eqnarray*}}%
Hence, passing \eqref{ubjump1}
to the limit as $n\to\infty$,
we conclude that
\begin{eqnarray}\label{ubjump2}
\begin{aligned}
& \lim_{n\to\infty}\Vert
w_n - w\Vert_{L^1(A;\RR\times
\RR^3)}=0.
\end{aligned}
\end{eqnarray}
Consequently, 
{\setlength\arraycolsep{0.5pt}
\begin{eqnarray*}
\calF(w;A)&&\lii \liminf_{n\to\infty}
\int_A f(w_n(x),\grad
w_n(x))\,\dx
\\
&& = \liminf_{n\to\infty}
\int_{-\frac{\lambda}{2(2n+1)}+\sigma}
^{\frac{\lambda}{2(2n+1)}+\sigma}
\int_{\kappa_1 - \frac{\lambda}{2}}
^{\kappa_1 + \frac{\lambda}{2}}
f \Big(\vartheta\Big((2n+1)
\frac{x-(\kappa_1, \sigma)}
{\lambda}\Big), \frac{2n+1}{\lambda}
\grad \vartheta\Big((2n+1)
\frac{x-(\kappa_1, \sigma)}
{\lambda}\Big)
\Big) \,\dx_1\,\dx_2\\
&&= \liminf_{n\to\infty} \int_{-\frac{1}{2}}^
{\frac{1}{2}}
\int_{-\frac{2n+1}{2}}^
{\frac{2n+1}{2}}
\frac{\lambda^2}{(2n+1)^2}
f \Big(\vartheta
(y), \frac{2n+1}{\lambda}
\grad \vartheta(y)
\Big) \,\dy_1\,\dy_2\\
&&= \liminf_{n\to\infty} \lambda\int_{-\frac{1}{2}}^
{\frac{1}{2}}
\int_{-\frac{1}{2}}^
{\frac{1}{2}}
\frac{\lambda}{2n+1}
f \Big(\vartheta
(y), \frac{2n+1}{\lambda}
\grad \vartheta(y)
\Big) \,\dy_1\,\dy_2,
\end{eqnarray*}}%
where we used $f(\cdot,\cdot,0,0)=0$ and in the last equality
we invoked the 1-periodicity
of
$\vartheta$ in the $e_1$
direction.
Hence,  Fatou's Lemma, together with \eqref{boundsf},
and \eqref{ubjump0}
yield
\begin{eqnarray*}
\begin{aligned}
&\calF(w;A) \lii \lambda
\int_Q f^\infty (\vartheta
(y), \grad \vartheta(y)) \,\dy < \lambda K(a,b, e_2)
+\lambda \epsi = K(a,b,e_2) \calH^1(A\cap
S_w) +\lambda \epsi,
\end{aligned}
\end{eqnarray*}
 from which we obtain \eqref{ubjump} by letting $\epsi\to0^+$.

{\sl Substep 1.1.2.}
We complete Substep~{1.1}.

Let $R\in SO(2)$ be such that $R e_2=\nu$, and define
\begin{eqnarray*}
\begin{aligned}
&\bar w(x):= w(Rx), \enspace x\in
R^T\! A = R^T\!\kappa + \lambda Q_{e_2},\quad
\bar\vartheta (y):= \vartheta (Ry), \enspace y\in Q.
\end{aligned}
\end{eqnarray*}
Let $\{\bar w_n\}_{n\in\NN}$
be the sequence in \eqref{ubjump0'}
with $\vartheta$ replaced
by $\bar\vartheta$. Then,
\eqref{ubjump2} reads
as $\bar w_n\to \bar
w$ in $L^1(R^T\!A;\RR\times\RR^3)$,
which in turn implies
that  $w_n\to
w$ in $L^1(A;\RR\times\RR^3)$,
where 
\begin{eqnarray*}
\begin{aligned}
& w_n(x):= \bar w_n(R^T
x),\quad x\in A,\enspace
n\in\NN.
\end{aligned}
\end{eqnarray*}
Finally, arguing as in
Substep~{1.1.1}, we obtain
{\setlength\arraycolsep{0.5pt}
\begin{eqnarray*}
\calF(w;A)&&\lii \liminf_{n\to\infty}
\int_A f(w_n(x),\grad
w_n(x))\,\dx = \liminf_{n\to\infty}
\int_{R^T\!A} f(\bar w_n(x),  \grad
\bar w_n(x)R^T)\,\dx\\
&&\lii\lambda \int_{
Q_{e_2}} f^\infty(\bar \vartheta(y),
 \grad
\bar \vartheta(y)R^T)\,\dy
= \lambda \int_{Q_\nu}
f^\infty(\vartheta(y),
\grad \vartheta(y))\,\dy,
\end{eqnarray*}}%
which, together with \eqref{ubjump0},
concludes Substep~{1.1}.

{\sl Substep 1.2.} We prove that
if  $\kappa \in\RR^2$,  $\lambda \in\RR^+$,
 $\nu\in S^1
$,  $\sigma\in\RR$, and  \(\delta>0\) are such that
for  \(x\in \kappa + (\lambda+\delta)Q_\nu\),
we have 
\begin{eqnarray*}
\begin{aligned}
w(x)=
\begin{cases}
b & \hbox{if } x\cdot\nu>\sigma
,\\
a &\hbox{if } x\cdot\nu<\sigma,
\end{cases}
\end{aligned}
\end{eqnarray*}
then
\begin{equation}
\label{closedcube}
\calF(w;\kappa + \lambda \overline Q_\nu) \lii
K(a,b,\nu) \, \calH^1((\kappa + \lambda \overline
Q_\nu)\cap
S_w).
\end{equation}

Let
 $\{\lambda_n\}_{n\in\NN} \subset (\lambda,\lambda+\delta)$
be a strictly decreasing sequence 
converging to $\lambda$.
By Lemma~\ref{CalFmeasure}
and Substep~{1.1},
we have 
\begin{equation*}
\calF(w;\kappa + \lambda \overline Q_\nu)
=
\lim_{n\to\infty} \calF(w;\kappa
+ \lambda_n
{Q_\nu}) \nonumber \\ 
\lii \lim_{n\to\infty}
K(a,b,\nu) \calH^1((\kappa
+ \lambda_n
{Q_\nu})\cap
S_w) = 
K(a,b,\nu) \, \calH^1((\kappa + \lambda \overline Q_\nu)\cap
S_w),
\end{equation*}
which proves \eqref{closedcube}.

{\sl Substep 1.3.} We
treat the case in which $A\in\calA(\Omega)$ is
arbitrary and
$w$ is of the form \eqref{purejump}.

Let $\nu_1\in S^1$ be a fixed vector such that
 $\{\nu_1,\nu\}$ is an orthonormal basis of $\RR^2$,
 and let $\{c_{ij}\in\RR^2\!:\, i,j\in\ZZ\}$ be the collection of nodes of a grid
in $\RR^2$ of size 1
with respect to the basis
$\{\nu_1,\nu\}$ such
that
\begin{eqnarray}\label{gridcij}
\begin{aligned}
& \Big\{\frac{c_{ij}}{2^{n-1}}\in\RR^2\!:\,
i,j\in\ZZ, \,n\in\NN\Big\}
\cap \{x\in\RR^2\!:\,
x\cdot\nu=\sigma\}=\emptyset.
\end{aligned}
\end{eqnarray}
Next, we write \(A\) as a union of convenient
closed squares for which the previous substep applies. 
For \(i, j \in \ZZ\) and \(n\in \NN\), let \(Q^{(n)}_{ij}\)
 represent the open square of size \(\frac{1}{2^{n-1}}\) whose  left inferior vertex is \(\frac{c_{ij}}{2^{n-1}}\), and let  \(\kappa^{(n)}_{ij} \in \RR^2\) be such that
\(Q^{(n)}_{ij} = \kappa^{(n)}_{ij} + \frac{1}{2^{n-1}}
Q_\nu\). Set \(B^{(0)} : = \emptyset\), and define
recursively the sets
{\setlength\arraycolsep{0.5pt}%
\begin{eqnarray*}
&&B^{(n)}:= \Big\{ Q^{(n)}_{ij} = \kappa^{(n)}_{ij} + \frac{1}{2^{n-1}} Q_\nu \!: \, i,j \in \ZZ, \, \overline
{Q^{(n)}_{ij}} \subset A, \, Q^{(n)}_{ij} \cap
\cup_{l=0}^{n-1} B^{(l)} =\emptyset \Big\}, \enspace
n\in\NN.
\end{eqnarray*}}%
Without loss of generality, we may assume that
\(B^{(n)} \not= \emptyset\) for all \(n\in\NN\), otherwise we simply
consider the subsequence of the sequence \(\{B^{(n)}
\}_{n\in\NN}\) obtained by removing all its empty
sets. Let \(I^{(n)}:= \{(i,j) \in \ZZ \times \ZZ\!
: \, Q^{(n)}_{ij} \in B^{(n)} \}\in\NN \) and \(A_k
:= \cup_{n=1}^k \tilde A_n\), where \(\tilde A_n:=
\cup_{(i,j)\in I^{(n)}} \big( \kappa^{(n)}_{ij} + \frac{1}{2^{n-1}}
\overline Q_\nu \big)\). Then,
{\setlength\arraycolsep{0.5pt}%
\begin{eqnarray*}
&&A=\bigcup_{k=1}^\infty
A_k, \quad A_1\subset A_2\subset\cdots\subset
A_k\subset \cdots
,\end{eqnarray*}}%
and, by construction and   \eqref{gridcij}, \( \big\{ \kappa^{(n)}_{ij}
+ \frac{1}{2^{n-1}}
 Q_\nu\!: \, n\in\NN, \, (i,j)\in I^{(n)} \big\} \) is a family of mutually disjoint sets
such that 
\begin{eqnarray*}
\begin{aligned}
&\calH^1\Big(\partial\Big(\kappa^{(n)}_{ij} + \frac{1}{2^{n-1}}
 Q_\nu\Big)\cap
\{x\in\RR^2\!:\,
x\cdot\nu=\sigma\}\Big)=0.
\end{aligned}
\end{eqnarray*}
Hence, 
using  Lemma~\ref{CalFmeasure} and  \eqref{closedcube},
we conclude that
{\setlength\arraycolsep{0.5pt}
\begin{eqnarray*}
\calF(w;A) &&= \lim_{k\to\infty}
\calF(w;A_k) \lii \liminf_{k\to\infty}
\sum_{n=1}^k \sum_{(i,j) \in I^{(n)} }
\calF\Big(w;\kappa^{(n)}_{ij}
+ \frac{1}{2^{n-1}}
\overline Q_\nu\Big)
\\
&& \lii \liminf_{j\to\infty}
\sum_{n=1}^k \sum_{(i,j) \in I^{(n)} }
 K(a,b,\nu)\, \calH^1\Big(\Big(\kappa^{(n)}_{ij}
+ \frac{1}{2^{n-1}}
\overline Q_\nu\Big)
\cap S_w\Big)=  K(a,b,\nu)\,
\calH^1(A
\cap S_w).
\end{eqnarray*}}%
This concludes Substep~{1.2}.

{\sl Substep 1.4.} We
now treat the case in
which  $A\in\calA(\Omega)$
is arbitrary and $w$ has a polygonal
interface; that is, 
\begin{eqnarray*}
\begin{aligned}
& w(x)= a\chi_E(x) +b\chi_{E^c}(x),
\end{aligned}
\end{eqnarray*}
where $E$ is polyhedral
open set with $\partial E= \cup_{i=1}^MH_i$, 
$H_i$ a closed segment
of a line of the
type $\{x\in\RR^2\!:\,
x\cdot\nu_i = \sigma_i\}$
for some $\nu_i\in S^1$
and $\sigma_i\in\RR$,
$i\in\{1,...,M\}$. 

Let $I:=\{i\in\{1,...,M\}\!:\,
\calH^1 (A\cap H_i)>0\}$.
Note that since $A$ is open
and $H_i$ is a closed
segment, $\calH^1 (A\cap H_i)=0$ is equivalent
to saying that $A\cap H_i=\emptyset$.  As in Substep~{1.1},
the only nontrivial case
is the case in which
${\rm card}\, I >0$.

Assume that ${\rm card}\, I =1,$ 
and let  $i\in\{1,...,M\}$
be such that $\calH^1 (A\cap H_i)>0$. Define
the sets
{\setlength\arraycolsep{0.5pt}
\begin{eqnarray*}
&& A_1:= A\cap \overline
E^c,\qquad A_2:= A\cap
E,\\
&&A_3:=\{x\in
A\cap \overline
E^c\!:\, x\cdot \nu_i>\sigma_i\}
\cup\{x\in
A\cap E\!:\, x\cdot \nu_i<\sigma_i\}\cup
(A\cap H_i). 
\end{eqnarray*}}%
We have that $A_1$, $A_2$,
and $A_3$ are open and satisfy $A=A_1\cup
A_2\cup A_3$, $w\equiv
b$ in $A_1$, $w\equiv
a$ in $A_2$, and
\begin{eqnarray*}
\begin{aligned}
w(x)=
\begin{cases}
b & \hbox{if } x\cdot\nu_i>\sigma_i
\hbox{ and }x\in A_3,\\
a &\hbox{if } x\cdot\nu_i<\sigma_i
\hbox{ and } x\in A_3.
\end{cases}
\end{aligned}
\end{eqnarray*}
Since $f(\cdot,\cdot,0,0)=0$,
we obtain $\calF(w;A_1)=
\calF(w;A_2)=0$, which,
together with Lemma~\ref{CalFmeasure} and Substep~{1.2},
yields
\begin{eqnarray*}
\begin{aligned}
&\calF(w;A)\lii \calF(w,A_3)
\lii K(a,b,\nu)
\calH^1(A_3
\cap S_w)= K(a,b,\nu)
\calH^1(A
\cap S_w).
\end{aligned}
\end{eqnarray*}
By induction, we assume
that the statement holds
true if ${\rm card}\,
I=k$ for some $k\in\{1,...,M-1\}$
and we prove that 
it is also true if  ${\rm card}\,
I=k+1$. Assume that
\begin{eqnarray*}
\begin{aligned}
& A\cap \partial E =
(A\cap H_1)\cup\cdots
\cup (A\cap H_{k+1}),
\end{aligned}
\end{eqnarray*}
and define 
\begin{eqnarray*}
\begin{aligned}
& A_1:= \{x\in A\!: \,
\dist(x, H_1) < \dist(x, H_2\cup\cdots \cup H_M)\},\quad
A_2:=A\backslash \overline
A_1.
\end{aligned}
\end{eqnarray*}
We have that $A_1$ and
$A_2$ are open sets such
that $A_1\cap H_1\not=\emptyset$,
$A_1\cap \big(H_2\cup\cdots\cup
H_{k+1}\big)=\emptyset$,
$A_2 \cap H_1=\emptyset$,
 and $A_2\cap \big(H_2\cup\cdots\cup
H_{k+1}\big)\not=\emptyset$.
Moreover, we observe
that $\partial
A_1\cap\partial A_2\subset
S:=\{x\in\RR^2\!:\,\dist(x, H_1) = \dist(x,
H_2\cup\cdots \cup H_M)\}$
and $\calH^1(S\cap S_w)=0$
since $\calH^1(H_i\cap H_j)=0$ for $i\not= j$.
Fix $\delta>0$. By the
induction hypothesis
applied to $A_1$ and
$A_2$, there exist sequences
$\{w_n^1\}_{n\in\NN}\subset
W^{1,1}(A_1;[\alpha,\beta]\times
S^2)$ and $\{w_n^2\}_{n\in\NN}\subset
W^{1,1}(A_2;[\alpha,\beta]\times
S^2)$ such that
\begin{eqnarray}\label{ubjump5}
\begin{aligned}
&\lim_{n\to\infty}\Vert
w_n^1-w\Vert_{L^1(A_1;\RR\times\RR^3)}
=0,\quad \lim_{n\to\infty}\Vert
w_n^2-w\Vert_{L^1(A_2;\RR\times\RR^3)}
=0,\\
&\lim_{n\to\infty}\int_{A_1} f(w_n^1(x),\grad w_n^1(x))\,\dx
\lii \int_{A_1\cap S_w}
 K(a,b,\nu_w(x))\,\d\calH^1(x)
+\delta,\\
&\lim_{n\to\infty}\int_{A_2} f(w_n^2(x),\grad w_n^2(x))\,\dx
\lii \int_{A_2\cap S_w}
 K(a,b,\nu_w(x))\,\d\calH^1(x)
+\delta.
\end{aligned}
\end{eqnarray}
Let $A'_1,\,A'_2\in
\calA_\infty(\Omega)$
satisfy $A'_1\subset\subset
A_1$, $A'_2\subset\subset
A_2$, and
\begin{eqnarray}\label{ubjump5'}
\begin{aligned}
& |Dw|(A_1\backslash\overline
{A_1'})\lii \delta, \qquad
|Dw|(A_2\backslash\overline
 {A_2'})\lii \delta.
\end{aligned}
\end{eqnarray}
By Lemma~\ref{slicingMan},
there exist 
sequences
$\{\tilde w_n^1\}_{n\in\NN}\subset
W^{1,1}(A_1';[\alpha,\beta]\times
S^2)$ and $\{\tilde w_n^2\}_{n\in\NN}\subset
W^{1,1}(A_2';[\alpha,\beta]\times
S^2)$ satisfying
\begin{eqnarray}\label{ubjump6}
\begin{aligned}
&\lim_{n\to\infty}\Vert\tilde
 w_n^1-w\Vert_{L^1(A_1';\RR\times\RR^3)}
=0,\qquad \lim_{n\to\infty}\Vert\tilde
 w_n^2-w\Vert_{L^1(A_2';\RR\times\RR^3)}
=0,\\
&\tilde w_n^1= w \hbox{
on $\partial A_1'$},\qquad
\tilde w_n^2= w \hbox{
on $\partial A_2'$},
\\
&\limsup_{n\to\infty}\int_{A_1'}
f(\tilde w_n^1(x),\grad\tilde
 w_n^1(x))\,\dx
\lii \liminf_{n\to\infty}\int_{A_1'}
f(w_n^1(x),\grad w_n^1(x))\,\dx,\\
&\limsup_{n\to\infty}\int_{A_2'}
f(\tilde w_n^2(x),\grad\tilde
 w_n^2(x))\,\dx
\lii \liminf_{n\to\infty}\int_{A_2'}
f(w_n^2(x),\grad w_n^2(x))\,\dx.
\end{aligned}
\end{eqnarray}
Moreover, by
Lemma~\ref{man+tr} and Remark~\ref{man+trhalflip},
together with the fact
that $\dist(\overline{A_1'},
\overline{A_2'})>0$,
there exist a positive
constant $\widetilde
C$ and a sequence
$\{\tilde w_n^3\}_{n\in\NN}\subset
W^{1,1}(A\backslash(\overline
{A_1'}\cup \overline
{A_2'});[\alpha,\beta]\times
S^2)$ such that
\begin{eqnarray}\label{ubjump7}
\begin{aligned}
&\lim_{n\to\infty}\Vert\tilde
 w_n^3-w\Vert_{L^1(A\backslash(\overline
{A_1'}\cup \overline
{A_2'});\RR\times\RR^3)}
=0,\qquad \tilde w_n^3= w \hbox{
on $\partial A_1'\cup\partial A_2'$},\\
&\limsup_{n\to\infty}\int_{A\backslash(\overline
{A_1'}\cup \overline
{A_2'})} |\grad \tilde
w^3_n(x)|\, \dx \lii
\widetilde C |Dw|\big(A\backslash(\overline
{A_1'}\cup \overline
{A_2'}) \big).
\end{aligned}
\end{eqnarray}
Define for $n\in\NN$,
\begin{eqnarray*}
\begin{aligned}
& w_n:= 
\begin{cases}
\tilde w_n^1 & \hbox{in
}
A_1',\\
\tilde w_n^2 & \hbox{in
}
A_2',\\
\tilde w_n^3 & \hbox{in
}
A\backslash(\overline
{A_1'}\cup \overline
{A_2'}).
\end{cases}
\end{aligned}
\end{eqnarray*}
In view of \eqref{ubjump6}
and \eqref{ubjump7},
we have that $\{w_n\}_{n\in\NN}\subset
W^{1,1}(A;[\alpha,\beta]\times
S^2)\) and  \( \lim_{n\to\infty}\Vert
 w_n-w\Vert_{L^1(A;\RR\times\RR^3)}
=0$. Consequently, by
definition of $\calF(w;A)$, and by
\eqref{ubjump5}, \eqref{boundsf},
 and \eqref{ubjump5'},
we obtain
{\setlength\arraycolsep{0.5pt}
\begin{eqnarray*}
\calF(w;A) &&\lii \liminf_{n\to\infty}
 \bigg(\int_{A_1'} f(\tilde
w_n^1(x),\grad\tilde
 w_n^1(x))\,\dx + \int_{A_2'}
f(\tilde w_n^2(x),\grad\tilde
 w_n^2(x))\,\dx\\
&&\hskip30mm + \int_{A\backslash(\overline
{A_1'}\cup \overline
{A_2'})} f(\tilde w_n^3(x),\grad\tilde
 w_n^3(x))\,\dx \bigg)\\
&&\lii  \int_{A_1\cap
S_w}
 K(a,b,\nu_w(x))\,\d\calH^1(x)
+  \int_{A_2\cap S_w}
 K(a,b,\nu_w(x))\,\d\calH^1(x)+
2\delta\\
&&\hskip30mm+ (3+\beta)\limsup_{n\to\infty}
\int_{A\backslash(\overline
{A_1'}\cup \overline
{A_2'})} |\grad \tilde
w_n^3(x)|\,\dx \\
&&\lii  \int_{A\cap
S_w}
K(a,b,\nu_w(x))\,\d\calH^1(x)
+ 2\delta[1+\widetilde
C(3+\beta)],
\end{eqnarray*}}%
and we deduce
Substep~{1.3} by letting
$\delta\to0^+$.

{\sl Substep 1.5.} We
conclude Step~1.

Let $\{E_n\}_{n\in\NN}$
be a sequence of polyhedral
open sets such that (see
Remark~\ref{AppByPoly}) 
\begin{eqnarray*}
\begin{aligned}
&\lim_{n\to\infty}\calL^2(E_n\Delta E)=0,\qquad \lim_{n\to\infty}
{\rm Per}_\Omega(E_n)
= {\rm Per}_\Omega(E),
\end{aligned}
\end{eqnarray*}
and define
\begin{eqnarray*}
\begin{aligned}
& w_n(x):= a\chi_{E_n}(x) +b\chi_{E_n^c}(x).
\end{aligned}
\end{eqnarray*}
Then
\begin{eqnarray*}
\begin{aligned}
& \lim_{n\to\infty} \Vert w_n - w\Vert_{L^1(\Omega:\RR\times
\RR^3)}= 0, \qquad \lim_{n\to\infty}
|Dw_n|(\Omega) = |D w|(\Omega).
\end{aligned}
\end{eqnarray*}
We now consider the homogeneous of degree-one extension
 $\widetilde K(a,b,\cdot)$
 of $K(a,b,\cdot)$
to the whole $\RR^2$
defined for $\upsilon\in
\RR^2$ by
\begin{eqnarray*}
\begin{aligned}
& \widetilde K(a,b,\upsilon):=
\begin{cases}
0 & \hbox{if } \upsilon=0,\\
|\upsilon|  K(a,b,\upsilon/|\upsilon|)
& \hbox{if } \upsilon\not=0.
\end{cases}
\end{aligned}
\end{eqnarray*}
In view of Lemma~\ref{pptyrecfct},
$\widetilde K(a,b,\cdot)$
is an upper semicontinuous
function in $\RR^2$ satisfying
$\widetilde K(a,b,\upsilon)\lii
C|\upsilon|$
for all $\upsilon\in\RR^2$
and for some positive
constant $C$.
Therefore, we can find
a decreasing sequence $\{h_m\}_{m\in\NN}$ of
continuous functions satisfying
for all $\upsilon\in \RR^2$,
\begin{eqnarray*}
\begin{aligned}
& \widetilde K(a,b,\upsilon)
\lii h_m(\upsilon)\lii
2C |\upsilon|,
\qquad \widetilde K (a,b,\upsilon)
= \inf_{m\in\NN} h_m(\upsilon).
\end{aligned}
\end{eqnarray*}
Using the lower semicontinuity
of $\calF(\cdot;A)$ with
respect to the $L^1$-convergence
(of sequences taking
values on $[\alpha,\beta]\times
S^2$), Substep~{1.3,} and
Reshetnyak's
Continuity Theorem, for every \(m\in\NN, \) we
obtain 
{\setlength\arraycolsep{0.5pt}
\begin{eqnarray*}
\calF(w;A)&&\lii \liminf_{n\to\infty}
\calF(w_n;A) \lii \liminf_{n\to\infty}
\int_{S_{w_n}\cap
A} \widetilde K(a,b,\nu(x))\,\d\calH^1(x)
\\
&&\lii \liminf_{n\to\infty}
\int_{S_{w_n}\cap
A} h_m(\nu(x))\,\d\calH^1(x)
= \int_{S_{w}\cap
A} h_m(\nu(x))\,\d\calH^1(x).
\end{eqnarray*}}%
We conclude
Step~1 by letting $m\to\infty$
and using
Lebesgue's Monotone Convergence
Theorem. 

\underbar{Step 2.} We
prove that \eqref{ubjump}
holds whenever $w$ is
of the form 
\begin{eqnarray}\label{sumfiniteper}
\begin{aligned}
w(x)= \sum_{i=1}^k a_i\chi_{E_i}(x),
\end{aligned}
\end{eqnarray}
where $k\in\NN$, $a_i\in [\alpha,\beta]\times
S^2$, $i\in\{1,...,k\}$, and $\{E_i\}_{i=1}^k$
is a family of mutually disjoint
sets of finite perimeter in $\Omega$, which covers
$\Omega$. 

By Theorem~\ref{DG&FE}
(see also \eqref{equivJump}),
we have that
for all $i,j\in\{1,...,k\}$,
{\setlength\arraycolsep{0.5pt}
\begin{eqnarray*}
&& (w^+(x), w^-(x), \nu_w(x))
\sim (a_i, a_j, \nu_{E_i}(x))
\hbox{ for all } x\in \calF^*E_i
\cap \calF^*E_j,\\
&& \bigcup_{i<j}
(\calF^*E_i
\cap \calF^*E_j) \subset
S_w \subset B \cup \bigcup_{i<j}
(\calF^*E_i
\cap \calF^*E_j)
\end{eqnarray*}}%
where $B$ is a suitable
Borel set satisfying
$\calH^1(B)=0$ and  $(\calF^*E_i
\cap \calF^*E_j) \cap
(\calF^*E_l
\cap \calF^*E_m) = \emptyset$
for all $i,j, l, m\in\{1,...,k\}$
such that $i\not= j$,
$l\not= m$, and  $\{i,j\} \not=
\{l,m\}$. Moreover,
\begin{eqnarray*}
\begin{aligned}
& Dw = \sum_{i=1}^k a_i
\otimes \nu_{E_i} {\calH^1}_{\lfloor
\calF^* E_i} = \sum_{i<j} (a_i - a_j)
\otimes \nu_{E_i} {\calH^1}_{\lfloor
(\calF^* E_i \cap \calF^* E_j)}.
\end{aligned}
\end{eqnarray*}
Therefore, having
in mind \eqref{ubcalF}
and the identification observed at the beginning
of Subsection~\ref{upperbound},
we conclude that $\calF(w;\cdot)\ll |Dw|
\ll {\calH^1}_{\lfloor
S_w}$ and
{\setlength\arraycolsep{0.5pt}
\begin{eqnarray*}
\calF(w;A) &&= \calF(w;A\cap
S_w) = \sum_{i<j}  \calF(w;A\cap
(\calF^* E_i \cap \calF^*
E_j))\\
&&= \sum_{i<j}  \calF(a_i\chi_{E_i}
+ a_j\chi_{E_i^c} ;A\cap
(\calF^* E_i \cap \calF^*
E_j)).
\end{eqnarray*}}%
On the other hand, by Step~1 together with
Theorem~\ref{DG&FE}, we obtain for $i,j\in\{1,...,k\}$
with $i<j$,
{\setlength\arraycolsep{0.5pt}
\begin{eqnarray*}
&& \calF(a_i\chi_{E_i}
+ a_j\chi_{E_i^c} ;A\cap
(\calF^* E_i \cap \calF^*
E_j))\\
&&\quad = \inf \big\{
\calF(a_i\chi_{E_i}
+ a_j\chi_{E_i^c} ; A')\!:\,
A'\in\calA(\Omega), \,
A' \supset A\cap
(\calF^* E_i \cap \calF^*
E_j) \big\}\\
&&\quad\lii \inf \bigg\{
\int_{A'\cap \calF^*E_i}
 K(a_i,a_j,\nu_{E_i}(x))
\,\d\calH^1(x)\!:\,
A'\in\calA(\Omega), \,
A' \supset A\cap
(\calF^* E_i \cap \calF^*
E_j) \bigg\}\\
&&\quad = \int_{A\cap
(\calF^* E_i \cap \calF^*
E_j)}
 K(a_i,a_j,\nu_{E_i}(x))
\,\d\calH^1(x).
\end{eqnarray*}}%
Consequently,
\begin{eqnarray*}
\begin{aligned}
 &\calF(w;A) &\lii  \sum_{i<j}  \int_{A\cap
(\calF^* E_i \cap \calF^*
E_j)}
 K(a_i,a_j,\nu_{E_i}(x))
\,\d\calH^1(x)= \int_{A\cap
S_w}
 K(w^+(x),w^-(x),\nu_{w}(x))
\,\d\calH^1(x),
\end{aligned}
\end{eqnarray*}
which concludes Step~2.

\underbar{Step 3.} We
establish Lemma~\ref{jumppartCalF}.

Let \(\phi\in C^\infty_c(\RR^3;[0,1])\) be a smooth
cut-off function such that \(\phi(z)=0\) if \(|z|
\lii \frac{1}{4}\), and  \(\phi(z)=1\) if \(|z|
\gii \frac{3}{4}\). 
Let $\bar\phi: \RR\times
(\RR^3 \backslash \{0\})
\to [\alpha,\beta]\times
S^2$ be the function
defined by $\bar\phi(r,s):=
(\tilde r, \tilde s)$,
where $\tilde r$ and
$\tilde s$ are given
by \eqref{deftilders}. Note that for all $\delta>0$,
 $\bar\phi$ is
a Lipschitz function
in $\RR\times (\RR^3
\backslash B(0,\delta))$.
Set \(\delta=\frac{1}{8}\), and let \(L_{\bar
\phi}:=Lip(\bar\phi_{|\RR\times (\RR^3
\backslash B(0,\frac{1}{8}))})\).
Consider the 
extension $\overline K: (\RR\times \RR^3 ) \times
(\RR\times \RR^3
) \times
 S^1 \to [0,+\infty)$
 of $K$ defined for $a=(r_1,s_1)$,
 $b=(r_2,s_2) \in \RR\times \RR^3
$ and
$\nu\in S^1$, by 
\[\overline K(a,b,\nu):=
\begin{cases}
0 & \hbox{if } s_1=0 \hbox{ or } s_2=0,\\
\phi(s_1) \phi(s_2) K(\bar\phi(a), \bar\phi(b), \nu)
& \hbox{if } s_1\not=0 \hbox{ and } s_2\not=0.
\end{cases}\]

Then, the properties
stated in Lemma~\ref{pptyrecfct}
hold  in $(\RR\times \RR^3 ) \times
(\RR\times \RR^3
) \times
 S^1$ for $\overline
K$, where the corresponding constant depends on
the constant in  Lemma~\ref{pptyrecfct}, on \(L_{\bar\phi}\),
and on \(\Vert\phi\Vert_{1,\infty}\).   Because $w$ takes
values on $[\alpha,\beta]\times S^2$,  arguing as in
\cite[Step 2 of Prop.~{4.8}]{AMTXCI}
we can construct a sequence
$\{w_n\}_{n\in\NN} \subset
BV(\Omega;\RR\times \RR^3)$
where each $w_n$
is of the type \eqref{sumfiniteper} (but whose
coefficients do not necessarily belong to \([\alpha,\beta]
\times S^2)\) and such that 
{\setlength\arraycolsep{0.5pt}
\begin{eqnarray}
&&\lim_{n\to\infty} \Vert
w_n - w\Vert_{L^\infty(\Omega;\RR
\times \RR^3)}=0, \label{AMTconvLi} \\
&&\liminf_{n\to\infty}
\int_{A \cap S_{w_n}}
\overline K (w_n^+(x),
w_n^-(x), \nu_{w_n}(x))\,
\d\calH^1(x)\\
&&\qquad \lii C |Dw|(A\backslash
S_w) +  \int_{A \cap
S_{w}}
\overline K (w^+(x),
w^-(x), \nu_{w}(x))\,
\d\calH^1(x) \nonumber
\\
&&\qquad= C |Dw|(A\backslash
S_w) +  \int_{A \cap
S_{w}}
 K (w^+(x),
w^-(x), \nu_{w}(x))\,
\d\calH^1(x),\label{AMTliminf}
\end{eqnarray}}%
where $C$ is a positive
constant  depending only
on the constants in Lemma~\ref{pptyrecfct}
for $\overline K$, and where in the last equality
we used Lemma~\ref{onManae} and Theorem~\ref{appdiffisac}.

In view of \eqref{AMTconvLi} and since $w$ takes values
on $[\alpha,\beta]\times S^2$, \(w_n\)
takes values in \(\RR\times (\RR^3
\backslash B(0,3/4))\) for all $n\in\NN$ sufficiently
large. Then, also \(w_n^{\pm}(x) \in\RR\times (\RR^3
\backslash B(0,3/4)) \) for \(\calH^1\)-\aev\ \(x\in
S_{w_n}\) and for all $n\in\NN$ sufficiently
large. For all such \(n\in\NN\),
 the function \[\bar w_n:= \bar\phi(w_n)\]
 belongs to $BV(\Omega;\RR\times
\RR^3)$, takes values on $[\alpha,\beta]\times S^2$,
and  is of the type \eqref{sumfiniteper}.
Moreover, by the Lipschitz continuity of \(\bar\phi\),
the equality \(\bar\phi(w)=w\), and \eqref{AMTconvLi},
we also have  $\lim_{n\to\infty} \Vert \bar
w_n - w\Vert_{L^1(\Omega;\RR
\times \RR^3)}=0$. Furthermore,
using Proposition~\ref{appppties}~{\it
(a)-iii), (b)-iii)} and
Theorem~\ref{appdiffisac}~{\it
(b)}, we have \(S_{\bar w_n} \subset S_{w_n}\), \(\calH^1(S_{w_n}
\backslash (J_{w_n} \cap J_{\bar w_n}))=0\), and
\((\bar w_n^+(x),
\bar w_n^-(x), \nu_{\bar
w_n}(x)) = (\bar\phi(w_n^+(x)),
\bar\phi(w_n^-(x), \nu_{
w_n}(x)) \) for all \(x\in J_{w_n} \cap J_{\bar w_n}
\). Thus,
{\setlength\arraycolsep{0.5pt}
\begin{eqnarray}\label{KbarwbarKw}
\int_{A \cap S_{\tilde
w_n}}
 K (\bar w_n^+(x),
\bar w_n^-(x), \nu_{\bar
w_n}(x))\,
\d\calH^1(x) &&\lii  \int_{A
\cap S_{
w_n}}
\overline K (w_n^+(x),
w_n^-(x), \nu_{
w_n}(x))\,
\d\calH^1(x).
\end{eqnarray}}%
Hence, using the
lower semicontinuity
of $\calF(\cdot, A)$
with respect to the $L^1$-convergence
(of sequences taking
values in $[\alpha,\beta]\times
S^2$), Step~2, \eqref{AMTliminf},
and \eqref{KbarwbarKw},
yields
{\setlength\arraycolsep{0.5pt}
\begin{eqnarray*}
\calF(w,A) &&\lii \liminf_{n\to\infty}
\calF(\bar w_n,A) \lii
\liminf_{n\to\infty}
\int_{A \cap S_{\bar
w_n}}
 K (\bar w_n^+(x),
\bar w_n^-(x), \nu_{\bar
w_n}(x))\,
\d\calH^1(x) \\
&&\lii C |Dw|(A\backslash
S_w) +  \int_{A \cap
S_{w}}
 K (w^+(x),
w^-(x), \nu_{w}(x))\,
\d\calH^1(x).
\end{eqnarray*}}%
Finally, 
{\setlength\arraycolsep{0.5pt}
\begin{eqnarray*}
\calF(w,A\cap S_w)
&&= \inf \big\{\calF(w
; A')\!:\,
A'\in\calA(\Omega), \,
A' \supset A\cap S_w
\big \}\\
&&\lii \inf\bigg\{ C |Dw|(A'\backslash
S_w) +  \int_{A' \cap
S_{w}}
 K (w^+(x),
w^-(x), \nu_{w}(x))\,
\d\calH^1(x)\!:\,
A'\in\calA(\Omega), \,
A' \supset A\cap S_w
\big \}\\
&&  = \int_{A \cap
S_{w}}
 K (w^+(x),
w^-(x), \nu_{w}(x))\,
\d\calH^1(x),
\end{eqnarray*}}%
which concludes the proof
of Lemma~\ref{jumppartCalF}.\end{proof}

\section*{Acknowledgements}

The authors acknowledge the funding of Fundação para a Ciência e a Tecnologia (Portuguese Foundation for Science
and Technology) through the ICTI CMU-Portugal Program in Applied Mathematics and UTACMU/
MAT/0005/2009. The authors also thank the Center for Nonlinear Analysis (NSF Grant DMS-0635983), where part of this research was carried out. 

R. Ferreira was partially supported by  the KAUST
SRI, Center for
Uncertainty Quantification  in Computational
Science and Engineering and by the Funda\c{c}\~ao
para a Ci\^encia e a Tecnologia 
through the grant
SFRH/BPD/81442/2011.
The work of I. Fonseca was partially supported
by the National Science Foundation under Grant
No. DMS-1411646.
The work of L.M. Mascarenhas was partially supported
by  UID/MAT/00297/2013.

\bibliographystyle{plain}

\bibliography{imaging}

\end{document}